\documentclass{article}

\usepackage{amsmath,amsthm,amsfonts,amssymb}
\usepackage{graphicx}
\usepackage[normalem]{ulem}
\usepackage{listings}
\usepackage{booktabs}
\usepackage{longtable}
\usepackage[T1]{fontenc}
\usepackage[scaled=.85]{beramono}
\usepackage{lmodern}
\allowdisplaybreaks

\lstdefinelanguage{Magma}{
  morekeywords={Group,sub,Simplify,AbelianQuotientInvariants,SearchForIsomorphism,
  Homomorphism,GrfFP,GrpFP,Id},
  sensitive=true,
  morecomment=[l]{//},
  morestring=[b]",
}
\newtheorem{theorem}{Theorem}[section]
\newtheorem{lm}[theorem]{Lemma}
\newtheorem{cy}[theorem]{Corollary}
\newtheorem{conjecture}[theorem]{Conjecture}
\newtheorem{prop}[theorem]{Proposition}

\newtheorem{df}[theorem]{Definition} 
\newtheorem{remark}[theorem]{Remark}
\newtheorem{example}[theorem]{Example}

\usepackage{color}
\usepackage[dvipsnames]{xcolor}

\newcommand{\beq}{\begin{equation}}
\newcommand{\eeq}{\end{equation}}
\newcommand{\be}{\begin{enumerate}}
\newcommand{\ee}{\end{enumerate}}
\newcommand{\bp}{\begin{proof}}
\newcommand{\ep}{\end{proof}}
\newcommand{\bi}{\begin{itemize}}
\newcommand{\ei}{\end{itemize}}
\newcommand{\bea}{\begin{eqnarray*}}
\newcommand{\eea}{\end{eqnarray*}}
\newcommand{\bml}{\begin{multline*}}
\newcommand{\eml}{\end{multline*}}
\newcommand{\Z}{\mathbb{Z}}
\newcommand{\gp}[2]{\langle\, #1 \mid #2 \,\rangle}

\usepackage[margin=1in]{geometry}
\newcommand{\conj}[2]{#1^{-1}#2\,#1}  
\newcommand{\inv}{^{-1}}

\newcommand{\ncl}{\operatorname{ncl}}

\DeclareMathOperator{\Aut}{Aut}

\edef\storedcatcodeat{\the\catcode`\@} \catcode`\@=11

\catcode`\@=\storedcatcodeat\relax

\begin{document}

\title{Equations in Products of Free Groups and 3-Manifold Groups, I}

\author{Olga Kharlampovich, Alina Vdovina}

\maketitle

 Abstract:

Perelman's proof of the Poincare conjecture shows that every simply
connected closed 3-manifold is homeomorphic to the 3-sphere. The fundamental groups of 3-manifolds attract lots of interest from mathematicians of different fields. As it was stated in a famous survey of Allen Hatcher  "The classification of 3-manifolds", one would want to know exactly which groups occur as fundamental groups of these manifolds. 

The Stallings--Jaco--Hempel reformulation of the Poincare conjecture
inspired several connections between low-dimensional topology, equations
over free groups, and combinatorial group theory. The reformulation reduces the problem to study epimorphisms from the fundamental group of a closed orientable surface onto the direct product of two free groups (they correspond to Heegaard splittings of 3-manifolds and were named splitting homomorphisms).  Olshankii (1989) constructed (in non-explicit form) first non-trivial examples of such splitting epimorphisms  and verified 
 the standardness of some of them.

 We construct up to equivalence all the splitting coordinate-surjective homomorphisms (among them, the genuine splitting epimorphisms are exactly those for which  our constructed associated group balanced presentation is trivial).  We give generators and relations of the corresponding  balanced presentation (so all closed orientable 3-manifold groups) that can be studied by algebraic
methods.  We analyse a big class of such homomorphisms/presentations (including all Olshanskii's epimorphisms)
and show that splitting epimorphisms are  rare, in this case the corresponding balanced presentation of the trivial group can be reduced to the standard one  by Andrews-Curtis transformations and  the epimorphisms  are standard.

\tableofcontents

\section{Introduction} 
Perelman's theorem that resolved the Poincare conjecture, states: Every simply connected, closed 3-manifold is homeomorphic to the 3-sphere. The fundamental groups of 3-manifolds attract lots of interest from mathematicians of different fields. As it was stated in  Allen Hatcher's survey  \cite{Hat}  one would want to know exactly which groups occur as fundamental groups of these manifolds.

Stallings-Jaco-Hempel's reformulation (1966-1976) of Poincare conjecture (Perelman theorem) inspired 
many interesting mathematical connections. The reformulation reduces the problem to study epimorphisms from the fundamental group of a closed orientable surface onto the direct product of two free groups (they correspond to Heegaard splittings of 3-manifolds and were named splitting homomorphisms). Grigorchuk-Kurchanov wrote a survey on this (1990). Olshankii (1989) constructed (in non-explicit form) first non-trivial examples of such splitting epimorphisms  and verified 
 the standardness of some of them.

This paper is the first in the series that, in particular,  aims to show that this approach to the problem was quite viable. We construct up to equivalence all the splitting coordinate-surjective homomorphisms (among them, the genuine splitting epimorphisms are exactly those for which
the associated balanced presentation is trivial).  We give a  concrete simple relators of the corresponding  balanced presentation, see, as one of the  simplest examples,  relations (\ref{eq:6}) (so all closed orientable 3-manifold groups) that can be studied by algebraic
methods. Our construction makes the associated presentations explicit enough to compute with and analyse. The construction is inspired by Kharlampovich-Miasnikov's proof of the implicit function theorem. We also analyse a big class of such homomorphisms/presentations (including all Olshanskii's epimorphisms)
 and show that in this class splitting epimorphisms are  rare and in this case they are standard and the corresponding balanced presentation of the trivial group can be reduced to the standard one  by Andrews-Curtis transformations. 

There are two  main motivations of this paper. One  is the  possibility of obtaining a purely group theoretic proof of the Poincare conjecture. Another is to study fundamental groups of closed orientable 3-manifolds in Hatcher's spirit  that become clear from their presentations that we construct. Main results of the paper  are Theorem \ref{th:14} proved without using the Perelman theorem (Poincar\'{e} conjecture) and, therefore, providing a different proof of this conjecture for the class of  3-manifolds under consideration, and 
Theorem \ref{thm:genus3-two-layers}.

We used some technical help with computations, typing and  figures  from Chat GPT 5.5-5.6. We also extensively used MAGMA software.

\subsection{Surface groups and Stallings-Jaco-Hempel approach} Let $S_g$ be the closed orientable surface of genus $g$ and $F_g$ be a free group of rank $g$. We have $$\pi _1(S_g)=\langle x_1,y_1,\ldots x_g,y_g| [x_1,y_1]\ldots [x_g,y_g]=1\rangle .$$  Let $F(\bar a)$ be the free group with basis $a_1,\ldots ,a_g$ and $F(\bar b)$ be a free group with basis $b_1,\ldots ,b_g$. Two homomorphisms $\phi ,\psi$ from $\pi _1(S_g)$ to a group $H$ are equivalent if
there exists an automorphism $\sigma$ of $\pi _1(S_g)$  and an automorphism $\tau$ of $H$
such that $\phi =\tau\psi\sigma.$ In his 1965 paper, John Stallings provided an algebraic criterion for the 3D Poincar\'{e} Conjecture.
\begin{theorem} \cite{J},\cite{S} The Poincare conjecture holds if and only if for each $g\geq 2$ every  epimorphism $\phi: \pi _1(S_g)\rightarrow F_g\times F_g$ non-trivially factors through a free product.
\end{theorem} 

We say that an epimorphism $\phi: \pi _1(S_g)\rightarrow F_g\times F_g$ is {\em standard} is it is equivalent to the one sending each $x_i$ to $b_i$ and $y_i$ to $a_i$. 
\begin{theorem} (Hempel \cite{H}) The Poincare conjecture holds if and only if for each $g\geq 2$ there exists only one equivalence class of epimorphisms $\phi: \pi _1(S_g)\rightarrow F_g\times F_g$ (equivalently, every such epimorphism is standard).
\end{theorem}
See also the discussion in  \cite{GK2}.

\subsection{Associated balanced presentations}

Our convention will be to multiply mappings right to left.
 
Let $\phi =(\phi_1, \phi _2): \pi _1(S_g)\rightarrow F(\bar a)\times F(\bar b).$

By  \cite[Theorem 1]{GK1}  there is only one equivalence class  of epimorphisms $\pi _1(S_g)\rightarrow F(\bar b)$.  Every such epimorphism is equivalent to 
$\psi: \pi _1(S_g)\rightarrow F(\bar b)$ such that
$\psi (x_i)=b_i, \psi (y_i)=1.$

Then for some $\delta\in Mod(S_g)$, $$\phi\delta (x_i)=b_iu_i(\bar a),\ \phi\delta (y_i)=v_i(\bar a)$$

We have 
 $[b_1u_1(\bar a),v_1(\bar a)]\ldots [b_gu_g(\bar a),v_g(\bar a)]=1.$ 
 Then
\begin{equation}\label{1} [u_1(\bar a),v_1(\bar a)]\ldots [u_g(\bar a),v_g(\bar a)]=1 \end{equation}
$u_i=\gamma (x_i), v_i=\gamma (y_i)$
where $\gamma: \pi _1(S_g)\rightarrow F(\bar a)$ onto. Then $\gamma=\gamma_0\sigma $, $\sigma \in Mod(S_g)$,
 $\gamma _0(x_i)=a_i,   \gamma _0(y_i)=1.$
 
 We know the generators of $Mod(S_g)=Out ^{+}(\pi _1(S_g))$, they are Dehn twists along a certain family of simple closed curves. So we can describe all possible values of 
$u_1(\bar a),v_1(\bar a),\ldots ,u_g(\bar a),v_g(\bar a)$ . This will be done in Section \ref{Autom}.
\begin{theorem}
The homomorphism
\[
\phi:\pi_1(S_g)\to F(a_1,\dots,a_g)\times F(b_1,\dots,b_g)
\]
is surjective if and only if
\[
G_{\phi}=\langle a_1,\dots,a_g\mid v_1,\dots,v_g\rangle =1.
\]
\end{theorem}

\begin{proof}
Suppose first that
\[
\langle a_1,\dots,a_g\mid v_1,\dots,v_g\rangle =1.
\]
Then the normal closure of the elements \(v_i(a)\) is the whole free group
\[
F(a_1,\dots,a_g).
\]
Hence all generators \(a_i\) belong to the image of \(\phi\), since
\[
\phi(y_i)=v_i(a).
\]
Therefore all words \(u_i(a)\) also belong to the image. Since
\[
\phi(x_i)=b_i u_i(a),
\]
it follows that
\[
b_i=\phi(x_i)u_i(a)^{-1}
\]
belongs to the image of \(\phi\). Thus the image contains all generators
\[
a_1,\dots,a_g,b_1,\dots,b_g,
\]
and consequently \(\phi\) is surjective.

Conversely, if \(\phi\) is surjective, then the projection of its image onto the first factor equals
\[
F(a_1,\dots,a_g).
\]
But the projection is generated by the conjugates of elements \(v_i(a)\), so their normal closure is the whole free group. Hence
\[
\langle a_1,\dots,a_g\mid v_1,\dots,v_g\rangle =1.
\]
\end{proof}

Throughout the paper,
the symbols \(u_i,v_i\) always denote words in the generators
\(a_1,\dots,a_g\).We use the same symbols \(a_1,\dots,a_g\) both for the generators of the free
group
\[
F(a_1,\dots,a_g)
\]
and for their images in quotient groups such as
\[
G=G_{\phi}=\langle a_1,\dots,a_g\mid v_1,\dots,v_g\rangle.
\]
No confusion should arise from this convention.

\subsection{Heegaard Splitting}
\begin{theorem}  $G_{\phi}$  is the fundamental group of the 3-manifold with the Heegaard splitting of genus $g$ with gluing map $\sigma$.
\end{theorem}
\begin{proof} Let
\[
   M=U\cup V,\qquad T=U\cap V\cong S_g
\]
be a genus \(g\) Heegaard splitting.
Then \(U,V\) are handlebodies, so
\[
   \pi_1(U)\cong F_g,\qquad \pi_1(V)\cong F_g.
\]
The inclusions induce onto maps
\[
\pi_1(T)\to \pi_1(U),\qquad \pi_1(T)\to \pi_1(V).
\]
Together they give the \emph{splitting homomorphism}
\[
   \phi:\pi_1(T)\longrightarrow \pi_1(U)\times \pi_1(V)\cong F_g\times F_g.
\]
For the corresponding Heegaard splitting \(M=U\cup V\), van Kampen theorem gives
\[
\pi_1(M)\cong
\frac{\pi_1(U)*\pi_1(V)}{\ncl(\phi_1(t)=\phi_2(t),\ t\in\pi_1(S_g))},
\]

\[
\phi(x_i)=(u_i(\bar a),b_i),
\qquad
\phi(y_i)=(v_i(\bar a),1).
\]
The two individual homomorphisms are defined as:

\begin{equation}
    \phi_1: \pi_1(S_g) \to \pi_1(V_1), \quad \phi_1(x_i) = u_i(a), \quad \phi_1(y_i) = v_i(a)
\end{equation}

\begin{equation}
    \phi_2: \pi_1(S_g) \to \pi_1(V_2), \quad \phi_2(x_i) = b_i, \quad \phi_2(y_i) = 1
\end{equation}

So we get relations
\[
  u_i(\bar a)=b_i,
  \qquad
  v_i(\bar a)=1.
\]
Eliminating the \(b_i\)'s,
\[
   \pi_1(M)\cong
   \left\langle a_1,\ldots,a_g\ \middle|\
        v_1,\ldots,v_g
   \right\rangle=G_\phi.
\]
 
\begin{cy} Groups $G_{\phi}$ for different coordinate-surjective homomorphisms $\phi$ form the class of fundamental groups of all closed orientable 3-manifolds.
\end{cy}

\end{proof}

\subsection{Andrews--Curtis transformations}

Let
\[
F_n=F(x_1,\dots,x_n)
\]
be the free group on the generators $x_1,\dots,x_n$, and let
\[
\mathbf r=(r_1,\dots,r_m)\in F_n^m
\]
be an ordered $m$-tuple of words.  An \emph{Andrews--Curtis transformation}
replaces $\mathbf r$ by another tuple obtained by one of the following elementary
moves:

\begin{enumerate}
    \item {Inversion of a relator:}
    \[
    r_i \longmapsto r_i^{-1}.
    \]

    \item {Multiplication of one relator by another:}
    for $i\neq j$,
    \[
    r_i \longmapsto r_i r_j
    \quad\text{or, equivalently,}\quad
    r_i \longmapsto r_j r_i.
    \]

    \item {Conjugation of a relator:}
    for any word $w\in F_n$,
    \[
    r_i \longmapsto w r_i w^{-1}.
    \]

    \item {Change of free basis:}
    for an automorphism $\alpha\in \operatorname{Aut}(F_n)$,
    \[
    (r_1,
    \dots,r_m)
    \longmapsto
    (\alpha(r_1),\dots,\alpha(r_m)).
    \]
\end{enumerate}

Two tuples are called \emph{Andrews--Curtis equivalent} if one can be obtained
from the other by a finite sequence of these transformations and their inverses.
Equivalently, the same terminology may be applied to presentations
\[
\langle x_1,\dots,x_n \mid r_1,
\dots,r_m\rangle .
\]

In the balanced case $m=n$, the Andrews--Curtis conjecture says that every
balanced presentation of the trivial group,
\[
\langle x_1,\dots,x_n \mid r_1,
\dots,r_n\rangle \cong 1,
\]
is Andrews--Curtis equivalent to the standard trivial presentation
\[
\langle x_1,\dots,x_n \mid x_1,
\dots,x_n\rangle .
\]

Some authors separate the last operation from the first three and call it a
Nielsen change of generators.  Others include it among the extended
Andrews--Curtis or extended Nielsen transformations.  The essential operations
on relators are inversion, multiplication by another relator, and conjugation.

\section {Automorphisms of surface groups}\label{Autom}
In this section we describe automorphisms
\[
\sigma\in \mathrm{Aut}(S_g) 
\] that we use to  describe all possible values of 
$u_1(\bar a),v_1(\bar a),\ldots ,u_g(\bar a),v_g(\bar a)$  and therefore all groups $G_{\phi}.$

 Since
orientation preserving automorphisms of the surface group correspond to
elements of the mapping class group, this reduces the problem to the study of
certain mapping classes obtained from compositions of Dehn twists and their action on the standard generators of $\pi_1(S_g)$. 

{\bf In this paper we will only consider the case  when $\sigma (x_i)=x_i, i=1,\ldots ,g,$ and, therefore $\phi (x_i)=b_ia_i$, $ \phi (y_i)=v_i(\bar a).$}

\subsection{Case  $\sigma (x_i)=x_i, i=1,\ldots ,g.$}

In this case we set $x_i=a_i$ and we can rewrite equation  (\ref {1}) in the form

 \begin{equation}\label{2}
 z_1^{-1}a_1z_1\ldots z_g^{-1}a_gz_g=a_g\ldots a_1,
 \end{equation}
 where $z_i=y_ia_i\ldots a_1.$ We always denote the conjugation $z^{-1}az=a^z.$

 The coordinate group $H$ of equation (\ref {2}) is the fundamental group of the graph of groups with two vertices, one vertex group is $F(\bar a)$ and the other  corresponds to  a sphere with $g+1$ holes joined by $g+1$  edges to  $F(\bar a)$. 
  \begin{prop} \label {l3} 1. Each $F(\bar a)$-automorphism $\alpha$ of $H$   is generated by  Dehn twists along non-boundary parallel simple closed curves on $S_{0,g+1}$ (the sphere with \(g+1\) boundary components) fixing $a_1,\ldots ,a_n$ and edge automorphisms  $T_{a_1},\ldots T_{a_g}$ \cite[Theorem 5.1]{BKM}.
 
 2. Every solution of  equation (\ref {2}) is obtained by pre-composing the solution obtained from $y_i=1, i=1,\ldots ,g$ by an  $F(\bar a)$-automorphism $\alpha$ of $H$.

\end{prop} 
\begin{proof}
1. Every nontrivial homotopy class of a simple closed curve on the sphere with boundaries
has a curve that divides the sphere into two parts, each containing some boundaries. 
If the curve divides the sphere into two parts such that one of them has exactly boundaries corresponding to edges going to $a_i, a_{i+1}$ then corresponding automorphisms $T_{i,i+1}$ look like
$$z_j\rightarrow z_j, j=1,\ldots ,i-1,i+2,\ldots ,g,\  z_i\rightarrow z_i(a_i^{z_i}a_{i+1}^{z_{i+1}}), \  z_{i+1}\rightarrow z_{i+1}(a_i^{z_i}a_{i+1}^{z_{i+1}}).$$

2.  The second statement follows from \cite{GK1}.  
\end{proof} 
\vspace{.5cm}

Let $j>i+1$. Denote $A_i=a_i^{z_i},  A_{ij}=a_i^{z_i}\ldots a_j^{z_j}$. Then $T_{ij}$ (the Dehn twist about a simple closed curve enclosing exactly the
$i$-th and $j$-th marked points) takes $z_i$ to $z_iA_{ij}{A_{i+1,j-1}}^{-1}$, takes $z_j$ to $z_j{A_{i+1,j-1}}^{-1}A_{ij}$ and does not move other variables. Here if $i+1=j-1$ we set $A_{i+1,j-1}=a_{i+1}^{z_{i+1}}.$

Notice that these automorphisms were used in the proof of the implicit function theorem in \cite{Imp}.

We have that 
\[
H=\mathrm{Stab}(a_1,\dots,a_g)\subset \mathrm{Mod}(\Sigma_g)
\]
that fixes each \(a_i\) (setwise).

Then:
\[
\mathrm{Stab}(a_1,\dots,a_g)\;\cong\; \mathbb{Z}^g \rtimes \mathrm{PMod}(S_{0,g+1})
\]

where
\(\mathbb{Z}^g\) is generated by Dehn twists \(T_{a_i}\),
PMod = pure  mapping class group (punctures return to their original positions)

The group $\operatorname{PMod}(S_{0,m})$ is generated by the elements
\[
T_{i,j}, \qquad 1 \leq i < j \leq m-1.
\]

The relations are as follows \cite[Lemma 4.1]{GW}:

1. Commutativity for separated pairs:
For $p<q<r<s$,
\[
[T_{p,q},T_{r,s}]=1.
\]

{2. Commutativity for nested pairs}
For $p<q<r<s$,
\[
[T_{p,s},T_{q,r}]=1.
\]

{3. Triangle relations}
For $p<q<r$,
\[
T_{p,r}T_{q,r}T_{p,q}
=
T_{q,r}T_{p,q}T_{p,r}
=
T_{p,q}T_{p,r}T_{q,r}.
\]

{4. Conjugation relations}
For $p<q<r<s$,
\[
\left[ T_{r,s}T_{p,r}T_{r,s}^{-1},\, T_{q,s} \right]=1.
\]

{5. Sphere relation}
\[
\bigl(T_{1,2}T_{1,3}\cdots T_{1,m-1}\bigr)
\bigl(T_{2,3}\cdots T_{2,m-1}\bigr)
\cdots
\bigl(T_{m-2,m-1}\bigr)=1.
\]

If we consider a subgroup generated only by $T_{i,i+1}$'s, then we have a RAAG normal form for elements.

 For example, repeated application of 
$$T_{g-1,g}^{*}\ldots  T_{2,3}^{*}T_{12}^{*}$$ is a normal form.

\begin{df} \label{onel} Balanced presentations obtained  from  $\sigma =T_{a_g}\ldots T_{a_1}T_{1,2}^{k_1}T_{2,3}^{k_2}\ldots  T_{g-1,g}^{k_{g-1}}$   will be called groups with one layer. The tuple $(k_1,\ldots ,k_{g-1})$ will be called a {\em one layer exponent tuple}.

 If  $\sigma =T_{a_g}\ldots T_{a_1}T_{1,2}^{k_1}T_{2,3}^{k_2}\ldots  T_{g-1,g}^{k_{g-1}}T_{1,2}^{m_1}T_{2,3}^{m_2}\ldots  T_{g-1,g}^{m_{g-1}}$, with all the powers $k_i,m_i$ of Dehn Twists non-trivial, then we have groups with two layers and so on.
\end{df} 
The purpose of this restriction is to isolate the simplest nontrivial families
of automorphisms while still producing highly nontrivial balanced
presentations. We will also slightly change equation (\ref{2}) and use equation
 \begin{equation}\label{2'}
 y_1^{-1}a_1y_1\ldots y_g^{-1}a_gy_g=a_1\ldots a_g.
 \end{equation}

The next lemma records precisely why this change of surface equation does not
change the associated balanced group.

\begin{lm}[Equivalence of the two surface equations and their balanced presentations]
\label{lem:equivalent-surface-equations}
Put
\[
 s_i=a_{i+1}a_{i+2}\cdots a_g\quad(1\leq i<g),
 \qquad s_g=1,
\]
and define
\[
 X_i=s_i^{-1}a_i s_i,
 \qquad
 Y_i=s_i^{-1}y_i s_i.
\]
Then
\[
 (a_1,y_1,\ldots,a_g,y_g)
 \longmapsto
 (X_1,Y_1,\ldots,X_g,Y_g)
\]
is an automorphism of the free group on these $2g$ letters, and equation
\eqref{2'} is transformed into the standard surface equation
\[
 [X_1,Y_1]\cdots[X_g,Y_g]=1.
\]

Let
\[
 \rho:F(a_1,y_1,\ldots,a_g,y_g)
 \longrightarrow F(a_1,\ldots,a_g)
\]
be the retraction defined by $\rho(a_i)=a_i$ and $\rho(y_i)=1$.  If
$\Omega$ is any of the automorphisms considered in this section, so that
$\Omega(a_i)=a_i$ for every $i$, put
\[
 R_i=\rho\Omega(y_i),
 \qquad
 \widetilde R_i=\rho\Omega(Y_i).
\]
Then
\[
 \widetilde R_i=s_i^{-1}R_i s_i.
\]
Consequently the two balanced relator tuples differ by conjugations of the
individual relators and by the triangular Nielsen change from the free basis
$(a_1,\ldots,a_g)$ to $(X_1,\ldots,X_g)$.  Hence the corresponding balanced
presentations are isomorphic and AC--Nielsen
equivalent.
\end{lm}

\begin{proof}
The change of variables is triangular from the right.  Indeed,
$X_g=a_g$ and $Y_g=y_g$, and, once $a_{i+1},\ldots,a_g$ have been recovered,
one recovers
\[
 a_i=s_iX_i s_i^{-1},
 \qquad
 y_i=s_iY_i s_i^{-1}.
\]
Thus it is a free-group automorphism.

With the convention $[u,v]=u^{-1}v^{-1}uv$, one has
\[
 y_i^{-1}a_i y_i=a_i[a_i,y_i].
\]
Writing $A=a_1\cdots a_g$ and moving each commutator factor past the later
$a$-letters gives
\[
 \begin{aligned}
 y_1^{-1}a_1y_1\cdots y_g^{-1}a_gy_g
 &=a_1[a_1,y_1]\cdots a_g[a_g,y_g]\\
 &=A\,[s_1^{-1}[a_1,y_1]s_1]\cdots
       [s_g^{-1}[a_g,y_g]s_g]\\
 &=A\,[X_1,Y_1]\cdots[X_g,Y_g].
 \end{aligned}
\]
Hence equation~\eqref{2'} is equivalent to the standard commutator equation.
Since $\Omega$ fixes every $a_j$, it fixes every $s_i$, and therefore
\[
 \widetilde R_i
 =\rho\Omega(s_i^{-1}y_i s_i)
 =s_i^{-1}\rho\Omega(y_i)s_i
 =s_i^{-1}R_i s_i.
\]
Individual conjugation of relators is an Andrews--Curtis move, while replacing
$(a_1,\ldots,a_g)$ by the free basis $(X_1,\ldots,X_g)$ is a simultaneous
Nielsen change of the ambient basis.  This proves the last assertion.
\end{proof}

Thus the relator tuple obtained from equation~\eqref{2'} is not literally the
same tuple as the one obtained from the standard commutator presentation, but
Lemma~\ref{lem:equivalent-surface-equations} shows that the two balanced
presentations are isomorphic and fixed-rank AC--Nielsen equivalent.  In
particular, one defines the trivial group if and only if the other does.  In
Section~\ref{Se:stand} we use the standard surface presentation, whereas in
Section~\ref{se:4} we use equation~\eqref{2'}.

{\bf Example.} Groups of genus \(n+1\) with one layer for presentation (\ref{2'}) are given by the relations

\begin{equation}\label{eq:6}
\begin{aligned}
&(a_1a_2)^{k_1}a_1=1,\\
&b_2=a_2(a_1a_2)^{k_1},\qquad
  a_2(a_1a_2)^{k_1}(a_2^{b_2}a_3)^{k_2}=1,\\
&b_3=a_3(a_2^{b_2}a_3)^{k_2},\qquad
  a_3(a_2^{b_2}a_3)^{k_2}(a_3^{b_3}a_4)^{k_3}=1,\\
&b_4=a_4(a_3^{b_3}a_4)^{k_3},\qquad
  a_4(a_3^{b_3}a_4)^{k_3}(a_4^{b_4}a_5)^{k_4}=1,\\
& \cdots \\
&(a_{n}^{b_{n}}a_{n+1})^{k_{n}}a_{n+1}=1,
\end{aligned}
\end{equation} where all $k_i$'s are non-zeros. The case when some $k_i$ are  zeros in the case of one layer reduces to free products of groups of smaller genus.  

{\bf Example.} Groups of genus 3 with two-layers presentation have the following form. Let $K$ be a free group.
\[
\begin{aligned}
&K = \langle a_1,a_2,a_3,x_1,x_2,x_3,y_1,y_2,y_3,z_1,z_2,z_3 \rangle, \\[3pt]
&L = K \Big/ \Big\langle 
y_1 = a_1 (a_1 a_2)^{k_1},\;
 y_2 = a_2 (a_1 a_2)^{k_1},\;
 y_3 = a_3, \\[3pt]
&\qquad z_1 = y_1,\;
 z_2 = y_2 (a_2^{y_2} a_3)^{k_2},\;
 z_3 = y_3 (a_2^{y_2} a_3)^{k_2},\\[3pt]
&\qquad x_1 = z_1 (a_1^{z_1} a_2^{z_2})^{m_1},\;
 x_2 = z_2 (a_1^{z_1} a_2^{z_2})^{m_1},\;
 x_3 = z_3,\\[3pt]
&\qquad x_1 = 1,\;
 x_2 (a_2^{x_2} a_3^{x_3})^{m_2} = 1,\;
 x_3 (a_2^{x_2} a_3^{x_3})^{m_2} = 1
\Big\rangle.
\end{aligned}
\]
Notice that for simplicity all $k_i$'s are non-zero.

\subsection{Olshanskii's epimorphisms}

Olshanskii in \cite[Section 3]{Ol} considered homomorphisms $\phi :\pi _1(S_2)\rightarrow F_2\times F_2 $ given (after simplification) by
$x_i\rightarrow b_ia_i$, $y_1\rightarrow a_1^{k}(a_1,a_2)^m, \  y_2\rightarrow a_2^{l}(a_1,a_2)^m.$  In genus 2, these are the only possible homomorphisms with $x_i\rightarrow b_ia_i$. The corresponding group is $$G_{\phi}=\langle a_1,a_2| a_1^k(a_1a_2)^m,  a_2^l(a_1a_2)^m\rangle.$$  For $|k|,|l|,|m|>1$ these groups are non-trivial because they have triangular quotients $$\bar G_{\phi}=\langle a_1,a_2| a_1^k, (a_1a_2)^m,  a_2^l\rangle.$$  For those $k,m,n$ for which they are finite of order $n$, Olshanskii constructs a finite cover of $S_2$ of genus $n+1$ such that the restriction 
$$\phi:\pi_1(S_{n+1})\rightarrow F_{n+1}\times F_{n+1}$$ is an epimorphism.
\begin{theorem}[to be proved in details in a subsequent paper] Olshanskii's epimorphisms  are standard.
\end{theorem}

\subsection{More on the Mapping Class Group $Mod(S_g)$}
We denote by $y_1,\ldots ,y_n$ meridian curves on the surface,  $x_1,\ldots ,x_n$ longitude curves and the bridge curves $c_i=y_ix_{i+1}\inv$ .

Here  $T_{x _i}(y_i)=x_iy_i$ and is identical on the other generators,   $T_{y _i}(x_i)=y_ix_i$ and is identical on the other generators.

Define the automorphism $T_{c_i}$ (the left Dehn twist about $c_i$) on the four
generators $\{x_i,y_i,x_{i+1},y_{i+1}\}$ by
\begin{equation}\label{eq:phi-def}
T_{c_i}:\quad
x_i\mapsto c_i\inv x_i,\qquad
y_i\mapsto y_i^{\,c_i}=\conj{c_i}{y_i},\qquad
x_{i+1}\mapsto x_{i+1}^{\,c_i}=\conj{c_i}{x_{i+1}},\qquad
y_{i+1}\mapsto c_i\inv y_{i+1},
\end{equation}
and $T_{c_i}$ fixes all other $x_j,y_j$ ($j\notin\{i,i{+}1\}$).
Geometrically, $x_i$ and $y_{i+1}$ cross the annulus of $c_i$ once and pick up a
left factor $c_i\inv$, while $y_i$ and $x_{i+1}$ run parallel to $c_i$ on
opposite sides and are conjugated by $c_i$.

  Denote the mapping class
group and extended mapping class group of $S_{g,1}$ as
$\operatorname{Mod}(S_{g,1})$ and $\operatorname{Mod}^{\pm}(S_{g,1})$
respectively.  
The
isomorphism given by setting the basepoint as the marked point
[\cite{FM} p.~235]
\[
  \varphi : \operatorname{Mod}^{\pm}(S_{g,1})
  \xrightarrow{\ \cong\ }
  \operatorname{Aut}(\pi_{1}(S_g))
\]
\begin{theorem}
The mapping class group
\[
\operatorname{Mod}^{+}(S_{g,1})
\]
is generated by the Dehn twists
\[
T_{c_i}, \quad i=1,\dots,g-1,
\]
\[
T_{x_i}, \quad i=1,\dots,g,
\]
together with
\[
T_{y_1},T_{y_2}.
\]
Moreover,
\[
\varphi(\operatorname{Mod}^{+}(S_{g,1}))
\]
together with a reflection \(R\) generate
\[
\operatorname{Aut}(\pi_1(S_g)).
\]
\end{theorem}
\begin{proof}  By the proof of Theorem 4.14 in \cite{FM},  one can obtain $T_{y_{i+2}}$ by conjugating $T_{y_i}$ by the element  in $T_{x_i}, C_i, T_{y_{i+1}}$, where $c_i$ is the belt curve in Figure 4.9 \cite{FM} and $C_i$ is the Dehn twists along that curve.  Also our $T_{c_i}$ can be obtaibed by conjugating $C_i$ by $T_{x_{i+1}} T_{y_{i+1}} T_{x_{i+1}} $. Therefore one can obtain $T_{y_{i+2}}$ by conjugating $T_{y_i}$ by the element  in $T_{x_i}, T_{x_{i+1}},T_{c_i}, T_{y_{i+1}}$. 
\end{proof}

Notice that these automorphisms were used in the proof of the implicit function theorem in \cite{Imp}. They are similar to Humphreys generators that represent a minimal, explicit set of $2g+1$ Dehn twists used to generate the mapping class group \cite{FM}.

\section{One-layer splitting epimorphisms are standard}
\label{Se:stand}

In this section we prove that every splitting epimorphism associated with 
 one-layer word is standard.  The proof uses the side-block quotient non-triviality
test and geometric moves on the two disk systems of a Heegaard diagram.

The proof does not use the Perelman theorem, geometrization,
or uniqueness of Heegaard splittings of the three-sphere.

\subsection{One-layer words and the side-block non-triviality test}

Let
\(
 K=(k_1,\ldots,k_{g-1})
\)
be a one-layer exponent tuple.  We use the recursive notation
\[
 \Theta_1=a_1,
 \qquad
 P_i=\Theta_i^{-1}a_i\Theta_i,
\]
\[
 Y_i=\Theta_i(P_i a_{i+1})^{k_i},
 \qquad
 \Theta_{i+1}=a_{i+1}(a_{i+1}P_i)^{k_i}
 \quad (1\leq i<g),
\]
and
\(
 Y_g=\Theta_g.
\)
Thus the corresponding splitting homomorphism is
\[
 \phi_K:\pi_1(S_g)\longrightarrow
 F(b_1,\ldots,b_g)\times F(a_1,\ldots,a_g),
\]
\[
 \phi_K(x_i)=(b_i,a_i),
 \qquad
 \phi_K(y_i)=(1,Y_i),
\]
and its associated balanced presentation is
\(
 G(K)=\left\langle a_1,\ldots,a_g\ \middle|\
 Y_1,\ldots,Y_g\right\rangle.
\)
By the criterion proved in Section~1,
\begin{equation}
 \phi_K\text{ is onto}
 \quad\Longleftrightarrow\quad
 G(K)=1.
 \label{eq:one-layer-onto-trivial}
\end{equation}

\medskip
\noindent
An exponent $k_i$ is called \emph{signed} if $k_i=1$ or $k_i=-1$.
An exponent which is not signed is called \emph{nonsigned}; since the
one-layer exponents considered here are nonzero, this means $|k_i|>1$.

\medskip
\noindent
A tuple
\[
 I=(k_t,\ldots,k_s),
 \qquad 1\leq t\leq s\leq g-1,
\]
of consecutive exponents is called an \emph{exponent interval}.
Let
\[
 G(I)=G_{s-t+2}(k_t,\ldots,k_s)
\]
be the ordinary one-layer group associated with this shorter tuple, written
in the local generators
\[
 a_1,\ldots,a_{s-t+2}.
\]
The recursion defining $G(I)$ starts from its own first generator; in
particular, it starts with
\[
 \Theta_1=P_1=a_1
\]
and does not use the global word $\Theta_t$ of
$G_g(k_1,\ldots,k_{g-1})$.

Replace the local generator $a_j$ of $G(I)$ by $a_{t+j-1}$.  The group
obtained in the generators
\[
 a_t,a_{t+1},\ldots,a_{s+1}
\]
is denoted by $G(t,I)$ and is called the \emph{block corresponding to the
exponent interval $I$}.  Thus a block is a group.

If $t\ne1$, the first relation defining $G(t,I)$ is the shifted first
relation of the shorter group $G(I)$; it is different from the $t$-th
relation of the entire group. 
If $s\ne g-1$, the last relation defining $G(t,I)$ is the shifted last
relation of $G(I)$; it is different from the $(s+1)$-st relation of the entire
group. 

For the prefix exponent interval
\[
 I=(k_1,\ldots,k_{i-1}),
\]
we keep the notation
\[
 G(i)=G_i(k_1,\ldots,k_{i-1}).
\]
Let
\[
 G_g=G_g(k_1,\ldots,k_{g-1}).
\]
For an interior bridge exponent $k_i$, between $a_i$ and $a_{i+1}$, the
left block is
\[
 G(i)=G_i(k_1,\ldots,k_{i-1}),
\]
and the shifted right block on $a_{i+1},\ldots,a_g$ is denoted by
\[
 T(g-i)=G_{g-i}(k_{i+1},\ldots,k_{g-1}).
\]

\begin{lm} (The right factor at a bridge)
After the relation
\[
 (a_i^{b_i}a_{i+1})^{k_i}=1
\]
is added to the presentation~\eqref{eq:6}, the relators on
$a_1,\ldots,a_i$ define $G(i)$ and the relators on
$a_{i+1},\ldots,a_g$ define the shifted block $T(g-i)$.  Hence the
resulting presentation is Tietze equivalent to the quotient
\[
 \frac{G(i)*T(g-i)}
 {\left\langle\!\left\langle
 (a_i^{b_i}a_{i+1})^{k_i}
 \right\rangle\!\right\rangle}.
\]
\end{lm}
\begin{proof}
The $i$-th relation of~\eqref{eq:6} has the form
\[
 b_i(a_i^{b_i}a_{i+1})^{k_i}=1.
\]
After the added relation is imposed, it gives $b_i=1$.  Since
\[
 b_i=a_i(a_{i-1}^{b_{i-1}}a_i)^{k_{i-1}},
\]
the word $b_i$ is conjugate to
\[
 (a_{i-1}^{b_{i-1}}a_i)^{k_{i-1}}a_i,
\]
which is the last relation of $G(i)$.
  Thus the relations on the left are
precisely those of $G(i)$, up to conjugating this last relation.

Also
\[
 b_{i+1}=a_{i+1}(a_i^{b_i}a_{i+1})^{k_i}=a_{i+1}.
\]
Consequently the next relation becomes
\[
 a_{i+1}(a_{i+1}a_{i+2})^{k_{i+1}}=1,
\]
which is conjugate to
\[
 (a_{i+1}a_{i+2})^{k_{i+1}}a_{i+1}=1,
\]
the first relation of the shifted block $T(g-i)$.  The subsequent formulas
for the words $b_j$ and the subsequent relations are exactly the shifted
formulas~\eqref{eq:6}.  Hence they define $T(g-i)$.  The only remaining
relation involving generators from both sides is the displayed bridge
relation.  This proves the lemma.
\end{proof}

\begin{theorem}
\label{thm:ordinary-one-layer-reduction}
  If $|k_i|>1$, and $G(i)=G_i(k_1,\ldots ,k_{i-1})\neq 1, T(g-i)=G_{g-i}(k_{i+1},\ldots ,k_{g-1})\neq 1,$ then 
$G_g\neq 1$.

\end{theorem}

\begin{proof}
We may  assume that killing the
boundary generator kills the corresponding side.  Namely,
\[
G(i)/\langle\!\langle a_i\rangle\!\rangle=1,
\qquad
T(g-i)/\langle\!\langle a_{i+1}\rangle\!\rangle=1.
\]
Indeed, if one of these quotients were non-trivial, it would already give a
non-trivial quotient of the whole group.  Therefore,  \[
a_i\ne 1\text{ in }G(i),
\qquad
a_{i+1}\ne 1\text{ in }T(g-i).
\]

Add the quotient relation
\[
(a_i^{b_i} a_{i+1})^{k_i}=1,
\]
where $(a_i^{b_i} a_{i+1})^{k_i}$ corresponds to the Dehn Twist $T_{i,i+1}^{k_i}$. The quotient is 
\[
Q_i=
\frac{
G(i)*T(g-i)
}{
\left\langle\!\left\langle (a_i ^{b_i}a_{i+1})^{k_i}
\right\rangle\!\right\rangle
}.
\]
Apply the generalized triangle group theorem of Baumslag--Morgan--Shalen:
 if
\[
\langle x,y\mid x^r=y^t=w^N=1\rangle
\]
is a generalized triangle group with \(|r|,|t|,|N|>1\), then it has a
special representation in which the indicated elements have the prescribed
orders, and in particular the group is non-trivial. 

Therefore
\(
Q_i\neq1.
\)
Since \(Q_i\) is a quotient of \(G_g\), this  implies  \(G_g\neq1\).\end{proof}

\medskip
\noindent

\begin{cy}
\label{cor:ordinary-one-layer-block-reduction}
Let
\[
 I=(k_t,\ldots,k_s)
\]
be an exponent interval, and let $G(t,I)$ be the corresponding block.  Let
$k_j=m$, $|m|>1$, be a nonsigned exponent in $I$.  If the two blocks
\[
 G(t,(k_t,\ldots,k_{j-1}))
 \qquad\text{and}\qquad
 G(j+1,(k_{j+1},\ldots,k_s))
\]
are both nontrivial, then $G(t,I)$ is nontrivial.
\end{cy}

\begin{proof}
Renumber $a_t,\ldots,a_{s+1}$ as
$a_1,\ldots,a_{s-t+2}$.  The block $G(t,I)$ becomes the ordinary one-layer
group $G(I)$.  The left side becomes the group $G(i)$ in
Theorem~\ref{thm:ordinary-one-layer-reduction}, and the right side is the
shifted block $T(g-i)$ by the lemma above.  The corollary is therefore an
immediate consequence of Theorem~\ref{thm:ordinary-one-layer-reduction}.
\end{proof}

\begin{df}[Recursive side-triviality tree]
\label{def:side-triviality-tree}
Let $G(t,I)$ be a trivial block.  A \emph{side-triviality tree} is obtained
by choosing, at every nonsigned exponent $k_j$ in $I$, one of the two side
blocks
\[
 G(t,(k_t,\ldots,k_{j-1}))
 \qquad\text{or}\qquad
 G(j+1,(k_{j+1},\ldots,k_s))
\]
which is trivial, and repeating the same procedure inside every chosen
trivial block.  A block whose exponent interval contains only signed
exponents is a terminal vertex of the tree.  At an endpoint, the side with
no exponent is the genus-one trivial block.
\end{df}

\begin{lm}
\label{lem:recursive-side-triviality}
Every trivial block $G(t,I)$ has a recursive side-triviality tree.
\end{lm}

\begin{proof}
Use strong induction on the number of nonsigned exponents in $I$.  There is
nothing to prove when every exponent in $I$ is signed.  Let $k_j$ be a
nonsigned exponent.  If both side blocks at $k_j$ were nontrivial,
Corollary~\ref{cor:ordinary-one-layer-block-reduction} would imply that
$G(t,I)$ is nontrivial.  Hence at least one side block is trivial.  Every
chosen trivial side block has fewer nonsigned exponents, so the induction
hypothesis applies to it.  Repeating the argument gives the required tree.
\end{proof}

\medskip
\noindent
The right-hand form of each reduction below is obtained by reversing the
order of the active generators and of the corresponding relators, and by
inverting the relevant relators when necessary.  These are Andrews--Curtis
transformations.  At the level of the splitting homomorphism this is the
handle-reversing symmetry: apply the inverse one-layer automorphism, interchange
the two free factors, and reverse the order of the handles.  Thus triviality
and standardness are unchanged when left and right are interchanged.  We use
this symmetry below without introducing a separate right-to-left recursion.

\subsection{AC-reduction from the recursive side-triviality tree}
\label{subsec:recursive-tree-AC-reduction}

We now give the presentation-level reduction which parallels the geometric
argument below.  It is given for the convenience of the reader only, Theorem \ref{thm:recursive-tree-AC-reduction}  can be also obtained as a corollary of Theorem \ref{th:14}. At every nonsigned
entry it uses the trivial side selected by the recursive side-triviality
tree of Definition~\ref{def:side-triviality-tree}.  
The allowed operations are inversions, conjugations, multiplications, and
permutations of the relators, together with simultaneous Nielsen changes of
the ambient free basis.    Singleton
relators which have already been obtained remain in the full tuple; we only
suppress them in the active notation after clearing their letters from the
remaining words. An explicit genus-ten example is given in Section~\ref{sec:examples}.

The basic clearing operation is the following.

\begin{lm}[Singleton clearing]
\label{lem:singleton-clearing}
Suppose an ordered relator tuple contains singleton relators
\[
 r_1=c_1,\ldots,r_s=c_s,
\]
where
\[
 (c_1,\ldots,c_s,t_1,\ldots,t_r)
\]
is a free basis of the ambient free group.  Let
\[
 \pi:F(c_1,\ldots,c_s,t_1,\ldots,t_r)
 \longrightarrow F(t_1,\ldots,t_r)
\]
be the retraction killing $c_1,\ldots,c_s$.  Then every other relator $R$
can be replaced, by Andrews--Curtis moves, by $\pi(R)$ while all
singleton relators $c_1,\ldots,c_s$ are retained.
\end{lm}

\begin{proof}
It is enough to delete one occurrence of one singleton letter.  Suppose
\[
 R=uc_j^{\varepsilon}v,
 \qquad \varepsilon\in\{1,-1\}.
\]
Temporarily replace the pivot relator $c_j$ by
\[
 v^{-1}c_j^{-\varepsilon}v
\]
using inversion, when necessary, and individual conjugation.  Multiply $R$
on the right by this temporary pivot.  The chosen occurrence cancels:
\[
 uc_j^{\varepsilon}v\,v^{-1}c_j^{-\varepsilon}v=uv.
\]
Restore the pivot relator to $c_j$.  Repeating this operation deletes every
occurrence of the singleton letters from $R$.  Applying the same procedure
to every non-pivot relator proves the lemma.
\end{proof}

\begin{df}
\label{def:relative-algebraic-active-interval}
Suppose that some relators have already been changed into singleton basis
elements and their letters have been cleared from all other relators.  Let
\[
 I=(k_r,\ldots,k_s)
\]
be the exponent interval which remains to be processed.  We call $I$ the
active interval if the following conditions hold.

\begin{enumerate}
\item The retained singleton relators, together with active basis elements
\[
 u_r,\ldots,u_{s+1},
\]
form a free basis of the original ambient free group.

\item The current relators in positions $r,\ldots,s+1$ are the images of the
actual relators of the full tuple under the Andrews--Curtis transformations
already performed.  They are not assumed to be the defining relators of the
block $G(r,I)$.

\item At an endpoint exposed by earlier singleton clearing, the two current
words which enter the next one-layer relation are powers of one primitive
live element.  When the interval is processed from the left, they have the
form
\[
 \Theta=t^{\alpha},\qquad P=t^{\beta}.
\]
\item
The defining presentation of every trivial block chosen in the
side-triviality tree is carried through the same basis changes and
relator operations.  If the chosen block lies to the left of a nonsigned
exponent, its last defining relation is retained in the presentation of
that block.  If the chosen block lies to the right, its first defining
relation is retained.  These relations belong only to the corresponding
block presentations; they are not inserted into the full relator tuple.

\end{enumerate}
\end{df}

In reducing the active interval $I=(k_r,\ldots,k_s)$, we first process the
$s-r+1$ actual relators in positions
\[
 r,r+1,\ldots,s,
\]
one for each exponent in $I$.  The actual relator in position $s+1$ is kept
until the end. When a proper block is used in the proof, its first or last defining
relation, according to which side of the nonsigned exponent the block
occupies, is retained until the other defining relations of that block
have been processed. 
\begin{lm}[Preservation of trivial blocks after singleton clearing]
\label{lem:trivial-blocks-singleton-clearing}
Suppose that a current relator has been changed into the singleton
$t^{\pm1}$.  Retain it as the singleton $t$ and clear the letter $t$ from the
remaining actual relators and from every block presentation being followed.
Then all these presentations still define the same groups as before.  In
particular, the choices of trivial blocks in the side-triviality tree remain
valid for the shorter active interval.
\end{lm}

\begin{proof}
Singleton clearing and the simultaneous changes of the ambient free basis
are Andrews--Curtis transformations, so every block presentation being
followed continues to define the same group.  A block disjoint from the part
already processed is unchanged, apart from these basis changes.  If a block
contains the processed part, apply the same transformations to its defining
presentation.  The processed relators become retained singleton relators.
After their letters are cleared, the lemma on the right factor at a bridge
identifies the remaining relators with the corresponding smaller block.
Thus a block marked trivial before the clearing leaves a trivial block after
the clearing.  The same argument applies at every later nonsigned exponent.
\end{proof}

\begin{prop}[Active-interval reduction]
\label{prop:open-algebraic-active-interval-reduction}
Let $I=(k_r,\ldots,k_s)$ be an active interval with its recursive
side-triviality tree.  By Andrews--Curtis transformations relative to the
retained singleton relators, the actual relators in positions
$r,\ldots,s$ can be changed into $s-r+1$ singleton basis elements.  One
primitive live element $t$ remains.  After the singleton letters are
cleared, every truncated relation of every block being followed is a power
of $t$.
\end{prop}

\begin{proof}
We use strong induction on the number of nonsigned exponents in $I$.

Suppose first that every exponent in $I$ is signed.  If the current pair is
\[
 \Theta=t^{\alpha},\qquad P=t^{\beta},
\]
and $z$ is the next unused active basis element, then at a signed exponent
$\varepsilon\in\{1,-1\}$ the actual relator is
\[
 R=t^{\alpha}(t^{\beta}z)^{\varepsilon}.
\]
For $\varepsilon=1$, the basis change
\[
 z\longmapsto t^{-\alpha-\beta}z
\]
sends $R=t^{\alpha+\beta}z$ to $z$.  For $\varepsilon=-1$, the basis change
\[
 z\longmapsto t^{-\beta}z^{-1}t^{\alpha}
\]
sends $R=t^{\alpha}z^{-1}t^{-\beta}$ to $z$.  Retain $z$ as a singleton
relator and clear it from all later actual relators and all block
presentations being followed.  Every new word which enters the next
relation is a word in $t,z$ before the clearing, and hence is a power of $t$
after $z$ is cleared.  Lemma~\ref{lem:trivial-blocks-singleton-clearing} shows that
the chosen trivial blocks remain valid.  Continuing through the signed
interval proves the proposition in this case.

Now suppose that $I$ contains a nonsigned exponent
\[
 k_j=m,\qquad |m|>1.
\]
Choose the trivial side specified by the side-triviality tree.  Suppose first
that the chosen side is the left block
\[
 G(r,(k_r,\ldots,k_{j-1})).
\]
Apply the induction hypothesis to the current actual relators in positions
$r,\ldots,j-1$, and apply the same transformations to the defining
presentation of this block.  These relators become singleton basis elements
and one primitive live element $t$ remains.  The truncated last relation of
the block becomes $t^u$, and the other current boundary word becomes $t^v$
for some integers $u,v$.  Since the chosen block is trivial, its transformed
presentation ends with
\[
 \langle t\mid t^u\rangle,
\]
and therefore
\begin{equation}
 u=\pm1.
 \label{eq:algebraic-terminal-exponent-sign}
\end{equation}
The truncated word $t^u$ has been used only in the presentation of the
chosen block; it has not been inserted into the full tuple.

Let $z$ be the first live basis element on the right.  The actual relator in
position $j$ is
\[
 R=t^u(t^vz)^m.
\]
First make the basis change
\[
 z\longmapsto t^{-v}z,
\]
which changes $R$ to $t^uz^m$.  If $u=1$, make
\[
 t\longmapsto tz^{-m};
\]
if $u=-1$, make
\[
 t\longmapsto z^m t^{-1}.
\]
Using~\eqref{eq:algebraic-terminal-exponent-sign}, either change sends the
actual relator in position $j$ to the singleton $t$.

Retain this singleton and clear its letter from all remaining actual
relators and all block presentations being followed.  The new live element
is $z$, and the remaining exponent interval is
\[
 (k_{j+1},\ldots,k_s).
\]
By Lemma~\ref{lem:trivial-blocks-singleton-clearing}, its chosen trivial blocks remain
valid, so the induction hypothesis applies to it.

If the side-triviality tree chooses the right block, use the symmetry
described above, apply the preceding left-side argument, and then reverse back.  This completes the induction.
\end{proof}

\begin{theorem}
\label{thm:recursive-tree-AC-reduction}
Let $K=(k_1,\ldots,k_{g-1})$ be a one-layer exponent tuple with
$G(K)=1$.  Then the ordered meridian tuple
\[
 (Y_1,\ldots,Y_g)
\]
is Andrews--Curtis equivalent to a free basis of
$F(a_1,\ldots,a_g)$, and hence to
\[
 (a_1,\ldots,a_g).
\]
\end{theorem}

\begin{proof}
By Lemma~\ref{lem:recursive-side-triviality}, the full exponent interval
\[
 (k_1,\ldots,k_{g-1})
\]
has a recursive side-triviality tree.  There are initially no retained
singleton relators.  Apply
Proposition~\ref{prop:open-algebraic-active-interval-reduction}.  The first
$g-1$ actual relators become singleton basis elements and one primitive live
element $t$ remains.  The genuine last relator is $Y_g=\Theta_g$, and after
the same transformations it is $t^q$ for some integer $q$.  The resulting
presentation is therefore the retained singleton presentation together with
\[
 \langle t\mid t^q\rangle.
\]
Since $G(K)=1$, one has $q=\pm1$.  Invert the last relator if necessary.
The resulting ordered tuple is a free basis.  A simultaneous change of this
free basis to $(a_1,\ldots,a_g)$ is one of the Andrews--Curtis
transformations allowed in this paper.
\end{proof}

\subsection{Heegaard diagrams, disk slides, and the two coordinates}

Let
\[
 M=H_A\cup_S H_B,
 \qquad S=\partial H_A=\partial H_B\cong S_g,
\]
be the Heegaard splitting associated with \(\phi_K\).  Write
\[
 q_A:\pi_1(S,*)\twoheadrightarrow\pi_1(H_A,*)=F(a_1,\ldots,a_g),
\]
\[
 q_B:\pi_1(S,*)\twoheadrightarrow\pi_1(H_B,*)=F(b_1,\ldots,b_g),
\]
so that \(\phi_K=(q_B,q_A)\).  Choose complete disk systems
\[
 \mathcal A=(A_1,\ldots,A_g)\subset H_A,
 \qquad
 \mathcal B=(B_1,\ldots,B_g)\subset H_B,
\]
and put
\[
 \alpha_i=\partial A_i,
 \qquad
 \beta_i=\partial B_i.
\]
After choosing a basepoint and whiskers, let \(\widehat\beta_i\) be the based
loop determined by \(\beta_i\), and set
\[
 r_i=q_A([\widehat\beta_i])\in F(a_1,\ldots,a_g).
\]
Initially \(r_i=Y_i\), up to inversion and conjugacy, while
\[
 q_B([\widehat\beta_i])=1.
\]

Once an isolated dual pair has been obtained, we choose a small genus-one
neighborhood of that pair.  To say that the pair is \emph{kept fixed} means
that every later disk slide and homeomorphism is supported outside this
chosen neighborhood.  Thus the two disks, their boundary curves, and their
intersection point are not changed by later moves.

Changes of the \(A\)-disk system and changes of the \(B\)-disk system have
different roles in the proof, namely:

\begin{enumerate}
\item Slides, permutations, and orientation changes of the \(A\)-disks are
used only as changes of the target basis of \(F(a_1,\ldots,a_g)\).  We do not
regard them as source automorphisms of the surface group.

\item Slides of the \(B\)-disks change the attaching curves.  They are induced
by homeomorphisms of \(H_B\); their restrictions to \(S\) therefore give
source automorphisms preserving the meridian kernel \(\ker q_B\).
\end{enumerate}

We recall the disk-slide operation precisely.

Let \(N\) be a sufficiently small regular neighborhood in \(H\) of
\[
   B_q\cup\alpha\cup B_p.
\]
The closure of the part of \(\partial N\) lying in the interior of
\(H\) consists of three properly embedded disks: one parallel to
\(B_q\), one parallel to \(B_p\), and a third disk, denoted
\[
   B_q*_{\alpha}B_p.
\]
The slide of \(B_q\) over \(B_p\) along \(\alpha\) is the operation
which replaces \(B_q\) by \(B_q*_{\alpha}B_p\) and leaves \(B_p\)
unchanged.

\begin{df}[Slide of one disk over another]
\label{def:disk-slide}
Let \(B_p,B_q\) be distinct disks of a complete disk system in \(H_B\), and
let \(\delta\subset S\) be an embedded arc with one endpoint on
\(\beta_q\), the other on \(\beta_p\), and interior disjoint from all curves
\(\beta_i\).  A regular neighborhood of
\[
 B_q\cup\delta\cup B_p
\]
has, in its frontier, a disk which is not parallel to \(B_p\) or \(B_q\).
Denote this disk by \(B_q*_{\delta}B_p\).  Replacing \(B_q\) by
\(B_q*_{\delta}B_p\), while leaving \(B_p\) unchanged, is called the
\emph{slide of \(B_q\) over \(B_p\) along \(\delta\)}.
\end{df}
\begin{figure}[ht!]
\centering
\includegraphics[width=0.6\textwidth]{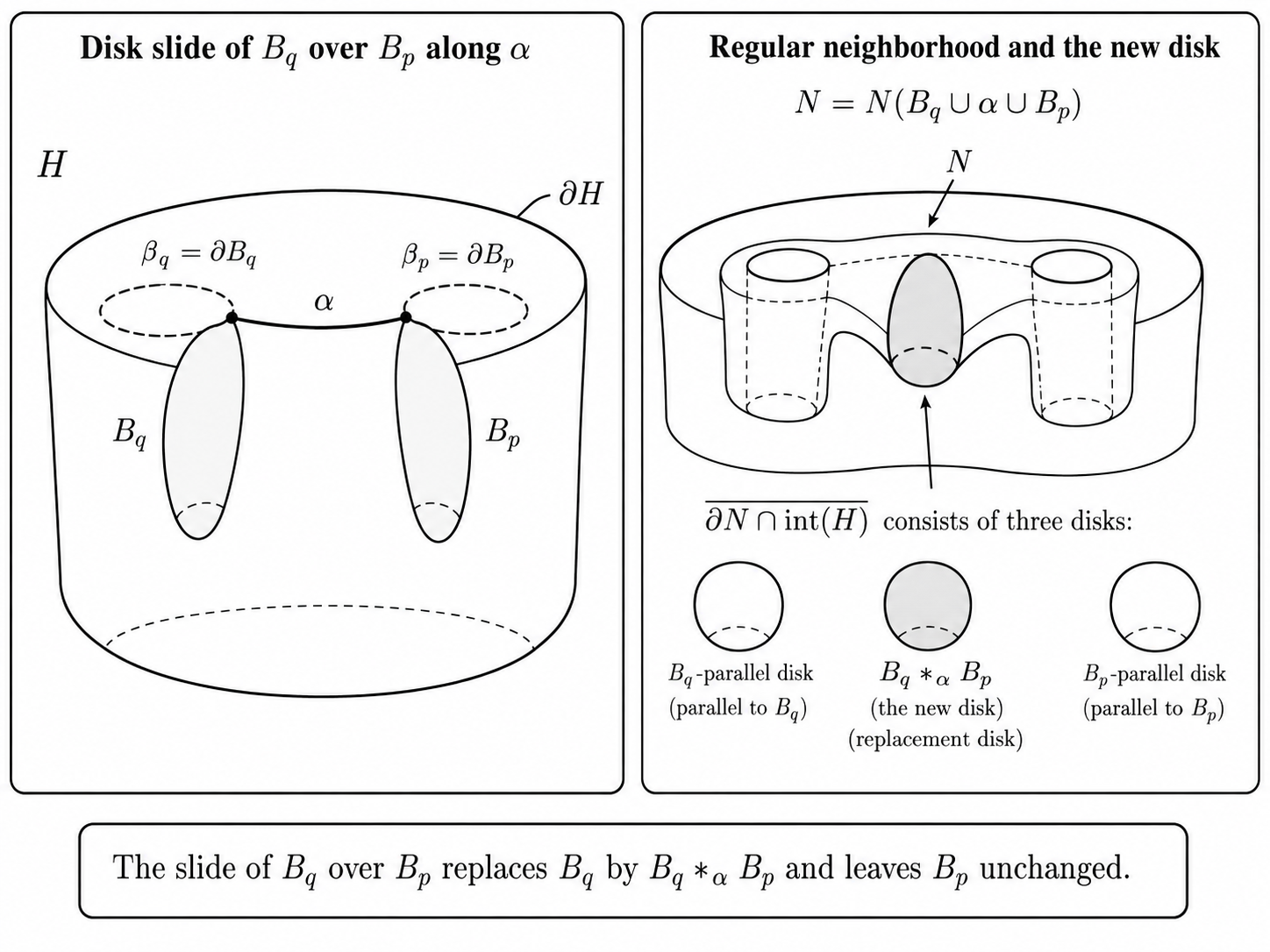}
\end{figure}
On the boundary, \(\partial(B_q*_{\delta}B_p)\) is the oriented band sum of
\(\beta_q\) and \(\beta_p\).  With based curves and whiskers, a slide changes
\(r_q\), according to the side and orientation of the band, to
\begin{equation}
 r_q\,w r_p^{\varepsilon}w^{-1}
 \quad\text{or}\quad
 w r_p^{\varepsilon}w^{-1}r_q,
 \qquad \varepsilon\in\{\pm1\}.
 \label{eq:disk-slide-word}
\end{equation}
Here \(w\in F(a_1,\ldots,a_g)\) is the \(q_A\)-image of the based route to
the copy of \(\beta_p\) used in the band sum.  A change of whisker may
conjugate the whole new word.  The word \(w\) is supplied by the actual band;
it is not an arbitrary prescribed element of the free group.  This is the
standard correspondence between disk slides and Nielsen moves in a
Heegaard diagram; see \cite{C,H,J}.

\begin{lm}
\label{lem:B-slide-preserves-kernel}
Let \(h_B:H_B\to H_B\) realize a slide of the \(B\)-disk system, let
\(U=(h_B|_S)_*\in\Aut\pi_1(S,*)\), and let
\(\nu_B=(h_B)_*\in\Aut F(b_1,\ldots,b_g)\).  Then
\begin{equation}
 q_B U=\nu_B q_B.
 \label{eq:B-slide-naturality}
\end{equation}
Consequently \(U(\ker q_B)=\ker q_B\).  In particular, every transformed
\(B\)-meridian still has trivial \(b\)-coordinate.
\end{lm}

\begin{proof}
Equation~\eqref{eq:B-slide-naturality} is the naturality relation for the
inclusion \(S=\partial H_B\hookrightarrow H_B\).  If the support of the slide
is disjoint from the basepoint, it holds as written; otherwise it holds up to
an inner automorphism, which has the same consequence for the kernel.  Thus
\[
 q_B U(y)=1
 \qquad\text{whenever}\qquad q_B(y)=1.
\]
Equivalently, the new boundary curve bounds the slid disk in \(H_B\).
\end{proof}

\begin{lm}[Relative realization of $F(a_1,\ldots ,a_g)$  Nielsen moves]
\label{lem:relative-nielsen-realization}
Suppose that some isolated dual pairs have already been kept fixed.  Every
Nielsen transformation of the free basis carried by the complementary
\(A\)-handles, extended by the identity on the fixed basis elements, is
realized by slides, permutations, and orientation changes of the active
\(A\)-disks supported outside the fixed neighborhoods.
\end{lm}

\begin{proof}
Permutations and inversions are realized by permuting and reorienting the
active handles.  The elementary transformations
\[
 u_j\longmapsto u_i^{\pm1}u_j,
 \qquad
 u_j\longmapsto u_j u_i^{\pm1}
\]
are realized by sliding the \(j\)-th active \(A\)-disk over the \(i\)-th
active \(A\)-disk, with the side and orientation of the slide determining the
formula.  These moves generate the Nielsen group.  Since they involve only
active handles, they can be supported outside the fixed neighborhoods.
\end{proof}

The following lemma is the geometric replacement for the formal operation of
clearing a singleton relator.

\begin{lm}[Isolation of a primitive attaching curve]
\label{lem:geometric-singleton-cancellation}
Suppose that, in the current active  basis
\((u_1,\ldots,u_n)\) of the target group $F(a_1,\ldots ,a_g)$, an attaching curve \(\beta_p\) represents
\(t^{\varepsilon}\), where \(t\) is one of the active basis elements and
\(\varepsilon=\pm1\).  Then, by changes of the active \(A\)-disk system and
slides of the active \(B\)-disks, all supported outside the previously fixed
neighborhoods, one can arrange that
\[
 |\alpha_t\cap\beta_p|=1,
\]
\[
 \alpha_i\cap\beta_p=\varnothing\quad(i\neq t),
 \qquad
 \alpha_t\cap\beta_q=\varnothing\quad(q\neq p).
\]
Consequently \((A_t,B_p)\) is an isolated dual pair and may be kept fixed.
All \(B\)-disk slides used in this process preserve \(\ker q_B\), and hence
preserve the trivial \(b\)-parts of the meridians.
\end{lm}

\begin{proof}
The curve \(\beta_p\) represents a primitive element of
\(\pi_1(H_A)\).  By the primitive-curve criterion for handlebodies, there is
an essential disk in the active \(A\)-handlebody meeting \(\beta_p\) in one
point; see, for example, \cite{H,Moriah}.  Extend it to a complete active disk
system whose remaining disks are disjoint from \(\beta_p\).  Complete disk
systems are related by disk slides, and the relative version of
Lemma~\ref{lem:relative-nielsen-realization} permits this change outside the
fixed neighborhoods.  Relabel the dual disk as \(A_t\).  We then have
\[
 |\alpha_t\cap\beta_p|=1,
 \qquad
 \alpha_i\cap\beta_p=\varnothing\quad(i\neq t).
\]

Let \(q\neq p\) and suppose that \(\beta_q\) meets \(\alpha_t\).  Along
\(\alpha_t\), choose an intersection point adjacent to the unique point of
\(\alpha_t\cap\beta_p\).  The intervening subarc of \(\alpha_t\) has interior
disjoint from the entire \(\beta\)-system.  A narrow band parallel to that
subarc defines a slide of \(B_q\) over \(B_p\).  With the appropriate
orientation, the slide removes the chosen intersection and creates no
intersection with any other \(\alpha_i\), because \(\beta_p\) is disjoint
from those curves.  Repeating the operation clears \(\alpha_t\) from every
\(\beta_q\), \(q\neq p\).  Lemma~\ref{lem:B-slide-preserves-kernel} shows that
all these slides preserve the trivial \(b\)-coordinate.  The pair
\((A_t,B_p)\) is now isolated.
\end{proof}

The next statement is the point which makes the relative induction
independent of the geometrically restricted conjugating words in
\eqref{eq:disk-slide-word}.

\begin{lm}
\label{lem:pivot-contraction}
Assume that the attaching word of $B_p$ is $t^{\pm1}$ in the current active
$A$-basis.  Perform any finite sequence of slides of the other active
$B$-disks over $B_p$, and let $r_q'$ be the resulting attaching words.  In
the quotient obtained by setting $t=1$, the images of $r_q'$ and $r_q$
agree up to conjugacy.  Consequently, after the isolated pair
$(A_t,B_p)$ is suppressed from the active notation, the remaining
presentation is Tietze equivalent to the presentation obtained by setting
$t=1$ in the words before the clearing slides.

The same conclusion holds for every block presentation being followed.  In
particular, every block which was trivial before the pair was suppressed
leaves the corresponding trivial block on the remaining exponent interval.
\end{lm}

\begin{proof}
By~\eqref{eq:disk-slide-word}, each slide multiplies one attaching word, on
the left or on the right, by
\[
 w t^{\delta}w^{-1},
 \qquad \delta\in\{1,-1\},
\]
where $w$ is determined by the band.  This factor becomes trivial after
$t=1$.  A change of whisker only conjugates the whole attaching word.
Therefore the remaining full presentation is Tietze equivalent to the
direct quotient presentation.

Apply the same basis changes and disk slides to every block presentation
being followed.  If the block contains the isolated pair, the singleton is
cleared there as well; the lemma on the right factor at a bridge identifies
the remaining relations with the corresponding smaller block.  If the block
does not contain the pair, it is unchanged apart from basis changes.  Thus
triviality of the chosen blocks is preserved.
\end{proof}

\begin{remark}
Lemma~\ref{lem:pivot-contraction} does not assert that an arbitrary formal
Andrews--Curtis conjugation is realized by a disk slide.  The bands supply the
conjugating words.  The lemma says that their precise values are irrelevant
once the primitive singleton is set equal to one.
\end{remark}

\subsection{Active intervals in the Heegaard diagram}

After some isolated dual pairs have been kept fixed, the remaining handles
correspond to a consecutive exponent interval.  We use the same active
interval as in the Andrews--Curtis reduction, but now keep track of the
actual $B$-attaching curves.

\begin{df}
\label{def:relative-geometric-active-interval}
Let
\[
 I=(k_r,\ldots,k_s)
\]
be an exponent interval.  A Heegaard diagram is in active form on $I$ if the
following conditions hold.

\begin{enumerate}
\item Every handle already processed belongs to an isolated dual pair whose
chosen neighborhood is kept fixed.

\item The remaining $A$-disks determine active basis elements
\[
 u_r,\ldots,u_{s+1}.
\]
Together with the basis elements belonging to the fixed pairs, they form a
basis of $F(a_1,\ldots,a_g)$.

\item The current $B$-attaching curves in positions $r,\ldots,s+1$ are the
images of the actual attaching curves under the disk moves already
performed.  Their words are not assumed to be the defining relators of the
block $G(r,I)$.

\item At an exposed left endpoint, the two current words which enter the
next one-layer relation have the form
\[
 \Theta=t^{\alpha},\qquad P=t^{\beta},
\]
where $t$ is primitive.  The defining presentation of every trivial block
chosen in the side-triviality tree is followed under the same moves.  Its
truncated relation is an auxiliary relation of that block presentation, not
an actual attaching curve of the full Heegaard diagram.
\end{enumerate}
\end{df}

In reducing $I=(k_r,\ldots,k_s)$, we first treat the actual attaching curves
in positions $r,\ldots,s$, one for each exponent in $I$.  The actual curve in
position $s+1$ is kept until the end.

Suppose that a nonsigned exponent
\[
 k_j=m,\qquad |m|>1,
\]
is being processed and that the chosen trivial side is the left block
\[
 G(r,(k_r,\ldots,k_{j-1})).
\]
After its attaching curves corresponding to the exponent entries have been
processed, let its truncated last relation have word $\Theta$ and let the
other current boundary word be $P$.  If $z$ is the first live generator on
the right, then the actual attaching curve in position $j$ has word
\begin{equation}
 \Theta(Pz)^m.
 \label{eq:actual-cut-relation}
\end{equation}
The truncated word $\Theta$ is used only in the presentation of the chosen
block.  It is not an attaching curve in the full diagram.  If the chosen
trivial block is on the right, the same statement is used after applying the
symmetry of exponent tuples.

After an isolated primitive pair is suppressed from the active notation,
Lemma~\ref{lem:pivot-contraction} shows that every block presentation being
followed has been changed only by Tietze transformations and by setting the
singleton equal to one.  The lemma on the right factor at a bridge then
identifies the remaining presentation with the corresponding block.  Hence
the recursive choices of trivial blocks remain valid.

\begin{prop}[Active-interval reduction in the Heegaard diagram]
\label{prop:open-active-interval-reduction}
Let $I=(k_r,\ldots,k_s)$ be an active interval with its recursive
side-triviality tree.  By changes of the active $A$-disk system and slides of
the active $B$-disk system, all performed relative to the fixed pairs, the
actual attaching curves in positions $r,\ldots,s$ can be changed into
$s-r+1$ isolated dual curves.  One active handle remains.  Every truncated
relation of every block being followed becomes a power of the primitive
generator of this remaining handle.
\end{prop}

\begin{proof}
We use strong induction on the number of nonsigned exponents in $I$.

If every exponent in $I$ is signed, let the current pair be
\[
 \Theta=t^{\alpha},\qquad P=t^{\beta},
\]
and let $z$ be the next unused active generator.  At a signed exponent
$\varepsilon=\pm1$ the attaching word is
\[
 R=t^{\alpha}(t^{\beta}z)^{\varepsilon}.
\]
For $\varepsilon=1$, the target basis change
\[
 z\longmapsto t^{-\alpha-\beta}z
\]
sends $R$ to $z$.  For $\varepsilon=-1$, the target basis change
\[
 z\longmapsto t^{-\beta}z^{-1}t^{\alpha}
\]
also sends $R$ to $z$.  By
Lemma~\ref{lem:relative-nielsen-realization}, these changes are realized by
changes of the active $A$-disk system.  Then
Lemma~\ref{lem:geometric-singleton-cancellation} isolates the corresponding
$B$-attaching curve by genuine $B$-disk slides.  Keep the resulting dual pair
fixed and set its singleton equal to one in the active words.
Lemma~\ref{lem:pivot-contraction} preserves the chosen trivial blocks.
Continuing through the signed interval proves the proposition in this case.

Now choose a nonsigned exponent $k_j=m$ and the trivial block specified by
the side-triviality tree.  Suppose first that the chosen block is
\[
 G(r,(k_r,\ldots,k_{j-1})).
\]
Apply the induction hypothesis to the actual attaching curves in positions
$r,\ldots,j-1$ and apply the same moves to the defining presentation of the
chosen block.  These curves become isolated dual curves and one primitive
element $t$ remains.  The truncated last relation of the chosen block has
word $t^u$, and the other current boundary word has the form $t^v$.  Since
the chosen block is trivial,
\(
 \langle t\mid t^u\rangle=1,
\)
so
\begin{equation}
 u=\pm1.
 \label{eq:terminal-exponent-sign}
\end{equation}
The truncated relation $t^u$ has been used only to determine the group of
the block.

Let $z$ be the first live generator on the right.  The actual attaching
curve in position $j$ has word
\[
 R=t^u(t^vz)^m.
\]
First make the target basis change
\(
 z\longmapsto t^{-v}z,
\)
which changes $R$ to $t^uz^m$.  If $u=1$, make
\(
 t\longmapsto tz^{-m};
\)
if $u=-1$, make
\(
 t\longmapsto z^m t^{-1}.
\)
Using~\eqref{eq:terminal-exponent-sign}, either change sends the actual cut
word to the primitive singleton $t$.  Realize these changes by active
$A$-disk moves and use
Lemma~\ref{lem:geometric-singleton-cancellation} to isolate the corresponding
$B$-disk by actual disk slides.

The remaining exponent interval is $(k_{j+1},\ldots,k_s)$, with $z$ as its
primitive boundary generator.  Lemma~\ref{lem:pivot-contraction} and the
right-factor lemma preserve all chosen trivial blocks, so the induction
hypothesis applies to the remaining interval.  If the chosen trivial block is on the right, use the symmetry described
above, apply the preceding left-side argument, and then reverse back.
\end{proof}

\begin{theorem}[Geometric active-interval theorem]
\label{thm:geometric-active-interval}
Assume that $G(K)=1$ and that the full exponent interval
\[
 (k_1,\ldots,k_{g-1})
\]
has its recursive side-triviality tree.  Then the full Heegaard diagram can
be changed into a complete collection of isolated dual pairs.
\end{theorem}

\begin{proof}
By Proposition~\ref{prop:open-active-interval-reduction}, the actual
attaching curves in positions $1,\ldots,g-1$ become isolated dual curves and
one primitive live element $t$ remains.  The genuine last attaching curve
has word $Y_g=\Theta_g$, which becomes $t^q$ for some integer $q$.  The
resulting presentation is the retained singleton presentation together with
\(
 \langle t\mid t^q\rangle.
\)
Since $G(K)=1$, one has $q=\pm1$.  Reorient the last curve if necessary and
apply Lemma~\ref{lem:geometric-singleton-cancellation} to isolate the final
dual pair.
\end{proof}

\subsection{Standardness }

\begin{theorem}\label{th:14}
Every  one-layer splitting epimorphism
\[
 \phi_K:\pi_1(S_g)\longrightarrow F_g\times F_g
\]
is standard.
\end{theorem}

\begin{proof}
First assume that all \(k_i\) are nonzero.  Since \(\phi_K\) is onto,
\eqref{eq:one-layer-onto-trivial} gives \(G(K)=1\).  By Lemma~\ref{lem:recursive-side-triviality}, the full exponent interval
\[
 (k_1,\ldots,k_{g-1})
\]
has a recursive side-triviality tree.  Apply Theorem~\ref{thm:geometric-active-interval} to the full Heegaard diagram.
There are initially no fixed pairs, and the genuine terminal attaching curve
is \(Y_g=\Theta_g\).  The theorem changes the diagram into a complete dual
diagram
\[
 |\alpha_i\cap\beta_j|=\delta_{ij}
\]
after reindexing and reorienting the curves.

Let \(U\in\Aut\pi_1(S_g)\) be the product of the boundary automorphisms
induced by the \(B\)-disk slides.  Let \(\nu_B\in\Aut F(b_1,\ldots,b_g)\)
be the product of the corresponding handlebody automorphisms, so that
\[
 q_B U=\nu_B q_B.
\]
Let \(\eta_A\in\Aut F(a_1,\ldots,a_g)\) be the product of the target basis
changes induced by the \(A\)-disk moves.  Thus all moves used in the proof
preserve the equivalence class of the splitting homomorphism: the
\(B\)-moves are accounted for by precomposition with \(U\) and
postcomposition with \(\nu_B^{-1}\), while the \(A\)-moves are accounted for
by postcomposition with \(\eta_A\).

A complete dual Heegaard diagram is the standard diagram.  Consequently
\[
 (\nu_B^{-1}\times\eta_A)\,\phi_K\,U
\]
is equivalent, after the final permutations and orientation changes, to the
canonical splitting homomorphism.  Hence \(\phi_K\) is standard.

If some \(k_i=0\), the one-layer diagram separates at that position as a
connected sum of lower-genus one-layer diagrams, and the associated balanced
presentation is the free product of the corresponding presentations.  Since
the full group is trivial, every factor is trivial.  Apply the preceding
argument to each component.  The required moves have disjoint supports, and
their boundary connected sum gives the standard genus-\(g\) diagram.
\end{proof}

Theorem \ref{thm:recursive-tree-AC-reduction} follows as a corollary.

\begin{proof}
An \(A\)-disk move is a simultaneous Nielsen change of the target basis.  A
\(B\)-disk slide has the form~\eqref{eq:disk-slide-word}; reorientation gives
relator inversion and reindexing gives relator permutation.  The proof never
adds or deletes a generator--relator pair in the full rank-\(g\) presentation:
when an isolated singleton is ``suppressed,'' it remains as a fixed dual pair
in the full diagram.  The final tuple is a free basis.  No claim is made that
an arbitrary preassigned Andrews--Curtis conjugator is realized by one disk
slide.
\end{proof}

\subsection{Generalized one-layer and framed graph-layer presentations}
\label{subsec:generalized-one-layer-definitions}

The preceding theorem concerns  one-layer presentations.  We now give
precise terminology for the broader conjectural class.

\begin{df}
Let
\[
 \mathcal T_g=\langle T_{ij}\mid 1\leq i<j\leq g\rangle
 \leq \operatorname{PMod}(S_{0,g+1})
\]
be the subgroup generated by the paper twists introduced in
Section~\ref{Autom}.  A \emph{graph-layer word} is a finite word
\[
 W=T_{i_mj_m}^{k_m}\cdots T_{i_1j_1}^{k_1},
 \qquad k_\ell\in\mathbb Z\setminus\{0\}.
\]
Two graph-layer words are called \emph{equivalent} if they represent the same
element of $\mathcal T_g$, equivalently if one can be obtained from the other
using the defining twist relations listed in Section~\ref{Autom}.

The \emph{support graph} $\Gamma(W)$ has vertex set $\{1,\ldots,g\}$ and an
edge $\{i,j\}$ whenever a power of $T_{ij}$ occurs in $W$.
\end{df}

\begin{df}[Leaf peel and recursively peelable twist element]
\label{def:recursively-peelable}
Let $W$ be a graph-layer word.  We say that $W$ \emph{admits a leaf peel at
$v$} if, after replacing $W$ by an equivalent representative and collecting
adjacent powers, it can be written
\begin{equation}
 W\equiv T_{vw}^{p}\,W'\,T_{vw}^{q},
 \label{eq:leaf-peel-form}
\end{equation}
where
\[
 p,q\in\mathbb Z,
 \qquad (p,q)\neq(0,0),
\]
$v$ is a leaf of the support graph of the displayed representative, the
unique edge incident to $v$ is $\{v,w\}$, and $W'$ contains no twist whose
support contains $v$.

An element $\Omega\in\mathcal T_g$ is called \emph{recursively peelable} if
one of the following holds:
\begin{enumerate}
\item $\Omega=1$;
\item $\Omega$ has a representative of the form~\eqref{eq:leaf-peel-form}
and the element represented by $W'$ is recursively peelable.
\end{enumerate}
A graph-layer word is recursively peelable when the element it represents is
recursively peelable.  A graph-layer element which is not recursively
peelable will be called \emph{genuinely multilayer}.
\end{df}

The definition is made at the level of the element of $\mathcal T_g$, not at
the level of one chosen spelling.  Thus disjoint twists may first be commuted,
and the triangle and sphere relations may be used before a leaf is peeled.

\begin{example}
Every ordinary one-layer path word
\[
 T_{g-1,g}^{k_{g-1}}\cdots T_{23}^{k_2}T_{12}^{k_1}
\]
is recursively peelable: first peel the leaf $g$, then $g-1$, and continue
along the path.  The nested word
\[
 T_{12}^{p}\,T_{23}^{q}\,T_{12}^{r}
\]
is also recursively peelable, since vertex $1$ is a leaf and the two outer
$T_{12}$-powers are removed in one leaf peel.  By contrast, an interleaving
word or a word whose support contains a cycle is not declared peelable unless
the twist relations first transform it to a representative satisfying the
recursive definition.
\end{example}

Let
\[
 \gamma_0:\pi_1(S_g)\longrightarrow F(a_1,\ldots,a_g),
 \qquad
 \gamma_0(x_i)=a_i,
 \quad
 \gamma_0(y_i)=1,
\]
be the standard retraction.

\begin{df}[Graph-layer presentation]
For $\Omega\in\mathcal T_g$, the associated \emph{graph-layer presentation}
is
\[
 G(\Omega)=
 \left\langle a_1,\ldots,a_g\ \middle|\
 \gamma_0\Omega(y_1),\ldots,\gamma_0\Omega(y_g)
 \right\rangle.
\]
It is called a \emph{generalized one-layer presentation} when $\Omega$ is
recursively peelable.
\end{df}

\begin{df}[Framing and framed graph-layer presentation]
Let
\[
 \mathbf f=(f_1,\ldots,f_g)\in\mathbb Z^g,
 \qquad
 V_{\mathbf f}=T_{x_1}^{f_1}\cdots T_{x_g}^{f_g}.
\]
The vector $\mathbf f$ is called the \emph{framing vector}.  The
\emph{framed graph-layer presentation} associated with
$(\mathbf f,\Omega)$ is
\[
 G(\mathbf f,\Omega)=
 \left\langle a_1,\ldots,a_g\ \middle|\
 \gamma_0V_{\mathbf f}\Omega(y_1),\ldots,
 \gamma_0V_{\mathbf f}\Omega(y_g)
 \right\rangle.
\]
The twists $T_{x_i}$ commute with the paper twists $T_{ij}$, and therefore
this presentation may equivalently be written
\begin{equation}
 G(\mathbf f,\Omega)=
 \left\langle a_1,\ldots,a_g\ \middle|\
 a_1^{f_1}\gamma_0\Omega(y_1),\ldots,
 a_g^{f_g}\gamma_0\Omega(y_g)
 \right\rangle.
 \label{eq:framed-graph-layer-presentation}
\end{equation}
It is called a \emph{framed generalized one-layer presentation} if
$\Omega$ is recursively peelable.  The ordinary presentations used in the
one-layer and two-layer sections correspond to the framing vector
$\mathbf 1=(1,\ldots,1)$.
\end{df}

By Proposition~\ref{l3}, every presentation considered in this paper which
comes from an automorphism fixing the $x_i$'s can be expressed in framed
graph-layer form.

\begin{conjecture}[Generalized one-layer conjecture]
\label{conj:generalized-one-layer}
If a framed graph-layer presentation $G(\mathbf f,\Omega)$ defines the
trivial group, then $\Omega$ is recursively peelable; equivalently,
$G(\mathbf f,\Omega)$ is a framed generalized one-layer presentation.
\end{conjecture}

The present paper proves the ordinary one-layer standardness theorem and
several nontriviality results for two layers, but it does not prove
Conjecture~\ref{conj:generalized-one-layer} in full.

\section {Balanced presentations of genus 3 from two layers of Dehn Twists}\label{se:4}
\begingroup
\setlength{\abovedisplayskip}{4pt plus 2pt minus 1pt}
\setlength{\belowdisplayskip}{4pt plus 2pt minus 1pt}
\setlength{\abovedisplayshortskip}{2pt plus 1pt}
\setlength{\belowdisplayshortskip}{2pt plus 1pt minus 1pt}
\setlength{\jot}{2pt}
We prove results that are necessary for studying
Conjecture~\ref{conj:generalized-one-layer}.  In this section we consider
balanced presentations arising from genus \(3\) automorphisms obtained from
two layers of Dehn twists.  The goal is to show that none of these
presentations defines the trivial group.  We treat the exponent-sum case
$(0,0)$ first, because its marked-diagram argument will also be used in the other
cases .  The remaining exponent-sum cases are then handled by a
combination of generalized-triangle quotients, the same diagrammatic
argument, and finitely many GAP computations.  
\begin{theorem}\label{thm:genus3-two-layers}
Let \(G\) be a balanced presentation associated to a genus \(3\) automorphism
obtained from two layers of Dehn twists (all $k_i$'s are non-zero). Then
$G\neq 1.$ 
\end{theorem}

The proof of this theorem begins now and will be completed at the end of Section \ref{se:-1,0}.

We begin with nontriviality of abelianization criterion which is valid in every genus, for
any number of adjacent path layers, and with arbitrary framing.

\begin{prop}[Abelianization of a  multilayer presentation]
\label{prop:multilayer-abelianization}
For $1\leq \nu\leq r$, let
\[
 \Lambda_\nu=
 T_{12}^{k_{\nu,1}}T_{23}^{k_{\nu,2}}\cdots
 T_{g-1,g}^{k_{\nu,g-1}},
\]
and put
\[
 \Omega=\Lambda_r\cdots\Lambda_1,
 \qquad
 s_i=\sum_{\nu=1}^r k_{\nu,i},
 \qquad s_0=s_g=0.
\]
Let $\mathbf f=(f_1,\ldots,f_g)$ be the framing vector.  Then the
abelianization of $G(\mathbf f,\Omega)$ has relation matrix
\begin{equation}
 M_g(\mathbf f,\mathbf s)=
 \begin{pmatrix}
 f_1+s_1&s_1&0&\cdots&0\\
 s_1&f_2+s_1+s_2&s_2&\ddots&\vdots\\
 0&s_2&f_3+s_2+s_3&\ddots&0\\
 \vdots&\ddots&\ddots&\ddots&s_{g-1}\\
 0&\cdots&0&s_{g-1}&f_g+s_{g-1}
 \end{pmatrix}.
 \label{eq:multilayer-abelianization-matrix}
\end{equation}
Consequently,
\[
 \det M_g(\mathbf f,\mathbf s)\neq\pm1
 \quad\Longrightarrow\quad
 G(\mathbf f,\Omega)\neq1.
\]
\end{prop}

\begin{proof}
In abelianization all conjugations disappear.  A power
$T_{i,i+1}^{k}$ adds $k(a_i+a_{i+1})$ to each of the two relators in
positions $i$ and $i+1$, and does not change the other relators.  Hence the
contributions of the different layers depend only on the total exponent
$s_i$ on each adjacent edge.  The framing factor $T_{x_i}^{f_i}$ adds
$f_i a_i$ to the $i$-th relator.  This gives the matrix
\eqref{eq:multilayer-abelianization-matrix}.

The abelianization is the cokernel of this square integer matrix.  If the
presented group were trivial, this cokernel would be zero, which is equivalent
to the determinant being $\pm1$.
\end{proof}

\begin{cy}
\label{cor:genus3-two-layer-abelianization}
For the ordinary framing $\mathbf f=(1,1,1)$ and two layers, put
\[
 s_1=k_1+m_1,
 \qquad
 s_2=k_2+m_2.
\]
Then
\[
 \det M_3(\mathbf 1,(s_1,s_2))
 =1+2s_1+2s_2+3s_1s_2.
\]
This determinant is $\pm1$ only for
\[
 (-2,-1),\quad(-1,-2),\quad(0,0),\quad(-1,0),\quad(0,-1).
\]
Thus every other exponent-sum pair gives a nontrivial abelianization and
hence a nontrivial group.
\end{cy}

\begin{proof}
For determinant $1$ one has
\[
 (3s_1+2)(3s_2+2)=4,
\]
and for determinant $-1$ one has
\[
 (3s_1+2)(3s_2+2)=-2.
\]
The integer factor pairs which are congruent to $2$ modulo $3$ give exactly
the five displayed pairs.
\end{proof}

The relations of $G=G_{\phi} $ are as follows: Let $K$ be a free group.
\[
\begin{aligned}
&K = \langle a_1,a_2,a_3,x_1,x_2,x_3,y_1,y_2,y_3,z_1,z_2,z_3 \rangle, \\[3pt]
&G = K \Big/ \Big\langle 
y_1 = a_1 (a_1 a_2)^{k_1},\;
 y_2 = a_2 (a_1 a_2)^{k_1},\;
 y_3 = a_3, \\[3pt]
&\qquad z_1 = y_1,\;
 z_2 = y_2 (a_2^{y_2} a_3)^{k_2},\;
 z_3 = y_3 (a_2^{y_2} a_3)^{k_2},\\[3pt]
&\qquad x_1 = z_1 (a_1^{z_1} a_2^{z_2})^{m_1},\;
 x_2 = z_2 (a_1^{z_1} a_2^{z_2})^{m_1},\;
 x_3 = z_3,\\[3pt]
&\qquad x_1 = 1,\;
 x_2 (a_2^{x_2} a_3^{x_3})^{m_2} = 1,\;
 x_3 (a_2^{x_2} a_3^{x_3})^{m_2} = 1
\Big\rangle.
\end{aligned}
\]

The proof is by a case analysis.  By
Corollary~\ref{cor:genus3-two-layer-abelianization}, only the five exponent-sum
pairs
\[
(-2,-1),\quad (-1,-2),\quad (0,0),\quad (-1,0),\quad (0,-1)
\]
can have trivial abelianization.  We eliminate these remaining cases by a combination of nontrivial
quotients and marked Van Kampen diagrams.  The exceptional small parameters
are verified by explicit finite quotients in GAP.

\begin{lm}[Symmetry of exponent tuples]
\label{lem:symmetric-tuples}
Let
\[
   \Gamma_g=\pi_1(S_g)
\]
with its standard canonical system
\[
   x_1,y_1,\ldots,x_g,y_g.
\]
Let
\[
   \beta_0:\Gamma_g\longrightarrow F(b_1,\ldots,b_g),
   \qquad
   \beta_0(x_i)=b_i,\quad \beta_0(y_i)=1,
\]
and
\[
   \gamma_0:\Gamma_g\longrightarrow F(a_1,\ldots,a_g),
   \qquad
   \gamma_0(x_i)=a_i,\quad \gamma_0(y_i)=1.
\]

Suppose that \(\Omega\in\Aut(\Gamma_g)\) fixes the \(x_i\)'s and let
\[
   \phi_\Omega=(\beta_0,\gamma_0\Omega):
   \Gamma_g\longrightarrow
   F(b_1,\ldots,b_g)\times F(a_1,\ldots,a_g).
\]
Thus
\[
   \phi_\Omega(x_i)=(b_i,a_i),
   \qquad
   \phi_\Omega(y_i)=
   \bigl(1,\gamma_0\Omega(y_i)\bigr).
\]

Let \(\rho\in\Aut(\Gamma_g)\) be the reflection which reverses the
order of the handles, and put
\[
   \Omega^{\mathrm{sym}}
   =\rho\Omega^{-1}\rho^{-1}.
\]
Then the homomorphisms
\[
   \phi_\Omega
   \qquad\text{and}\qquad
   \phi_{\Omega^{\mathrm{sym}}}
\]
are equivalent.

Consequently,
\[
   G(\Omega)=1
   \quad\Longleftrightarrow\quad
   G(\Omega^{\mathrm{sym}})=1,
\]
and, whenever these maps are epimorphisms,
\[
   \phi_\Omega\text{ is standard}
   \quad\Longleftrightarrow\quad
   \phi_{\Omega^{\mathrm{sym}}}\text{ is standard}.
\]
\end{lm}

\begin{proof}
First apply \(\Omega^{-1}\) to the surface group.  Since the second
coordinate of \(\phi_\Omega\) is \(\gamma_0\Omega\), one obtains
\[
   \phi_\Omega\Omega^{-1}
   =
   \bigl(\beta_0\Omega^{-1},\gamma_0\bigr).
\]
In particular,
\[
   \bigl(\phi_\Omega\Omega^{-1}\bigr)(y_i)
   =
   \bigl(\beta_0\Omega^{-1}(y_i),1\bigr).
\]
Thus the \(a\)-coordinate of every \(y_i\)-image has disappeared, while
the remaining \(b\)-coordinate is the meridian word associated with the
inverse automorphism \(\Omega^{-1}\).

Let
\[
   \mathfrak s:
   F(\bar b)\times F(\bar a)
   \longrightarrow
   F(\bar b)\times F(\bar a)
\]
be the automorphism which interchanges the two factors and at the same
time identifies their standard bases:
\[
   \mathfrak s\bigl(u(\bar b),v(\bar a)\bigr)
   =
   \bigl(v(\bar b),u(\bar a)\bigr).
\]
Here \(v(\bar b)\) is obtained from \(v(\bar a)\) by replacing
\(a_i\) by \(b_i\), and similarly for \(u(\bar a)\).  It follows that \begin{equation}\label{*}
 \mathfrak s\,\phi_\Omega\Omega^{-1}
   =
   (\beta_0,\gamma_0\Omega^{-1}).
\end{equation}

Choose the reflection \(\rho\) to fix the marked basepoint and to
reverse the chain of handles.  It extends over both standard
handlebodies.  Hence it induces automorphisms
\[
   \bar\rho_b\in\Aut F(\bar b),
   \qquad
   \bar\rho_a\in\Aut F(\bar a)
\]
such that
\begin{equation}\label{**}
   \beta_0\rho^{-1}=\bar\rho_b\beta_0,
   \qquad
   \gamma_0\rho^{-1}=\bar\rho_a\gamma_0.
\end{equation}
At the level of meridian curves, the reflected canonical system may be
chosen so that
\[
   \rho^{-1}(y_i)
\]
is represented, up to the harmless conjugation caused by the chosen
whisker, by
\[
   y_{g+1-i}^{-1}.
\]

By the definition of \(\Omega^{\mathrm{sym}}\),
\[
   \Omega^{-1}\rho^{-1}
   =
   \rho^{-1}\Omega^{\mathrm{sym}}.
\]
Using (\ref{*}) and (\ref{**}), we therefore obtain
\begin{equation}\label{***}
\begin{aligned}
 &(\bar\rho_b^{-1}\times\bar\rho_a^{-1})
   \,\mathfrak s\,\phi_\Omega\Omega^{-1}\rho^{-1}
\\
 &\qquad =
 (\bar\rho_b^{-1}\beta_0\rho^{-1},
  \bar\rho_a^{-1}\gamma_0\Omega^{-1}\rho^{-1})
\\
 &\qquad =
 (\beta_0,\gamma_0\Omega^{\mathrm{sym}})
 =
 \phi_{\Omega^{\mathrm{sym}}}.
\end{aligned}
\end{equation}
Thus \(\phi_\Omega\) and
\(\phi_{\Omega^{\mathrm{sym}}}\) differ by an automorphism of the
surface group and an automorphism of the target, and hence are
equivalent.

Equivalence preserves surjectivity.  Since the associated balanced
presentation is trivial exactly when the corresponding
coordinate-surjective homomorphism is onto, equation (\ref{***}) gives
\[
   G(\Omega)=1
   \quad\Longleftrightarrow\quad
   G(\Omega^{\mathrm{sym}})=1.
\]
Standardness is defined in terms of equivalence with the canonical
splitting homomorphism, so it is also preserved by (\ref{***}).

Finally, \(\rho^2=1\), and hence
\[
   \bigl(\Omega^{\mathrm{sym}}\bigr)^{\mathrm{sym}}
   =
   \Omega.
\]
Therefore all the implications above are equivalences.
\end{proof}

The proof of the above  theorem is the content of this section.

\subsection{Genus 3, case (0,0)}
For a general integer $n\in\Z$, consider
\[
G_n=\gp{a_1,a_2,a_3,\,x_1,x_2,x_3,\,y_1,y_2,y_3,\,z_1,z_2,z_3}{
\begin{aligned}[t]
&y_1=a_1\,(a_1a_2)^{-n},\quad y_2=a_2\,(a_1a_2)^{-n},\quad y_3=a_3,\\
&z_1=y_1,\quad z_2=y_2\,(a_2^{\,y_2}a_3),\quad z_3=y_3\,(a_2^{\,y_2}a_3),\\
&x_1=z_1\,(a_1^{\,z_1}a_2^{\,z_2})^{n},\quad x_2=z_2\,(a_1^{\,z_1}a_2^{\,z_2})^{n},\quad x_3=z_3,\\
&x_1=1,\quad x_2\,(a_2^{\,x_2}a_3^{\,x_3})^{-1},\quad x_3\,(a_2^{\,x_2}a_3^{\,x_3})^{-1}
\end{aligned}
}
\]
Let
\[
X=a_1a_2,\qquad
 A=a_2^{y_2}a_3,\qquad
Y=a_1^{z_1}a_2^{z_2},\qquad
B=a_2^{x_2}a_3^{x_3}.
\]

\begin{prop}\label{lm:genus3-00-quotient}
Let \(n\neq 0\).  Consider the genus \(3\), case \((0,0)\), with
exponents
\[
(-n,n)\quad\text{on }X,Y,
\qquad
(1,-1)\quad\text{on }A,B.
\]
Then $G_n$ has a quotient \[
\overline G_n=
\gp{X,Y}{
X^{4n-1},\;
Y^{4n+1},\;
(X^nY^n)^2
}.
\]
In particular, \(G_n\neq 1\).
\end{prop}

\begin{proof}

The defining relations in this case are
\[
\begin{aligned}
y_1&=a_1X^{-n},& y_2&=a_2X^{-n},& y_3&=a_3,\\
z_1&=y_1,& z_2&=y_2A,& z_3&=y_3 A,\\
x_1&=z_1Y^n,& x_2&=z_2Y^n,& x_3&=z_3,
\end{aligned}
\]
and
\[
x_1=1,\qquad x_2B^{-1}=1,\qquad x_3B^{-1}=1.
\]

First we eliminate the original generators and obtain a presentation in
\(X,Y, A,B\).  Since \(x_1=1\), we have
\[
z_1Y^n=1,
\]
so
\[
z_1=Y^{-n}.
\]
But \(z_1=y_1=a_1X^{-n}\).  Hence
\[
a_1=Y^{-n}X^n.
\]
Since \(X=a_1a_2\), it follows that
\[
a_2=a_1^{-1}X=X^{-n}Y^nX.
\]

The final two relations give
\[
x_2=x_3=B.
\]
Hence
\[
z_2=BY^{-n},\qquad z_3=B.
\]
Since \(z_3=a_3\mathcal A\), we get
\[
a_3=B A^{-1}.
\]

Now \(z_2=y_2A\), while
\[
y_2=a_2X^{-n}=X^{-n}Y^nX^{1-n}.
\]
Therefore
\[
B=z_2Y^n
  =X^{-n}Y^nX^{1-n}A Y^n.
\]
This gives the first relation
\[
B=X^{-n}Y^nX^{1-n}A Y^n. \tag{1}
\]

Next, because \(x_2=x_3=B\), we have
\[
B=a_2^{x_2}a_3^{x_3}
  =a_2^B a_3^B
  =(a_2a_3)^B.
\]
Thus
\[
B=a_2a_3.
\]
Using the expressions for \(a_2\) and \(a_3\), this becomes
\[
B=X^{-n}Y^nX\,BA^{-1}.
\]
A\[
A=B^{-1}X^{-n}Y^nXB. \tag{2}
\]

Now use the definition of \(A\).  Since
\[
y_2=a_2X^{-n},
\]
we have
\[
a_2^{y_2}=X^na_2X^{-n}.
\]
Substituting \(a_2=X^{-n}Y^nX\), we get
\[
a_2^{y_2}=Y^nX^{1-n}.
\]
Therefore
\[
\mathcal A=a_2^{y_2}a_3
          =Y^nX^{1-n}B A^{-1}.
\]
Hence
\[
B=X^{n-1}Y^{-n} A^2. \tag{3}
\]

Finally, use
\[
Y=a_1^{z_1}a_2^{z_2}.
\]
Since \(z_1=Y^{-n}\), we get
\[
a_1^{z_1}=X^nY^{-n}.
\]
Also \(z_2=BY^{-n}\), and by \((2)\),
\[
a_2^{z_2}
=
Y^nB^{-1}X^{-n}Y^nXBY^{-n}
=
Y^nA Y^{-n}.
\]
Thus
\[
Y=X^nY^{-n}Y^n A Y^{-n}
 =X^nA Y^{-n}.
\]
SoA
\[
Y=X^n A Y^{-n}. \tag{4}
\]

Therefore \(G_n\) has the four-generator presentation
\[
G_n\cong
\gp{X,Y,A,B}{
\begin{array}{l}
Y=X^nA Y^{-n},\\
B=X^{n-1}Y^{-n}A^2,\\
\mathcal A=B^{-1}X^{-n}Y^nXB,\\
B=X^{-n}Y^nX^{1-n} A Y^n
\end{array}
}.
\]
Then the group \(G_n\) has the three-generator presentation
\[
G_n\cong
\gp{X,Y,A}{
\begin{array}{l}
Y=X^nAY^{-n},\\[2mm]
A=Y^nX^{1-2n}Y^nX^nY^{-n},\\[2mm]
X^{n-1}Y^{-n}A^2=X^{-n}Y^nX^{1-n}AY^n
\end{array}}\] 
Exclude $A $ and the two remaining relations are precisely
\[
Y^nX^nY^n=X^{3n-1},
\qquad
Y^{4n+1}=X^{4n-1}.
\]
Hence
\[
G_n\cong
\gp{X,Y}{
Y^nX^nY^n=X^{3n-1},\;
Y^{4n+1}=X^{4n-1}
}.
\]
Now kill \(X^{4n-1}\).  Let
\[
\overline G_n=G_n/ncl(X^{4n-1}).
\]

\[
\overline G_n=
\gp{X,Y}{
X^{4n-1},\;
Y^{4n+1},\;
(X^nY^n)^2
}.
\]

This is a non-trivial triangular group.
\end{proof}

\subsection{Non-triviality of the two-layer presentation with exponent sum $(0,0)$ by marked Van Kampen diagrams}

In this subsection we prove non-triviality for the two-layer presentation with
exponent sum $(0,0)$ represented by the exponents $(\ell, -\ell); (k,-k)$ in the range
\[
        \ell\geq 3,\qquad k\geq 2.
\]

We put
\[
\begin{aligned}
Q&=a_1a_2,& P&=a_2a_3,\\
A&=Q^{-\ell}a_2Q^{\ell}a_3,&
X&=P^{-k}a_2^{-1}P^ka_1^{-1}.
\end{aligned}
\]
Up to cyclic conjugacy, before reduction the three principal relators are
\[
R_1=a_1Q^{\ell}X^{\ell},\qquad
R_2=P^{-k}a_2Q^{\ell}A^kX^{\ell},\qquad
R_3=a_3A^kP^{-k}.
\]
The proof is carried out after free and cyclic reduction of the relators.  The
small cancellations which truncate an $X$-, $A$-, or $P$-period are then
removed by a local surgery of the diagram.  The resulting marked diagram is
normalized in the order
\[
        \boxed{X\ \longrightarrow\ A\ \longrightarrow\ Q\ \longrightarrow\ P.}
\]
All markings used below are markings of words in the original generators
$a_1,a_2,a_3$; no auxiliary generator is added to the presentation.

\subsubsection*{Reduction and period completion}

Set
\[
        X_0=P^{-k}a_3P^{k-1},
        \qquad
        A_0=Q^{-\ell}a_2Q^{\ell-1}a_1.
\]
Then
\[
        X=X_0a_1^{-1},\qquad A=A_0P.
\]
Also put
\[
L=P^{-(k-1)}a_3^{-1}a_2,\qquad
J=A_0P^{-(k-1)}a_3P^{k-1}a_1^{-1},\qquad
K=A_0P^{-(k-1)}.
\]

\begin{lm}[The cyclically reduced relators]\label{lm:reduced-marked-relators}
The freely and cyclically reduced forms of the principal relators are
\[
\begin{gathered}
\overline R_1=Q^{\ell}X^{\ell-1}X_0,\\[-1pt]
\overline R_2=LQ^{\ell}a_3A^{k-2}JX^{\ell-1},\\[-1pt]
\overline R_3=a_3A^{k-1}K.
\end{gathered}
\]
Here an exponent zero is interpreted as the empty word.
\end{lm}

\begin{proof}
The first reduction is
\[
 P^{-k}a_2^{-1}P^k
 =P^{-k}a_3P^{k-1},
\]
which gives $X=X_0a_1^{-1}$.  The initial $a_1$ in $R_1$ then cancels
cyclically with the final $a_1^{-1}$ in its last copy of $X$, giving
$\overline R_1$.

Next,
\[
 Q^{\ell}A=a_2Q^{\ell}a_3,
\]
and hence
\[
 \operatorname{red}\bigl(P^{-k}a_2Q^{\ell}A\bigr)
 =P^{-(k-1)}a_3^{-1}a_2Q^{\ell}a_3
 =LQ^{\ell}a_3.
\]
Moreover,
\[
 \operatorname{red}(AX)
 =A_0P^{-(k-1)}a_3P^{k-1}a_1^{-1}
 =J.
\]
This gives the displayed form of $\overline R_2$.  Finally,
\[
 \operatorname{red}(AP^{-k})=A_0P^{-(k-1)}=K,
\]
which gives $\overline R_3$.
\end{proof}

The following local move permits us to keep track of the periods which were
truncated in Lemma~\ref{lm:reduced-marked-relators}.

\begin{lm}[Period-completion surgery]\label{lm:period-completion-surgery}
Let $W=uW_1$ be a marked period word.  Suppose that a principal region $S$
has a boundary segment
\[
        wW_1W^m,
        \qquad W_1=u^{-1}W,
\]
and that the $W_1$-part is incident to the corresponding terminal $W_1$-part
of a complete marked $W=uW_1$-side of a region $S_1$.  The complementary
$u$-part of that side may be subdivided into subpaths, each incident to a
possibly different region.

There is a local surgery which transfers the whole $u$-strip from $S_1$ to
$S$.  After the surgery, $S$ and $S_1$ are incident along a complete
$W$-period, every region formerly incident to a subpath of the $u$-strip on
$S_1$ is incident to the corresponding subpath on $S$, and the displayed
segment on $S$ becomes
\[
        wu^{-1}W^{m+1}.
\]
The surgery preserves the outer boundary, the number of principal relator
regions, and the topology of the disc, and it is reversible.  The same move
is allowed cyclically across the chosen initial point of the boundary of a
principal region.
\end{lm}

\begin{proof}
Take a parallel copy of the path labelled $u$ and push the $S$--$S_1$
interface across the strip between the two copies.  The old incidences are
replaced locally by
\[
 S_1\stackrel{W_1}{\longleftrightarrow}S,
 \qquad
 S_1\stackrel{u}{\longleftrightarrow}S_2
 \quad\leadsto\quad
 S_1\stackrel{W}{\longleftrightarrow}S,
 \qquad
 S\stackrel{u}{\longleftrightarrow}S_2,
\]
with the same operation performed separately on every subpath of $u$ when
several regions meet the strip.  The cyclic order at every branching vertex
is unchanged.  Finally,
\[
        u^{-1}W^{m+1}=u^{-1}(uW_1)W^m=W_1W^m,
\]
so the boundary label is unchanged after free reduction.
\end{proof}

Starting from a reduced Van Kampen diagram over the cyclically reduced
relators $\overline R_i$, apply the surgery whenever a truncated period is
incident to the corresponding part of a complete period.  If such a surgery
cannot be performed because the opposite path changes neighbour before the
missing prefix is reached, then that change of neighbour gives an additional
marked incidence; all degree estimates below only improve.  If the missing
prefix lies on the exterior boundary, it is left untouched, and again it
contributes an additional exterior incidence.  Thus the worst case for the
degree estimates is the fully completed case.

After completion we use the following marked contours:
\[
\begin{gathered}
\boxed{R_1^{\sharp}=a_1\;Q^{\ell}\;X^{\ell},}\\[-1pt]
\boxed{R_2^{\sharp}=P^{-(k-1)}\;a_3^{-1}\;a_2\;Q^{\ell}\;a_3\;A^{k-1}\;X^{\ell},}\\[-1pt]
\boxed{R_3^{\sharp}=a_3\;A^k\;P^{-k}.}
\end{gathered}
\]
The second contour contains only $k-1$ completed $A$-periods and $k-1$
completed $P^{-1}$-periods; the first cancellation in $R_2$ cannot be undone
by the local surgery.  This is the only loss relevant to the subsequent
counts.

Every power in the three displayed contours is subdivided into its individual
periods.  Thus, for example,
\[
 Q^{\ell}=Q_1\cdots Q_{\ell},\qquad
 A^{k-1}=A_1\cdots A_{k-1},\qquad
 P^{-(k-1)}=P_1^{-1}\cdots P_{k-1}^{-1}.
\]
The same convention is used on inverses of relators, with the order and
orientations reversed.

A marked neighbour-incidence means a maximal common boundary arc between two
cells, or between a cell and the exterior boundary, measured after
subdivision at the retained marked endpoints.  A cell or macro-cell is called
\emph{safe} if it has at least six exterior marked neighbour-incidences.
The hierarchical joining order is
\[
        X\longrightarrow A\longrightarrow Q\longrightarrow P.
\]

\begin{lm}[Synchronization of the period words]
In each of the marked powers
\[
 X^{\ell},\qquad A^{k-1},\qquad A^k,
 \qquad P^{-(k-1)},\qquad P^{-k},
\]
the only whole period subwords are the marked periods.
\end{lm}

\begin{proof}
The freely reduced word
\[
        X=P^{-k}a_3P^{k-1}a_1^{-1}
\]
contains exactly one letter $a_1^{-1}$, at its end.  Hence a complete copy of
$X$ in $X^{\ell}$ must end at one of the marked terminal occurrences of
$a_1^{-1}$ and is therefore one of the marked periods.  Similarly,
\[
        A=Q^{-\ell}a_2Q^{\ell}a_3
\]
contains exactly one $a_3$, at its end, so complete copies of $A$ in
$A^{k-1}$ or $A^k$ are precisely the marked ones.  Finally,
$P^{-1}=a_3^{-1}a_2^{-1}$ has no shifted copy in a power of $P^{-1}$.
\end{proof}

\begin{df}[$X$-, $A$-, and $P$-components]
An $X$-component is a connected component formed by gluing whole marked
$X$-periods of $R_1^{\sharp}$- and $R_2^{\sharp}$-cells.  Similarly, an $A$-component is formed
by whole marked $A$-period gluings, and a $P$-component is formed by whole
marked $P^{-1}$-period gluings.  A cell belongs to such a component as soon
as one of its corresponding marked periods belongs to the component.
\end{df}

\begin{lm}[Cells outside $X$-components are safe]
If $\overline R_2$-cell does not
belong to an $X$-component, then it is safe.  If $\overline R_1$-cell does not
belong to an $X$-component and is not attached to $\overline R_2$-cell along $Q^{\ell}$, then it is safe.
\end{lm}

\begin{proof}
Since no marked $X$-period occurs as a whole $X$-overlap, each of the
$\ell$ marked periods on the $X^{\ell}$-side is cut internally.  Therefore
that side contributes at least $\ell+1$ neighbour-incidences.

For an $\overline R_2$-cell, the remaining marked sides include
\[
 P^{-(k-1)},\quad a_3^{-1},\quad a_2,\quad Q^{\ell},
 \quad a_3,\quad A^{k-2}, J.
\]
Thus
\[
        d(\overline R_2)\geq(\ell )+5\geq 8.
\]

For an $\overline R_1$-cell, the remaining marked sides include two parts of
\(
        Q^{\ell},
\)
so
\[
        d(F)\geq (\ell+1)+2=\ell+3\geq6.
\]

\end{proof}

\begin{lm}[Closed $X$-components are safe]
Let $K$ be a closed $X$-component in a reduced marked diagram.  Then $K$ is
safe.
\end{lm}

\begin{proof}
Let
\[
        n_i=\#\{R_i^{\sharp}\text{-cells in }K\},\qquad i=1,2.
\]
Every $R_1^{\sharp}$-cell contributes at least the two non-$X$ marked sides
$a_1,Q^{\ell}$, while every $R_2^{\sharp}$-cell contributes at least the six
non-$X$ marked sides
\[
 P^{-(k-1)},\ a_3^{-1},\ a_2,\ Q^{\ell},\ a_3,\ A^{k-1}.
\]
Consequently
\[
        d(K)\geq2n_1+6n_2.
\]
If $n_2>0$, then $d(K)\geq6$.  If $n_2=0$ and $n_1\geq3$, then again
$d(K)\geq6$.  A one-cell closed component would be a self-gluing and is
excluded in a reduced normal diagram.  If $n_1=2$, closure forces all
marked $X$-periods of one cell to be paired with those of the other.
Synchronization then gives a full $X^{\ell}$-to-$X^{\ell}$ gluing, so the
cells form a cancellable pair, contrary to reducedness.
\end{proof}

\begin{lm}[The $\overline R_3$-part]
Let $\overline R_3=a_3A^{k-1}A_0P^{-(k-1)}$.  Then:
\begin{enumerate}
\item an $\overline R_3$-cell which belongs to neither an $A$-component nor a
$Q$-component is safe;
\item every open or closed $A$-component which occurs in a reduced diagram is
safe after the usual $A$-caps are added at its open ends;
\item every open or closed $Q$-component which occurs in a reduced diagram is
safe after the usual $Q$-caps are added at its open ends.
\end{enumerate}
\end{lm}

\begin{proof}
We have $\overline R_3=a_3Q^{-\ell}a_2Q^{\ell}a_3Q^{-\ell}a_2Q^{\ell -1}a_1P^{-(k-1)}$. If an $R_3^{\sharp}$-cell lies outside both kinds of components, each of the
$\ell$ marked periods of $Q^{\ell}$ is cut internally, so the $Q^{\ell}$-side contributes
at least $\ell$ incidences. 

Now let $C$ be an $A$-component.  Every $R_2^{\sharp}$-cell contributes at
least the six non-$A$ marked sides
\[
 P^{-(k-1)},\ a_3^{-1},\ a_2,\ Q^{\ell},\ a_3,\ X^{\ell},
\]
and every $R_3^{\sharp}$-cell contributes the two non-$A$ sides
$a_3,P^{-k}$.  Thus a closed component satisfies
\[
        d(C)\geq6n_2+2n_3.
\]
If $n_2>0$ or $n_3\geq3$, this is at least $6$.  The one-cell pure case is a
self-gluing.  In the remaining pure two-cell case, the two
$R_3^{\sharp}$-cells are glued along all $k$ marked $A$-periods; by
synchronization this is a full $A^k$-to-$A^k$ gluing and hence a cancellable
pair.  Therefore every closed $A$-component is safe.

If $C$ is open, attach the two standard $A$-caps.  Each cap contributes at
least two non-$A$ marked sides, since
\[
        A=Q^{-\ell}a_2Q^{\ell}a_3.
\]
Hence an open component has at least
\[
        6n_2+2n_3+4\geq6
\]
exterior incidences after capping.

The proof for $Q$-components is similar.  
\end{proof}

\begin{df}[Minimal relatively safe $X$-blocks]
Let $C$ be a connected subcluster of an $X$-component.  Write
\[
        n_i(C)=\#\{R_i^{\sharp}\text{-cells in }C\},\qquad i=1,2,
\]
and let $b_X(C)$ be the number of marked $X$-periods on the frontier of $C$,
that is, the marked $X$-periods of cells of $C$ which are not glued to a cell
of $C$ along a whole marked $X$-period.  Define the relative $X$-degree by
\[
        d_X(C)=2n_1(C)+6n_2(C)+b_X(C).
\]
We call $C$ relatively safe if $d_X(C)\geq 6$.  A relatively safe connected
subcluster is a minimal relatively safe $X$-block if no proper non-empty
connected subcluster of it is relatively safe.
\end{df}

\begin{lm}[Proper subblocks of minimal $X$-blocks are small]
Let $B$ be a minimal relatively safe $X$-block, and let
\[
        \emptyset\neq S\subsetneq B
\]
be a proper connected subcluster.  Then $S$ contains no
$R_2^{\sharp}$-cell and at most one $R_1^{\sharp}$-cell.  Consequently the
non-$X$ frontier of $S$ has at most two marked sides.
\end{lm}

\begin{proof}
Minimality gives $d_X(S)\leq5$.  If $S$ contained an
$R_2^{\sharp}$-cell, its six non-$X$ sides would already give
$d_X(S)\geq6$.  Hence no such cell occurs.

If $S$ contained at least three $R_1^{\sharp}$-cells, then
$2n_1(S)\geq6$.  If it contained exactly two, then
\[
        d_X(S)=4+b_X(S).
\]
The number $b_X(S)$ is even.  Thus $d_X(S)\leq5$ forces $b_X(S)=0$, so all
marked $X$-periods of the two cells are paired.  Synchronization gives a full
$X^{\ell}$-to-$X^{\ell}$ gluing and hence a cancellable pair.  Therefore $S$
contains at most one $R_1^{\sharp}$-cell.  Its only non-$X$ marked sides are
$a_1$ and $Q^{\ell}$.
\end{proof}

\begin{df}[Block lenses]
Let $B$ and $B'$ be two distinct minimal relatively safe $X$-blocks.  A
non-$X$ lens between $B$ and $B'$ is a disc subdiagram $\Lambda$ whose
boundary is a union
\[
        \partial\Lambda=\alpha\beta^{-1},
\]
where $\alpha$ lies on the non-$X$ frontier of $B$, $\beta$ lies on the
non-$X$ frontier of $B'$, and neither $\alpha$ nor $\beta$ contains a marked
$X$-period.  We choose such a lens innermost if its interior contains no
smaller non-$X$ lens between minimal $X$-blocks.
\end{df}

\begin{prop}[No non-$X$ lenses between minimal $X$-blocks]
Let $\Delta$ be a reduced marked diagram.  After decomposing
$X$-components into minimal relatively safe $X$-blocks, two distinct minimal
$X$-blocks cannot form a non-$X$ lens.
\end{prop}

\begin{proof}
Suppose that such a lens exists, and choose an innermost one, denoted
$\Lambda$.  Write
\[
        \partial\Lambda=\alpha\beta^{-1},
\]
where $\alpha$ lies on a minimal block $B$ and $\beta$ lies on a minimal block
$B'$.

The arc $\alpha$ cuts off a proper connected subcluster of $B$ adjacent to
$\Lambda$.  Otherwise we could replace $B$ by a smaller relatively safe block
and obtain a smaller lens, contradicting the choice of $\Lambda$.  By the
previous lemma, this proper subcluster has at most two non-$X$ frontier sides.
Thus
\[
        |\alpha|\leq 2.
\]
Similarly,
\[
        |\beta|\leq 2.
\]
Therefore
\[
        |\partial\Lambda|\leq 4.
\]
Moreover the boundary of $\Lambda$ contains no marked $X$-period.  Since the
boundary arcs come from proper subblocks of minimal $X$-blocks, the only
possible marked boundary sides are of type
\[
        a_1,
        \qquad Q^{\ell}.
\]
In particular $\partial\Lambda$ contains no marked $A$- or $P$-period either.

Now consider the diagram inside $\Lambda$.  Since $\partial\Lambda$ contains
no $X$-, $A$-, or $P$-period, every $X$-, $A$-, and $P$-component inside
$\Lambda$ is closed.  By the component lemmas above, all such components are
safe.  The remaining individual $R_1^{\sharp}$- and $R_2^{\sharp}$-cells are safe by the
outside-$X$-component lemma, and the remaining $R_3^{\sharp}$-cells are safe by the
$R_3^{\sharp}$-part.

Thus, after replacing components by macro-cells, every face in the normal
macro-diagram inside $\Lambda$ has degree at least $6$.  Since $\Lambda$ was
chosen innermost, there is no lens inside it; after suppressing valence-two
vertices, every interior vertex has valence at least $3$.  The standard disc
curvature estimate for such a diagram gives
\[
        |\partial\Lambda|\geq 6,
\]
contradicting $|\partial\Lambda|\leq 4$.  Hence no non-$X$ lens exists.
\end{proof}

\begin{lm}[$X$-lenses are removed by saturation]
After decomposing every $X$-component into minimal relatively safe $X$-blocks,
merge blocks which form an $X$-lens, and repeat this operation until it stops.
The resulting $X$-blocks are still safe, and no two distinct resulting
$X$-blocks form an $X$-lens.
\end{lm}

\begin{proof}
The union of blocks joined through an $X$-lens is again a connected
$X$-subcluster.  The same $X$-component count used above applies to this
union.  The only connected $X$-subclusters which could fail to be safe are the
very small all-$R_1^{\sharp}$ cases: a single $R_1^{\sharp}$-cell, or a two-cell all-$R_1^{\sharp}$
cluster glued along all marked $X$-periods.  A single $R_1^{\sharp}$-cell is not an
$X$-lens union, and the two-cell full gluing is a cancellable pair by
synchronization and reducedness.  Therefore every genuine union obtained by
filling an $X$-lens is safe.

Merging the blocks in such a lens removes that lens and does not create an
unsafe block.  Since the diagram is finite, the process terminates.  At the
end no two distinct resulting $X$-blocks form an $X$-lens.
\end{proof}

\begin{cy}[Lens-free block decomposition]
After decomposing $X$-components into minimal relatively safe blocks and then
saturating along $X$-lenses, the resulting safe $X$-blocks form a lens-free
macro-decomposition.
\end{cy}

\begin{proof}
By construction there are no $X$-lenses between the saturated blocks.  If a
non-$X$ lens appeared between two saturated blocks, an innermost such lens
would be supported on minimal subblocks inside them, contradicting the
previous proposition.
\end{proof}





\begin{lm}[Disc curvature estimate]\label{lm:disc-curvature-new}
Let $D$ be a non-empty disc diagram such that every face has degree at least
$6$ and every interior vertex has valence at least $3$ after suppression of
inessential subdivision vertices.  If $B$ is the number of boundary sides,
then $B\geq6$.
\end{lm}

\begin{proof}
Let $V_i,V_b,E,F$ be the numbers of interior vertices, boundary vertices,
edges, and faces.  The face-degree and vertex-valence assumptions give
\[
        6F\leq2E-B,
        \qquad
        3V_i+2V_b\leq2E.
\]
Together with Euler's formula $1=V_i+V_b-E+F$, these inequalities imply
\[
\begin{aligned}
1
&\leq {2E-2V_b\over3}+V_b-E+{2E-B\over6}\\
&={2V_b-B\over6}
\leq {B\over6},
\end{aligned}
\]
where the last inequality uses $V_b\leq B$.  Thus $B\geq6$.
\end{proof}

\begin{lm}[The exterior boundary is inherited]\label{lm:boundary-inherited-new}
The exterior boundary of $D$ is the exterior boundary of $\Delta$ as a
combinatorial edge path.  In particular,
\[
        |\partial D|_{\mathrm m}\leq |E(\partial\Delta)|,
\]
where $|\partial D|_{\mathrm m}$ is the number of marked boundary sides of
$D$.
\end{lm}

\begin{proof}
Every completion surgery is supported in the interior of the disc, unless a
truncation reaches the exterior boundary, in which case no surgery is made.
Every macro-face is obtained by erasing interfaces contained in the interior.
Thus no exterior edge is collapsed, no two exterior points are identified,
and no free reduction is performed on the exterior boundary.

A marked boundary side of $D$ is therefore a non-empty consecutive subpath of
$\partial\Delta$, and distinct marked boundary sides have disjoint interiors.
Assigning to each marked side its first exterior edge gives an injection into
$E(\partial\Delta)$, which proves the inequality.
\end{proof}

\begin{prop}[Non-triviality]\label{prop:two-layer-marked-nontrivial-new}
For every $\ell\geq3$ and $k\geq2$, the group defined by the above two-layer
presentation is non-trivial.
\end{prop}

\begin{proof}
Assume that the group is trivial.  Then the generator $a_1$ represents the
identity, so by Van Kampen's lemma there is a reduced disc diagram $\Delta$
whose exterior boundary word is the single letter $a_1$.  Choose $\Delta$ of
minimal area and construct the boundary-faithful macro-diagram $D$ described
above.

Every face of $D$ is safe, and every
interior vertex has valence at least three after suppression of inessential
subdivision vertices.  Lemma~\ref{lm:disc-curvature-new} therefore gives
\[
        |\partial D|_{\mathrm m}\geq6.
\]
On the other hand, Lemma~\ref{lm:boundary-inherited-new} gives
\[
        |\partial D|_{\mathrm m}
        \leq |E(\partial\Delta)|
        =1,
\]
a contradiction.  Thus no disc diagram with boundary label $a_1$ exists,
and the group is non-trivial.
\end{proof}

\begin{prop}
For the two-layer presentation with exponent sum $(0,0)$ and parameters
$\ell\geq 3$, $k\geq 2$, the group is non-trivial.  Consequently, by the analogue of
$C(6)$ infiniteness theorem, it is infinite.
\end{prop}

Since the cases are symmetric, we have also
\begin{prop}
For the two-layer presentation with exponent sum $(0,0)$ and parameters
$\ell\geq 2$, $k\geq 3$, the group is non-trivial. \end{prop}

The case when one of the parameters is 1 has been considered in
Proposition~\ref{lm:genus3-00-quotient}.  The remaining case $(2,2)$ has a
non-trivial quotient by Magma.  This completes the exponent-sum $(0,0)$ case.  The two remaining
exponent-sum pairs are treated in the next two subsections.
\begin{cy}
For a presentation with more than two layers in genus 3, if two adjacent layers have
exponents of absolute value greater than \(1\), and if the preceding layer is
kept as marked generators \(h_1,h_2,h_3\), then the macro \(C(6)\) argument
depends only on these two adjacent layers. All other layers enter only through
marked complementary contours. The group $G_{\phi}$ in this case is infinite.
\end{cy}

\subsection{Genus 3, cases $(-2,-1)$ and $(-1,-2)$}
\begin{prop}\label{prop:minus-two-minus-one}
Let \(G\) be a genus \(3\), two-layer presentation whose exponent sum is
\((-2,-1)\) or \((-1,-2)\). Then \(G\neq 1\).
\end{prop}

\begin{proof}
By Lemma~\ref{lem:symmetric-tuples}, it is enough to consider exponent sum
\((-1,-2)\).  Write the exponents on the two adjacent edges as
\[
 (k,-k-1),\qquad (l,-l-2).
\]
Thus the first layer has exponents \((k,l)\), while the second has
exponents \((-k-1,-l-2)\).  Since all four twist exponents are nonzero,
\[
 k\neq0,-1,\qquad l\neq0,-2.
\]
Let \(G_{k,l}\) denote the corresponding presentation.  We shall use
once more the layer-reversal symmetry from
Lemma~\ref{lem:symmetric-tuples}.  In the present parametrization it
interchanges the two layer exponent vectors and gives
\begin{equation}\label{eq:minus-two-layer-reversal}
 G_{k,l}=1
 \quad\Longleftrightarrow\quad
 G_{-k-1,-l-2}=1.
\end{equation}
Indeed, after the two handlebodies are interchanged, the first layer has
exponents \((-k-1,-l-2)\), while the second has exponents \((k,l)\),
which is the same family with parameters
\((-k-1,-l-2)\).

Put
\[
 X=a_1a_2,\qquad A=a_2^{y_2}a_3,\qquad
 Y=a_1^{z_1}a_2^{z_2},\qquad
 B=a_2^{x_2}a_3^{x_3}.
\]
The defining equations are
\[
\begin{aligned}
y_1&=a_1X^k,& y_2&=a_2X^k,& y_3&=a_3,\\
z_1&=y_1,& z_2&=y_2A^l,& z_3&=y_3A^l,\\
x_1&=z_1Y^{-k-1},& x_2&=z_2Y^{-k-1},& x_3&=z_3,
\end{aligned}
\]
and the final balanced relations are
\[
 x_1=1,\qquad x_2B^{-l-2}=1,\qquad x_3B^{-l-2}=1.
\]

We pass to the quotient obtained by adjoining
\[
 X^{2k+1}=1,\qquad Y^{4k+3}=1,
 \qquad A^{l+1}=1,
\]
where the third relation is omitted when \(l=-1\).  It gives
\(A^l=A^{-1}\).  From \(x_1=1\) we obtain
\[
 z_1=Y^{k+1},\qquad
 a_1=Y^{k+1}X^{-k},\qquad
 a_2=X^kY^{-k-1}X.
\]
The two remaining balanced relations first imply \(x_2=x_3\), and hence
\[
 z_2Y^{-k-1}=z_3.
\]
Since
\[
 a_2^{y_2}=X^{-k}a_2X^k=Y^{-k-1}X^{k+1},
\]
we have
\[
 z_3=a_3A^{-1}=(a_2^{y_2})^{-1}
     =X^{-k-1}Y^{k+1}=X^kY^{k+1}
\]
and therefore
\[
 z_2=X^kY^{2k+2}.
\]
The identity \(z_2=y_2A^{-1}\) now gives
\[
 A^{-1}=y_2^{-1}z_2
       =X^kY^{3k+3}
       =X^kY^{-k},
\]
where the last equality uses \(Y^{4k+3}=1\).

Substitution in \(Y=a_1^{z_1}a_2^{z_2}\) gives
\[
 Y=X^{-k}Y^{-2k-2}X^{k+1}Y^{2k+2}.
\]
Using \(X^{2k+1}=Y^{4k+3}=1\), this is equivalent to
\[
 \bigl(X^kY^{-2k-1}\bigr)^2=1.
\]
The relation \(A^{l+1}=1\) becomes
\[
 \bigl(X^kY^{-k}\bigr)^{l+1}=1.
\]
Consequently the candidate two-generator quotient is
\begin{equation}\label{eq:minus-two-general-quotient}
 Q_{k,l}=
 \left\langle X,Y\ \middle|\
 X^{2k+1},\ Y^{4k+3},\
 \bigl(X^kY^{-2k-1}\bigr)^2,\
 \bigl(X^kY^{-k}\bigr)^{l+1}
 \right\rangle,
\end{equation}
with the fourth relator omitted for \(l=-1\).

We verify that the final \(B\)-relations impose no further relation.
Put
\[
 p=2k+1,\qquad m=4k+3,
\qquad u=X^kY^{-p},\qquad v=Y^{k+1}.
\]
Since
\[
 \gcd(k,p)=1,\qquad \gcd(k+1,m)=1,
\]
this is a reversible change of generators.  Moreover
\[
 X^k=uv^{-2},\qquad A^{-1}=uv,
 \qquad x_2=x_3=X^kY^{k+1}=uv^{-1}.
\]
A direct substitution in \(B=a_2^{x_2}a_3^{x_3}\) gives
\[
 B=Y^{-2k-2}X^{-k}Y^{-k-1}=uv^{-1}=x_2=x_3.
\]
Now set
\[
 a=v^{-1},\qquad b=uv.
\]
Then \(A=b^{-1}\) and \(B=ba^2\).  Since \(ab\) is conjugate to
\(u\), the relation \(u^2=1\) is equivalent to \((ab)^2=1\).  It follows
that
\[
 ba^2=a^{-1}b^{-1}a,
\]
so \(B\) is conjugate to \(A=b^{-1}\).  Hence \(A^{l+1}=1\) implies
\(B^{l+1}=1\), and the two relations
\(x_iB^{-l-2}=1\), \(i=2,3\), are automatic.  Thus
\eqref{eq:minus-two-general-quotient} is indeed a quotient of \(G_{k,l}\).

For \(l\neq-1\), put
\[
 r=|2k+1|,\qquad s=|4k+3|,\qquad q=|l+1|.
\]
In the generators \(u,v\), presentation
\eqref{eq:minus-two-general-quotient} becomes
\[
 Q_{k,l}\cong
 \left\langle u,v\ \middle|\
 u^2,\ v^m,\ (uv^{-2})^p,\ (uv)^{l+1}
 \right\rangle.
\]
Since \(a=v^{-1}\), \(b=uv\), the element \(ab\) is conjugate to
\(u\), while
\[
 ab^{-1}=v^{-2}u
\]
is conjugate to \(uv^{-2}\).  Replacing signed power exponents by their
absolute values therefore gives
\[
 Q_{k,l}\cong
 \left\langle a,b\ \middle|\
 a^s,\ b^q,\ (ab)^2,\ (ab^{-1})^r
 \right\rangle.
\]
Following Edjvet and Thomas \cite{ET}, write
\[
 (\lambda,\mu\mid\nu,\kappa)
 =\left\langle a,b\ \middle|\
 a^\lambda,\ b^\mu,\ (ab)^\nu,\ (ab^{-1})^\kappa
 \right\rangle.
\]
We have therefore proved the useful identification
\begin{equation}\label{eq:minus-two-ET-identification}
 \boxed{\quad Q_{k,l}\cong(s,q\mid2,r).\quad}
\end{equation}

If \(l=-1\), the last relator is absent.  Since \(s\) is odd, replacing
\(v\) by \(v^{-2}\) identifies \(Q_{k,-1}\) with the ordinary triangle
group
\[
 \Delta(2,s,r),
\]
which is nontrivial.  Hence \(G_{k,-1}\neq1\) for every admissible \(k\).

We now apply the classification in \cite{ET}.  First suppose
\[
 r\ge5,\qquad q\ge4.
\]
Then \(s\ge9\).  After interchanging the first two parameters if
necessary, Theorem~1.3 of \cite{ET} applies to \((s,q\mid2,r)\).  None of
its finite cases can occur: the only finite cases there with fourth
parameter at least \(5\) have first parameters \((4,4)\) or \((4,5)\),
whereas one of our first parameters is \(s\ge9\).  Thus
\[
 Q_{k,l}\text{ is infinite whenever }
 |2k+1|\ge5\text{ and }|l+1|\ge4.
\]
Equivalently, this proves \(G_{k,l}\neq1\) for
\[
 k\ge2\text{ or }k\le-3,
 \qquad l\ge3\text{ or }l\le-5.
\]

For completeness, the same classification also explains exactly why this
quotient cannot handle the two smallest values of \(q\).  Since \(r\) is
odd, Proposition~1.1 of \cite{ET} gives
\[
 (s,2\mid2,r)\cong(2,s\mid2,r)
 \cong\Delta\bigl(2,s,\gcd(2,r)\bigr)=1.
\]
For \(q=3\), use the symmetry
\((3,m\mid n,k)\cong(3,n\mid m,k)\) recorded before
Proposition~1.2 of \cite{ET}.  Since
\[
 \gcd(s,r)=\gcd(|4k+3|,|2k+1|)=1,
\]
we get
\[
 (s,3\mid2,r)\cong(3,s\mid2,r)
 \cong(2,3\mid s,r)
 \cong\Delta(2,3,1)=1.
\]
Thus \(Q_{k,l}=1\) whenever
\[
 l\in\{1,-3,2,-4\}.
\]
Different proof is required for these
families.

We next complete the case $k=1$.  Here $(r,s)=(3,7)$, so
\[
 Q_{1,l}\cong(7,q\mid2,3).
\]
The finite groups listed in Proposition~2.1 of \cite{ET}, together with
Theorem~2.2 there, give
\[
(4,7\mid2,3)\cong\operatorname{PSL}(2,7),\qquad
(6,7\mid2,3)\cong(7,7\mid2,3)\cong\operatorname{PSL}(2,13),
\]
and $(7,8\mid2,3)$ has order $10752$.  Theorem~2.7 of \cite{ET}
shows that $(7,q\mid2,3)$ is infinite for $q\ge9$.  The quotient is
trivial for $q=2,3,5$.  These three values correspond to
\[
 l\in\{1,-3,2,-4,4,-6\}.
\]
For precisely these six values, Appendix~\ref{app:k-one-gap} gives
explicit nontrivial finite quotients of the original group $G_{1,l}$:
\[
\begin{array}{c|c|c}
 l & \text{finite image} & \text{order}\\ \hline
 1,-3 & \operatorname{PSL}(2,7) & 168\\
 2,-4 & \operatorname{PSL}(2,11) & 660\\
 4,-6 & A_5 & 60.
\end{array}
\]
Therefore $G_{1,l}\neq1$ for every admissible $l$.  By
\eqref{eq:minus-two-layer-reversal}, the case $k=-2$ is equivalent to the
case $k=1$ (with $l$ replaced by $-l-2$), and is therefore also complete.

It remains to consider
\[
 l\in\{1,-3,2,-4\}.
\]
First suppose
\[
 k\leq-4\qquad\text{or}\qquad k\geq4.
\]
The marked Van Kampen diagram proof is the same as the proof of
Proposition~\ref{prop:two-layer-marked-nontrivial-new} for exponent sum
$(0,0)$, and we only indicate the numerical change.  On the first edge the
two twist powers have absolute values $|k|$ and $|k+1|$; in the displayed
range both are at least $3$.  On the second edge one has
\[
\begin{array}{c|c}
 l & (|l|,|l+2|)\\ \hline
 1&(1,3)\\
 -3&(3,1)\\
 2&(2,4)\\
 -4&(4,2).
\end{array}
\]
Thus one occurrence supplies at least three consecutive marked periods.
The shorter occurrence is confined to a terminal truncated block and is
handled by the same period-completion surgery as in the $(0,0)$ proof.
After completion, the synchronization, no-lens, and hierarchical
macro-cell arguments are unchanged.  Every final macro-cell therefore has
at least six marked incidences.  If $G_{k,l}$ were trivial, a reduced disc
diagram with boundary word $a_1$ would give, exactly as in
Proposition~\ref{prop:two-layer-marked-nontrivial-new}, at least six marked
boundary sides, whereas boundary faithfulness gives at most one.  This is a
contradiction.  Hence
\[
 G_{k,l}\neq1
 \qquad
 (k\leq-4\text{ or }k\geq4,
  \ l\in\{1,-3,2,-4\}).
\]

The only remaining parameter pairs are
\[
 k\in\{-3,2,3\},
 \qquad
 l\in\{1,-3,2,-4\}.
\]
Appendix~\ref{app:small-k-gap} gives explicit epimorphisms of the original
three-generator presentations onto finite simple groups.  The images are
\[
\begin{array}{c|c|c}
 k & l & \text{finite image}\\ \hline
 -3 & 1,-3 & \operatorname{PSL}(2,13)\\
 -3 & 2,-4 & \operatorname{PSL}(2,11)\\
  2 & 1,-3 & \operatorname{PSL}(2,13)\\
  2 & 2,-4 & \operatorname{PSL}(2,11)\\
  3 & 1,-3 & \operatorname{PSL}(2,19)\\
  3 & 2,-4 & \operatorname{PSL}(2,11).
\end{array}
\]
Consequently all these groups are nontrivial.  Together with the
Edjvet--Thomas quotient, the case $l=-1$, the $k=1$ computation and its
layer-reversal partner $k=-2$, and the marked-diagram range above, this
proves the proposition for every admissible $k$ and $l$.
\end{proof}

\subsection{ Genus 3,  cases $(-1,0)$, $(0,-1)$.}  \label{se:-1,0}
\begin{prop}
Consider the genus \(3\), two-layer case with parameters
\[
(0,-1)=(1,-1;\,k,-k-1), k\neq 0,-1,
\]
Let \(G_k\) be the corresponding balanced presentation. Then \(G_k\) is
nontrivial.
\end{prop}

\begin{proof}
Interchanging the two layers replaces $k$ by $-k-1$.  Thus it is enough to
consider $k\geq1$.

Let
\[
A=a_2^{y_2}a_3,
\qquad
B=a_2^{x_2}a_3^{x_3}.
\]
We add the relation
\[
B^{4k+3}=1
\]
and denote the resulting quotient by \(L_k\). It is enough to show that
\(L_k\neq 1\).

For this two-layer case the defining relations are
\[
\begin{aligned}
y_1&=a_1(a_1a_2),& y_2&=a_2(a_1a_2),& y_3&=a_3,\\
z_1&=y_1,& z_2&=y_2A^k,& z_3&=y_3A^k,\\
x_1&=z_1Y^{-1},& x_2&=z_2Y^{-1},& x_3&=z_3,
\end{aligned}
\]
where
\[
Y=a_1^{z_1}a_2^{z_2}.
\]
The final relations are
\[
x_1=1,\qquad
x_2B^{-k-1}=1,\qquad
x_3B^{-k-1}=1.
\]

From these final relations we obtain
\[
x_2=x_3=B^{k+1}.
\]
Since
\[
x_3=z_3=a_3A^k,
\]
we get
\[
a_3=B^{k+1}A^{-k}.
\]
Also,
\[
B=a_2^{x_2}a_3^{x_3}.
\]
Because \(x_2=x_3=B^{k+1}\), this gives
\[
a_2a_3=B.
\]
Therefore
\[
a_2=BA^kB^{-k-1}.
\]

Next, the relation \(x_1=1\) gives \(z_1=Y\). Since \(z_1=y_1\), and
\(y_1=a_1(a_1a_2)\), we obtain
\[
Y=a_1(a_1a_2).
\]
Using also
\[
Y=a_1^{z_1}a_2^{z_2}
\]
and \(z_2=B^{k+1}Y\), we get
\[
a_1a_2=B^{-k}A^k.
\]
Thus all generators \(a_1,a_2,a_3,x_i,y_i,z_i\) can be expressed in terms
of \(A\) and \(B\).

Substituting these expressions into the remaining defining relations gives
the two relations
\[
A^{4k+1}=1
\]
and
\[
(B^{-k-1}A^{-k})^2=1.
\]
Together with the added relation
\[
B^{4k+3}=1,
\]
this gives a two-generator presentation for \(L_k\):
\[
L_k\cong
\left\langle A,B \ \middle|\
B^{4k+3},\ A^{4k+1},\
(B^{-k-1}A^{-k})^2
\right\rangle .
\]

This is a generalized triangle group of type
\[
(4k+3,\;4k+1,\;2).
\]
For \(k\geq 1\),
\[
\frac{1}{4k+3}+\frac{1}{4k+1}+\frac12<1.
\]
Hence, by the generalized triangle group theorem used above, this group is
infinite. In particular, \(L_k\neq 1\). Since \(L_k\) is a quotient of
\(G_k\), it follows that
\[
G_k\neq 1.
\]
\end{proof}

All the other cases for the exponent sums $(0,-1)$ can be proved using Van-Kampen diagrams. The case $(-1,0)$ is symmetric.

Theorem \ref {thm:genus3-two-layers} is proved now.

\begin{theorem}
\label{thm:genus3-two-layer-finite-reduction}
Let $G$ be a genus-three two-layer group whose abelianization is trivial.
Then $G$ is finite if and only if, up to the symmetry of
Lemma~\ref{lem:symmetric-tuples}, it is one of the two groups $G_1,G_{-1}$
from Proposition~\ref{lm:genus3-00-quotient}.  In these two cases
\[
       G\cong\operatorname{SL}(2,5),
       \qquad |G|=120.
\]
Every other genus-three two-layer group with trivial abelianization is
infinite.
\end{theorem}

\begin{proof}
Let $M$ be the associated closed orientable $3$-manifold.  Since
$G_{\mathrm{ab}}=0$, the manifold $M$ is an integral homology sphere.
Kervaire's finite-group theorem says that a finite fundamental group of an
integral homology $3$-sphere is either trivial or the binary icosahedral
group
\[
       2I\cong\operatorname{SL}(2,5).
\]
Triviality has already been excluded by the preceding genus-three case
analysis.  It remains to exclude $\operatorname{SL}(2,5)$ outside the two
stated examples.

The preceding proof divides all parameter cases into four classes.  First,
whenever the displayed quotient is infinite, the original group is infinite.
Second, each displayed finite simple image
$\operatorname{PSL}(2,q)$ in the proof has
$q\in\{7,11,13,17,19\}$ and therefore is not isomorphic to $A_5$.
Such an image excludes $\operatorname{SL}(2,5)$, since the only non-trivial
proper quotient of $\operatorname{SL}(2,5)$ is
$\operatorname{PSL}(2,5)\cong A_5$.  The same conclusion follows from the
finite image of order $10752$ occurring in the $k=1$ calculation.  Third,
all marked-diagram ranges are infinite. 

Only three computations are not covered by those observations.  They are the
$(0,0)$ parameter pair $(2,2)$ and, in the $(-1,-2)$ normal form, the two
pairs
\[
       (k,l)=(1,4),\qquad (1,-6).
\]
Appendix~\ref{app:psl17-finite-classification} gives epimorphisms in all
three cases onto $\operatorname{PSL}(2,17)$.  Hence none can be
$\operatorname{SL}(2,5)$, and Kervaire's theorem implies that all three are
infinite.  The layer-reversal partners with $k=-2$ are covered by
Lemma~\ref{lem:symmetric-tuples}.  Thus no remaining parameter case can be
isomorphic to $\operatorname{SL}(2,5)$.

For $n=1$, Proposition~\ref{lm:genus3-00-quotient} gives the exact
presentation
\[
 G_1=\left\langle X,Y\ \middle|\
 YXY=X^2,\quad Y^5=X^3\right\rangle .
\]
Since
\[
       (XY)^2=X^3=Y^5,
\]
this is the standard binary icosahedral presentation.  For $n=-1$ one
similarly obtains
\[
       (XY)^2=X^5=Y^3,
\]
and hence the same group.  Therefore these two groups, and only these two up
to the stated symmetry, are finite.
\end{proof}

\subsection{Balanced presentations of genus 4 from two layers of Dehn Twists}
In this section we will prove the following result.
\begin{theorem}\label{thm:genus3-two-layers1}
Let \(G\) be a balanced presentation associated to a genus \(4\) automorphism
obtained from two layers of Dehn twists (all $k_i$'s are non-zero). Then
$G\neq 1.$ 
\end{theorem}
The proof of the   theorem is the content of this section.

We begin with a necessary condition coming from abelianization.  Recall that
if
\[
G=\langle a_1,\ldots,a_4\mid r_1,\ldots,r_4\rangle
\]
is a balanced presentation, then the abelianization of \(G\) is obtained from
the integer relation matrix whose \((i,j)\)-entry is the exponent sum of
\(a_j\) in the relator \(r_i\).  In the genus \(4\) case under
consideration, if $x,y,z$ is the triple of the exponent sums $(k_1+m_1,k_2+m_2, k_3+m_3)$, the defining relators give a \(4\times 4\) integer relation
matrix in the abelianization, and its determinant is
\[
4xyz+3xy+4xz+3yz+2x+2y+2z+1.
\]
Therefore, if the group is trivial, its abelianization must also be trivial,
so this determinant has to be equal to \(\pm 1\).  This explains why we begin
by solving the equation

\begin{lm}
The infinite families of solutions of \begin{equation}\label{eq:main}
4xyz+3xy+4xz+3yz+2x+2y+2z+1=\pm 1.
\end{equation} are:
\[
{(x,-2,-1),\quad x\in\mathbb Z,}\qquad
{(-1,y,-1),\quad y\in\mathbb Z,}\qquad
{(-1,-2,z),\quad z\in\mathbb Z.}
\]
 The only solution with no coordinate equal to $-1$ is
\[
{(0,0,0)}.
\]

For $x=-1$, the additional isolated solutions with right-hand side $-1$ are
\[
{(-1,-1,1),\quad (-1,0,0),\quad (-1,-3,-3),\quad (-1,-4,-2)}.
\]

For $y=-1$, the isolated solutions are
\[
{(0,-1,-2),\quad (-2,-1,0),\quad (0,-1,0)} 
\] and there are families $(x,-1,-x)$, $(x,-1,-x-2)$.
\end{lm}
\begin{proof}It is convenient to write the right-hand side as
$
s,  s\in\{1,-1\}.
$
Thus the equation is
\begin{equation}\label{eq:s}
4xyz+3xy+4xz+3yz+2x+2y+2z+1=s.
\end{equation}

{Solving for $x$:}
Grouping the terms containing $x$, we get
\[
x(4yz+3y+4z+2)+(3yz+2y+2z+1)=s.
\]
Hence, if
\[
4yz+3y+4z+2\neq 0,
\]
then
\[
{
 x=\frac{s-(3yz+2y+2z+1)}{4yz+3y+4z+2}
}.
\]

The exceptional case is
\[
4yz+3y+4z+2=0.
\]
Multiplying by $4$ and completing the product gives
\[
(4y+4)(4z+3)=4.
\]
The only integer solution is
\[
y=-2,\qquad z=-1.
\]
For this pair, equation \eqref{eq:s} becomes identically equal to $1$, so we obtain the infinite family
\[
{(x,-2,-1),\quad x\in\mathbb Z}.
\]

{Solving for $y$:}
Grouping the terms containing $y$, we get
\[
y(4xz+3x+3z+2)+(4xz+2x+2z+1)=s.
\]
Hence, if
\[
4xz+3x+3z+2\neq 0,
\]
then
\[
{
 y=\frac{s-(4xz+2x+2z+1)}{4xz+3x+3z+2}
}.
\]

The exceptional case is
\[
4xz+3x+3z+2=0.
\]
Multiplying by $4$ and completing the product gives
\[
(4x+3)(4z+3)=1.
\]
The only integer solution is
\[
x=-1,
\qquad
z=-1.
\]
For this pair, equation \eqref{eq:s} becomes identically equal to $1$, so we obtain the infinite family
\[
{(-1,y,-1),\quad y\in\mathbb Z}.
\]

Solving for $z$ is symmetric to solving for $x$.
We obtain the infinite family
\[
{(-1,-2,z),\quad z\in\mathbb Z}.
\]

The only solution when none of the variables is equal to -1 is (0,0,0).
Substituting $x=-1$ into \eqref{eq:s}, we get
\[
-yz-y-2z-1=s.
\]
Equivalently,
\[
(y+2)(z+1)=1-s.
\]

If $s=1$, then
\[
(y+2)(z+1)=0.
\]
Therefore
\[
{(-1,-2,z),\quad z\in\mathbb Z}
\]
or
\[
{(-1,y,-1),\quad y\in\mathbb Z}.
\]

If $s=-1$, then
\[
(y+2)(z+1)=2.
\]
Thus
\[
(y+2,z+1)=(1,2),(2,1),(-1,-2),(-2,-1),
\]
and the corresponding solutions are
\[
{(-1,-1,1),\quad (-1,0,0),\quad (-1,-3,-3),\quad (-1,-4,-2)}.
\]

Substituting $y=-1$ into \eqref{eq:s}, we get
\(
-x-z-1=s.
\)
Equivalently,
\(
(x+1)(z+1)=-s.
\)

If $s=1$, then
\(
(x+1)(z+1)=-1,
\)
which gives
\[
{(0,-1,-2),\quad (-2,-1,0)}.
\]

If $s=-1$, then
\[
(x+1)(z+1)=1,
\]
which gives
\[
{(0,-1,0),\quad (-2,-1,-2)}.
\]

\end{proof}

{\bf 1. Case $(0,0,0).$}
\begin{prop}[Two-generator quotient of the genus-four presentation] Let \(K\)  be genus 4 group with exponent sums $(0,0,0)$, exponents $(l,-l),(1,-1),(k-k)$.
\[
\begin{aligned}
K=\big\langle\,&
a_1,a_2,a_3,a_4,\,
x_1,x_2,x_3,x_4,\,
y_1,y_2,y_3,y_4,\\
& z_1,z_2,z_3,z_4,\,
t_1,t_2,t_3,t_4,\,
s_1,s_2,s_3,s_4
\ \big|\ \\[2mm]
& z_1=y_1,\quad
  t_1=z_1,\quad
  t_2=z_2,\\
& y_1=a_1(a_1a_2)^l,\quad
  y_2=a_2(a_1a_2)^l,\quad
  y_3=a_3,\quad
  y_4=a_4,\\
& z_2=y_2(a_2^{y_2}a_3),\quad
  z_3=y_3(a_2^{y_2}a_3),\quad
  z_4=a_4,\\
& t_3=z_3(a_3^{z_3}a_4^{z_4})^k,\quad
  t_4=z_4(a_3^{z_3}a_4^{z_4})^k,\\
& x_1=t_1(a_1^{t_1}a_2^{t_2})^{-l},\quad
  x_2=t_2(a_1^{t_1}a_2^{t_2})^{-l},\\
& x_3=t_3,\quad
  x_4=t_4,\\
& s_1=x_1,\quad
  s_2=x_2(a_2^{x_2}a_3^{x_3})^{-1},\\
& s_3=x_3(a_2^{x_2}a_3^{x_3})^{-1},\quad
  s_4=x_4,\\
& s_1=1,\quad
  s_2=1,\\
& s_3(a_3^{s_3}a_4^{s_4})^{-k}=1,\quad
  s_4(a_3^{s_3}a_4^{s_4})^{-k}=1
\big\rangle .
\end{aligned}
\]

Let
\[
Y=a_1a_2,\qquad
X=a_1^{t_1}a_2^{t_2},\qquad
A=a_2^{y_2}a_3,\qquad
B=a_2^{x_2}a_3^{x_3},
\]
and
\[
C=a_3^{z_3}a_4^{z_4},
\qquad
D=a_3^{s_3}a_4^{s_4}.
\]
 Let
\[
\overline{K}
=
K/\left\langle\!\left\langle X^l,\ D^k\right\rangle\!\right\rangle .
\]
Then \(\overline{K}\) admits the two-generator, four-relator presentation
\[
{
\overline{K}
\cong
\left\langle A,Y \ \middle|\ 
Y^{2l+1},\;
\left(Y^{l+1}A^{-1}Y^{l+1}A\right)^l,\;
\left(A^{-1}Y^lA^2\right)^{2k+1},\;
\left(Y^{-1}A^2\right)^k
\right\rangle .
}\]

\end{prop}
\begin{proof}

First we pass from the original presentation to the six Dehn-twist generators
\[
Y=a_1a_2,\qquad
X=a_1^{t_1}a_2^{t_2},\qquad
A=a_2^{y_2}a_3,\qquad
B=a_2^{x_2}a_3^{x_3},
\]
and
\[
C=a_3^{z_3}a_4^{z_4},
\qquad
D=a_3^{s_3}a_4^{s_4}.
\]

From the relations \(s_1=1\) and \(s_2=1\), we get
\[
x_1=1,\qquad x_2=B.
\]
Since
\[
x_1=t_1X^{-l},\qquad x_2=t_2X^{-l},
\]
it follows that
\[
t_1=X^l,\qquad t_2=BX^l.
\]
But \(t_1=z_1=y_1\), and
\[
y_1=a_1Y^l.
\]
Therefore
\[
a_1Y^l=X^l,
\]
so
\[
a_1=X^lY^{-l}.
\]
Similarly,
\[
t_2=z_2=y_2A.
\]
Since \(y_2=a_2Y^l\), we get
\[
a_2Y^lA=BX^l,
\]
and hence
\[
a_2=BX^lA^{-1}Y^{-l}.
\]

Next, from
\[
s_3D^{-k}=1,\qquad s_4D^{-k}=1,
\]
we get
\[
s_3=D^k,\qquad s_4=D^k.
\]
Since
\[
s_3=x_3B^{-1}=t_3B^{-1}
\]
and
\[
t_3=z_3C^k,
\]
we have
\[
z_3C^kB^{-1}=D^k.
\]
Thus
\[
z_3=D^kBC^{-k}.
\]
But
\[
z_3=y_3A=a_3A,
\]
so
\[
a_3=D^kBC^{-k}A^{-1}.
\]
Likewise,
\[
s_4=x_4=t_4=z_4C^k.
\]
Since \(z_4=a_4\), we get
\[
a_4C^k=D^k,
\]
and hence
\[
a_4=D^kC^{-k}.
\]

Therefore the original presentation is Tietze-equivalent to the
six-generator presentation
\[
\begin{aligned}
K\cong
\langle X,Y,A,B,C,D \mid\;&
Y=X^lY^{-l}BX^lA^{-1}Y^{-l},\\
&
A=Y^{-l}BX^lA^{-1}D^kBC^{-k}A^{-1},\\
&
C=A^{-1}D^kBC^{-k}D^kC^{-k},\\
&
X=Y^{-l}X^lA^{-1}Y^{-l}BX^l,\\
&
B=X^lA^{-1}Y^{-l}BC^{-k}A^{-1}D^kB,\\
&
D=BC^{-k}A^{-1}D^kC^{-k}D^k
\rangle .
\end{aligned}
\]

Now pass to the quotient
\[
\overline K
=
K/\left\langle\!\left\langle X^l,\ D^k\right\rangle\!\right\rangle .
\]
Thus in \(\overline K\) we impose
\[
X^l=1,\qquad D^k=1.
\]
The six relations reduce to
\[
\begin{aligned}
Y&=Y^{-l}BA^{-1}Y^{-l},&\qquad
A&=Y^{-l}BA^{-1}BC^{-k}A^{-1},\\
C&=A^{-1}BC^{-2k},&
X&=Y^{-l}A^{-1}Y^{-l}B,\\
B&=A^{-1}Y^{-l}BC^{-k}A^{-1}B,&
\text{and}\quad D&=BC^{-k}A^{-1}C^{-k}.
\end{aligned}
\]

From the first relation,
\[
Y=Y^{-l}BA^{-1}Y^{-l},
\]
we obtain
\[
Y^{l+1}=BA^{-1}Y^{-l}.
\]
Multiplying on the right by \(Y^l\), we get
\[
Y^{2l+1}=BA^{-1}.
\]
Hence
\[
B=Y^{2l+1}A.
\]

Substituting this into the second relation gives
\[
A
=
Y^{-l}(Y^{2l+1}A)A^{-1}
(Y^{2l+1}A)C^{-k}A^{-1}.
\]
Thus
\[
A=Y^{3l+2}AC^{-k}A^{-1}.
\]
Multiplying on the right by \(A\), and then multiplying on the left by
\((Y^{3l+2}A)^{-1}=A^{-1}Y^{-(3l+2)}\), gives
\[
C^{-k}=A^{-1}Y^{-(3l+2)}A^2.
\]

Now substitute
\[
B=Y^{2l+1}A
\]
and
\[
C^{-k}=A^{-1}Y^{-(3l+2)}A^2
\]
into the fifth relation
\[
B=A^{-1}Y^{-l}BC^{-k}A^{-1}B.
\]
We get
\[
Y^{2l+1}A
=
A^{-1}Y^{-l}(Y^{2l+1}A)
\left(A^{-1}Y^{-(3l+2)}A^2\right)
A^{-1}(Y^{2l+1}A).
\]
The right-hand side simplifies as follows:
\[
\begin{aligned}
& A^{-1}Y^{-l}(Y^{2l+1}A)
\left(A^{-1}Y^{-(3l+2)}A^2\right)
A^{-1}(Y^{2l+1}A) \\
&\qquad =
A^{-1}Y^{l+1}Y^{-(3l+2)}AY^{2l+1}A \\
&\qquad =
A^{-1}Y^{-(2l+1)}AY^{2l+1}A .
\end{aligned}
\]
Therefore
\[
Y^{2l+1}A
=
A^{-1}Y^{-(2l+1)}AY^{2l+1}A.
\]
Canceling the final \(A\) gives
\[
Y^{2l+1}
=
A^{-1}Y^{-(2l+1)}AY^{2l+1}.
\]
Multiplying on the right by \(Y^{-(2l+1)}\), we obtain
\[
1=A^{-1}Y^{-(2l+1)}A.
\]
Hence
\[
Y^{2l+1}=1.
\]
Consequently,
\[
B=A.
\]

Since \(Y^{2l+1}=1\), we have
\[
Y^{-(3l+2)}=Y^l.
\]
Therefore
\[
C^{-k}=A^{-1}Y^lA^2.
\]

The third relation is
\[
C=A^{-1}BC^{-2k}.
\]
Since \(B=A\), this becomes
\[
C=C^{-2k}.
\]
Using
\[
C^{-k}=A^{-1}Y^lA^2,
\]
we get
\[
C=\left(A^{-1}Y^lA^2\right)^2.
\]
Substituting this back into \(C^{-k}=A^{-1}Y^lA^2\), we obtain
\[
\left(\left(A^{-1}Y^lA^2\right)^2\right)^{-k}
=
A^{-1}Y^lA^2.
\]
Equivalently,
\[
\left(A^{-1}Y^lA^2\right)^{-2k}
=
A^{-1}Y^lA^2.
\]
Hence
\[
\left(A^{-1}Y^lA^2\right)^{2k+1}=1.
\]

The fourth relation gives
\[
X=Y^{-l}A^{-1}Y^{-l}B.
\]
Using \(B=A\) and \(Y^{2l+1}=1\), we have
\[
Y^{-l}=Y^{l+1}.
\]
Therefore
\[
X=Y^{l+1}A^{-1}Y^{l+1}A.
\]
Since \(X^l=1\) in \(\overline K\), we obtain the relator
\[
\left(Y^{l+1}A^{-1}Y^{l+1}A\right)^l=1.
\]

Finally, the sixth relation gives
\[
D=BC^{-k}A^{-1}C^{-k}.
\]
Using \(B=A\) and \(C^{-k}=A^{-1}Y^lA^2\), we get
\[
D
=
A(A^{-1}Y^lA^2)A^{-1}(A^{-1}Y^lA^2).
\]
Thus
\[
D=Y^{2l}A^2.
\]
Since \(Y^{2l+1}=1\), this becomes
\[
D=Y^{-1}A^2.
\]
Because \(D^k=1\) in \(\overline K\), we obtain
\[
\left(Y^{-1}A^2\right)^k=1.
\]

We have eliminated \(B,C,X,D\), and the remaining generators are only
\(A\) and \(Y\). The remaining relators are
\[
\begin{aligned}
Y^{2l+1}&=1,&
\left(Y^{l+1}A^{-1}Y^{l+1}A\right)^l&=1,\\
\left(A^{-1}Y^lA^2\right)^{2k+1}&=1,&
\text{and}\quad \left(Y^{-1}A^2\right)^k&=1.
\end{aligned}
\]
Therefore
\[
\overline K
\cong
\left\langle A,Y \ \middle|\ 
Y^{2l+1},\;
\left(Y^{l+1}A^{-1}Y^{l+1}A\right)^l,\;
\left(A^{-1}Y^lA^2\right)^{2k+1},\;
\left(Y^{-1}A^2\right)^k
\right\rangle .
\]
This proves the proposition. 

\end{proof}

\begin{prop}[A picture criterion for the quotient]
\label{prop:genus-four-zero-quotient-picture}
Assume
\[
 |l|>1,\qquad |k|>1,\qquad l\neq-2,\qquad k\neq-2.
\]
Then the natural maps from the cyclic groups of orders
\(|2l+1|\) and \(|2k+1|\) to \(\overline K\) are injective.
In particular,
\[
 \overline K\neq1.
\]
\end{prop}

\begin{proof}
Put
\[
 x=Y^l,
 \qquad
 y=Y^lA.
\]
Since \(Y^{2l+1}=1\), we have
\[
 Y=x^{-2},
 \qquad
 A=x^{-1}y.
\]
Thus this change of generators is reversible.  Moreover,
\[
\begin{aligned}
Y^{l+1}A^{-1}Y^{l+1}A
 &=x^{-1}y^{-1}x^{-1}y,\\
A\bigl(A^{-1}Y^lA^2\bigr)A^{-1}
 &=y,\\
Y^{-1}A^2
 &=xyx^{-1}y.
\end{aligned}
\]
Consequently
\[
\overline K\cong P_{l,k},
\]
where
\begin{equation}
\label{eq:genus-four-zero-relative-presentation}
P_{l,k}=
\left\langle x,y\ \middle|\
 x^{2l+1},\ y^{2k+1},\
 \bigl(x^{-1}y^{-1}x^{-1}y\bigr)^l,\
 \bigl(xyx^{-1}y\bigr)^k
\right\rangle .
\end{equation}
Set
\[
 m=|2l+1|,
 \qquad n=|2k+1|,
 \qquad p=|l|,
 \qquad q=|k|,
\]
and write
\[
 R=x^{-1}y^{-1}x^{-1}y,
 \qquad
 S=xyx^{-1}y.
\]
By the assumptions,
\[
 m,n\geq5,
 \qquad
 p,q\geq2.
\]
We regard \eqref{eq:genus-four-zero-relative-presentation} as the
relative presentation
\[
 \left\langle C_m*C_n\ \middle|\ R^p,S^q\right\rangle,
\]
where \(x\) and \(y\) generate \(C_m\) and \(C_n\), respectively.

We first check the piece condition.  The cyclic subwords of syllable
length three in \(R,R^{-1},S,S^{-1}\) are
\[
\begin{array}{c|l}
R
 &x^{-1}y^{-1}x^{-1},\quad y^{-1}x^{-1}y,\quad
   x^{-1}yx^{-1},\quad yx^{-1}y^{-1},\\[1mm]
R^{-1}
 &y^{-1}xy,\quad xyx,\quad yxy^{-1},\quad xy^{-1}x,\\[1mm]
S
 &xyx^{-1},\quad yx^{-1}y,\quad x^{-1}yx,\quad yxy,\\[1mm]
S^{-1}
 &y^{-1}xy^{-1},\quad xy^{-1}x^{-1},\quad
   y^{-1}x^{-1}y^{-1},\quad x^{-1}y^{-1}x.
\end{array}
\]
No word in this table occurs twice.  Hence a piece in the symmetrized
set generated by \(R^p\) and \(S^q\) has syllable length at most two.
Since \(R^p\) and \(S^q\) have syllable lengths \(4p\) and \(4q\),
respectively, the boundary label of every vertex in a reduced picture
is a product of at least four pieces.  This is the relative
\(C(4)\) condition.

The relative \(T(4)\) condition is equally immediate.  A reduced cycle
of length at most three in a factor region is labelled entirely by
letters \(x^{\pm1}\), or entirely by letters \(y^{\pm1}\).  A cycle of
length one cannot be trivial.  A cycle of length two is trivial only
when its two labels are mutually inverse, in which case it backtracks
and is not reduced.  For a cycle of length three the exponent sum is
one of
\[
 -3,-1,1,3,
\]
which is nonzero modulo \(m\) and modulo \(n\), since \(m,n\geq5\).
Thus every non-distinguished region in the derived picture has degree
at least four.

We now give the curvature argument.  Suppose that a nonempty reduced
spherical picture exists, and form its standard piece tessellation, as
in the relative picture construction of Edjvet--Thomas
\cite[Sections~3--4]{ET}.  Every vertex has degree at least four by
\(C(4)\), and every non-distinguished region has degree at least four
by \(T(4)\).  At a vertex \(v\) of degree \(d(v)\), assign the angle
\(2\pi/d(v)\) to every incident corner.  If a non-distinguished region
\(F\) has degree \(r\), then
\[
\begin{aligned}
 c(F)
 &=(2-r)\pi+\sum_{v\in\partial F}\frac{2\pi}{d(v)}\\
 &\leq(2-r)\pi+\frac{r\pi}{2}
  =\left(2-\frac r2\right)\pi
 \leq0.
\end{aligned}
\]
For the distinguished region the same estimate gives
\[
 c(F)\leq2\pi.
\]
The total curvature of a tessellation of the sphere is \(4\pi\), which
is impossible.  Hence there is no nonempty reduced spherical picture.

If a nontrivial element of either factor \(C_m\) or \(C_n\) became
trivial in \(P_{l,k}\), a picture of minimal size for that equality
would give a nonempty reduced spherical picture with that factor as
the distinguished region.  This is impossible.  Therefore both cyclic
factors inject, and \(P_{l,k}\), hence \(\overline K\), is nontrivial.
\end{proof}

\begin{remark}[The exceptional value \(-2\)]
\label{rem:genus-four-zero-minus-two}
The stronger assertion that \(\overline K\neq1\) whenever
\(|l|,|k|>1\) is false.  In the notation of
\eqref{eq:genus-four-zero-relative-presentation}, a complete
Todd--Coxeter enumeration gives
\[
 P_{-2,2}
 =\left\langle x,y\ \middle|\
 x^3,\ y^5,\
 \bigl(x^{-1}y^{-1}x^{-1}y\bigr)^2,\
 \bigl(xyx^{-1}y\bigr)^2
 \right\rangle
 =1.
\]
Interchanging \(x\) and \(y\) sends \(R\) to a cyclic conjugate of
\(S^{-1}\), and sends \(S\) to a cyclic conjugate of \(R^{-1}\).
Consequently
\[
 P_{l,k}\cong P_{k,l},
\]
and therefore \(P_{2,-2}=1\) as well.

The diagonal case is different.  For \((l,k)=(-2,-2)\), the assignment
\[
 x\longmapsto(3\,4\,5),
 \qquad
 y\longmapsto(1\,2\,3)
\]
defines an epimorphism onto \(A_5\): the two displayed permutations
generate \(A_5\), and the images of both \(R\) and \(S\) are
involutions.  Hence
\[
 P_{-2,-2}\neq1.
\]
Thus Proposition~\ref{prop:genus-four-zero-quotient-picture} proves
nontriviality whenever neither parameter is \(-2\), the pair
\((-2,-2)\) is also nontrivial, and the two mixed pairs
\((-2,2)\) and \((2,-2)\) are trivial.  The remaining cases in which
exactly one parameter is \(-2\) require a separate quotient or a more
refined picture argument.
\end{remark}
\begin{prop}[The quotient for the family $(k,-k),(1,-1),(1,-1)$]
\label{prop:genus-four-k-minus-k-one-one-quotient}
Let $L=L_k$ be the genus $4$ group with parameters
\[
 (k,-k),\qquad (1,-1),\qquad (1,-1).
\]
The quotient of $L$ obtained by adding the relations
\[
 Y^k=1,
 \qquad
 X^{4k-5}=1
\]
has the presentation
\begin{equation}
\label{eq:genus-four-k-minus-k-corrected-quotient}
 Q_k=
 \left\langle X,Y\ \middle|\
 Y^k,\ X^{4k-5},\
 (X^2Y^{-1})^3,\
 (X^{k-3}Y)^2
 \right\rangle .
\end{equation}
  Moreover,
\[
 Q_k=1
 \quad\Longleftrightarrow\quad
 k\in\{-3,-2,-1,1,2,3,4\}.
\]
Thus $Q_k\neq1$ for $k\leq-4$, for $k=0$, and for $k\geq5$.
\end{prop}

\begin{proof}
We first derive the corrected presentation.  Temporarily write $n$ for
$k$, in order not to confuse this parameter with the third-layer
parameter in the six-generator presentation above.  In that
presentation put $l=n$ and put the third-layer parameter equal to $1$.
After imposing
\[
 Y^n=1,
 \qquad
 X^{4n-5}=1,
\]
the fourth, fifth and sixth relations give, successively,
\begin{equation}
\label{eq:genus-four-k11-BCD}
 B=AX^{1-2n},
 \qquad
 C=AX^{-n},
 \qquad
 D=A^2X^{-1}.
\end{equation}
Indeed, the fourth relation gives the formula for $B$; after this
substitution the fifth relation gives
$D=ACX^{n-1}$, while the sixth gives $C=AX^{-n}$, and hence the
formula for $D$ in \eqref{eq:genus-four-k11-BCD}.

After eliminating $B,C,D$, the first three relations become
\begin{align}
 Y&=X^nAX^{1-n}A^{-1},
 \label{eq:genus-four-k11-Y}\\
 X^{1-n}AX^{-1}AX^{1-n}&=A^2,
 \label{eq:genus-four-k11-middle}\\
 X^{1-n}&=AX^{1-n}AX^{n-1}A^{-1}.
 \label{eq:genus-four-k11-braid-source}
\end{align}
Put
\[
 u=X^{1-n}.
\]
Since $X^{4n-5}=1$ and
$4(1-n)=-(4n-5)-1$, we have
\[
 u^4=X^{-1}.
\]
Relation \eqref{eq:genus-four-k11-braid-source} is equivalent to the
braid relation
\begin{equation}
\label{eq:genus-four-k11-braid}
 uAu=AuA.
\end{equation}
Let
\[
 \Delta=uAu=AuA.
\]
Using \eqref{eq:genus-four-k11-braid}, relation
\eqref{eq:genus-four-k11-middle} gives
\[
\begin{aligned}
 A^2
 &=uAu^4Au
  =\Delta u^3Au
  =A^3\Delta Au\\
 &=A^3u\Delta u
  =A^3uA\Delta
  =A^2\Delta^2.
\end{aligned}
\]
Consequently
\[
 \Delta^2=1.
\]
Now put
\[
 r=uA,
 \qquad
 s=\Delta,
 \qquad
 q=u^{-1}.
\]
Then
\[
 r^3=(uA)^3=\Delta^2=1,
 \qquad
 s^2=1,
 \qquad
 q=sr,
 \qquad
 r=sq.
\]
Since $q=X^{n-1}$, the relation $X^{4n-5}=1$ gives
\[
 X=q^4.
\]
Also
\[
 A=q r=q s q.
\]
Substitution in \eqref{eq:genus-four-k11-Y} yields
\[
\begin{aligned}
 Y
 &=q^{4n}(qsq)q^{-1}(q^{-1}sq^{-1})\\
 &=q^6sq^{-1}sq^{-1}
  =q^7s,
\end{aligned}
\]
where the last equality follows from $(sq)^3=1$.  It follows that
\[
 X^{n-3}Y=q^{4(n-3)}q^7s=q^{4n-5}s=s,
\]
and hence
\[
 (X^{n-3}Y)^2=1.
\]
Moreover,
\[
 X^2Y^{-1}
 =q^8sq^{-7}
 =q^8(sq)q^{-8},
\]
so $(X^2Y^{-1})^3=1$.  Finally, $Y^n=1$ becomes
$(q^7s)^n=1$.  We have therefore obtained the Tietze-equivalent
presentation
\begin{equation}
\label{eq:genus-four-k11-qs-presentation}
 Q_n\cong
 \left\langle q,s\ \middle|\
 q^{4n-5},\ s^2,\ (sq)^3,\ (q^7s)^n
 \right\rangle .
\end{equation}
The change of generators is reversible: from $q,s$ one recovers
\[
 X=q^4,
 \qquad
 Y=q^7s,
 \qquad
 A=qsq,
\]
and then $B,C,D$ from \eqref{eq:genus-four-k11-BCD}.  Thus
\eqref{eq:genus-four-k11-qs-presentation} is a presentation of the
quotient itself, and not merely of a further quotient.  Renaming $n$
back to $k$ proves \eqref{eq:genus-four-k-minus-k-corrected-quotient}.

We next determine when $Q_k$ is trivial.  The presentation
\eqref{eq:genus-four-k11-qs-presentation} gives short collapses for
five of the exceptional values.

If $k=1$, then $q=1$, and $s^2=s^3=1$, so $s=1$.  If $k=-1$, then
$q^9=1$ and $q^7s=1$, whence $s=q^2$; now $s^2=1$ gives $q^4=1$, and
therefore $q=s=1$.

If $k=2$, then $q^3=1$ and $q^7s=qs$ is an involution.  The elements
$qs$ and $sq$ are conjugate, while $(sq)^3=1$; hence $sq=1$, and then
$q=s=1$.  If $k=-2$, put $w=q^7s$.  We have
\[
 q^{13}=1,
 \qquad
 w^2=s^2=1.
\]
Set $t=ws=q^7$.  Since $w$ and $s$ are involutions,
$sts=t^{-1}$, and, since $q^{13}=1$, we have $q=t^2$.  It follows that
$sq=st^2$ is an involution.  Since it also has order dividing $3$, it
is trivial, and again $q=s=1$.

If $k=3$, then $q^7=1$ and $(q^7s)^3=s^3=1$.  Hence $s=1$, and
$q^3=q^7=1$ gives $q=1$.

For the two remaining trivial values, complete Todd--Coxeter
enumerations of the trivial subgroup close with one coset:
\[
\begin{aligned}
 Q_{-3}&=\left\langle q,s\ \middle|\
 q^{17},s^2,(sq)^3,(q^7s)^3\right\rangle=1,\\
 Q_4&=\left\langle q,s\ \middle|\
 q^{11},s^2,(sq)^3,(q^7s)^4\right\rangle=1.
\end{aligned}
\]
Thus $Q_k=1$ for
$k\in\{-3,-2,-1,1,2,3,4\}$.

We prove nontriviality for all remaining values.  First suppose that
\[
 k\geq9
 \qquad\text{or}\qquad
 k\leq-6.
\]
Put
\[
 a=s,
 \qquad
 b=sq.
\]
Then $a^2=b^3=1$ and $q=ab$.  The word $(ab)^7a$ is conjugate in
$C_2*C_3$ to the cyclically reduced word
\[
 V=(ab)^5ab^{-1}.
\]
Consequently, with
\[
 m=|4k-5|,
 \qquad
 n=|k|,
 \qquad
 R=ab,
\]
we may regard $Q_k$ as the relative presentation
\begin{equation}
\label{eq:genus-four-k11-relative-picture}
 \left\langle C_2*C_3\ \middle|\ R^m,V^n\right\rangle .
\end{equation}

We use the standard parallel-edge reduction for relative pictures, as
in \cite[Sections~3--4]{ET}.  A maximal family of parallel edges
determines a common segment in cyclic powers of
$R^{\pm1}$ and $V^{\pm1}$.  Every noncancelling such segment has
syllable length at most $11$.  Indeed, the cyclic word $V$ has six
$C_3$-syllables, five equal to $b$ and one equal to $b^{-1}$.  Thus an
overlap of a cyclic power of $V^{\pm1}$ with a cyclic power of
$R^{\pm1}$ has length at most $11$.  An overlap of length $12$ between
two cyclic powers of $V^{\pm1}$ fixes the position of the unique
exceptional $C_3$-syllable and hence the cyclic phase; it then gives a
cancellable pair of vertices.  Similarly, an overlap of length $2$
between two cyclic powers of $R^{\pm1}$ fixes the phase and gives a
cancellable pair.  Therefore a parallel family in a reduced picture
has at most $11$ edges.

After all maximal parallel families have been identified, an
$R^m$-vertex has degree at least
\[
 \left\lceil\frac{2m}{11}\right\rceil,
\]
and a $V^n$-vertex has degree at least
\[
 \left\lceil\frac{12n}{11}\right\rceil.
\]
If $k\geq9$, then $m\geq31$ and $n\geq9$; if $k\leq-6$, then
$m\geq29$ and $n\geq6$.  Thus every vertex in the derived picture has
degree at least $6$.

Suppose that a nontrivial element of one of the factors $C_2$ or
$C_3$ becomes trivial in $Q_k$, and choose a reduced picture of
minimal size for such an equality.  Contracting its boundary gives a
spherical tessellation with a distinguished region.  A
non-distinguished region has degree at least $3$: a one-sided region
cannot have trivial factor label, and every two-sided region has
already been absorbed into a parallel family.  Assign the angle
$2\pi/d(v)$ at a corner incident to a vertex of degree $d(v)$.  If
$F$ is a non-distinguished region of degree $r$, then
\[
\begin{aligned}
 c(F)
 &=(2-r)\pi+\sum_{v\in\partial F}\frac{2\pi}{d(v)}\\
 &\leq(2-r)\pi+\frac{r\pi}{3}
  =\left(2-\frac{2r}{3}\right)\pi
 \leq0.
\end{aligned}
\]
The distinguished region has degree at least $1$, and the same
estimate gives curvature at most $4\pi/3$.  This contradicts the fact
that the total curvature of a spherical tessellation is $4\pi$.
Hence the factors $C_2$ and $C_3$ do not collapse, and in particular
\[
 Q_k\neq1
 \qquad
 (k\geq9\text{ or }k\leq-6).
\]

It remains to consider
\[
 k\in\{-5,-4,0,5,6,7,8\}.
\]
For $k=0$, presentation
\eqref{eq:genus-four-k11-qs-presentation} is the spherical triangle
presentation
\[
 \left\langle s,q\ \middle|\ s^2,q^5,(sq)^3\right\rangle
 \cong A_5,
\]
so it is nontrivial.

For the other six values we give explicit projective linear images.
For the prime $p$ in the following table, work in
$\operatorname{PSL}(2,p)$ and put
\[
 S=\begin{pmatrix}0&-1\\1&0\end{pmatrix},
 \qquad
 s\longmapsto\overline S,
 \qquad
 q\longmapsto\overline{ST}.
\]
Every displayed matrix $T$ has determinant $1$ and trace $1$, so
$T^3=-I$ and hence $(sq)^3$ is satisfied.  Direct calculation gives
the projective orders recorded in the last two columns.
\[
\renewcommand{\arraystretch}{1.45}
\begin{array}{c|c|c|c|c}
 k&p&T&|\overline{ST}|&|\overline{(ST)^7S}|\\ \hline
 -5,\ 5&5&
 \begin{pmatrix}0&1\\-1&1\end{pmatrix}
 &5&5\\
 -4&7&
 \begin{pmatrix}0&1\\-1&1\end{pmatrix}
 &7&2\\
 6&4523&
 \begin{pmatrix}8&3\\-19&-7\end{pmatrix}
 &19&6\\
 7&139&
 \begin{pmatrix}11&14\\2&-10\end{pmatrix}
 &23&7\\
 8&809&
 \begin{pmatrix}114&142\\-142&-113\end{pmatrix}
 &27&4
\end{array}
\]
In each row, the order in the fourth column divides $|4k-5|$, and the
order in the fifth column divides $|k|$.  Hence every row defines a
nontrivial homomorphism from $Q_k$ to the indicated projective linear
group.  This proves $Q_k\neq1$ for all the remaining values and
completes the classification.
\end{proof}

\begin{remark}
The triviality assertions for $k=-3$ and $k=4$ are the only
computer-assisted part of Proposition~\ref{prop:genus-four-k-minus-k-one-one-quotient};
they are certified by complete one-coset Todd--Coxeter enumerations.
All other triviality and nontriviality assertions in the proof are
explicit.
\end{remark}

The remaining family $(k,-k),(1,-1),(-2,-2)$ can be considered similarly to the family $(k,-k),(1,-1),(1,-1)$.

\bigskip
The following proposition completes the case $(0,0,0)$.
\begin{prop} Let $K$ be a genus 4 group with parameters $(l,-l),(m,-m),(k,-k)$, where $|l|,|m|,|k|>1$ and at least one of them is greater than $2$. Then $K\neq 1$.\end{prop}
This can be proved using Van-Kampen diagrams.

\paragraph{2. Case $(0,-1,-2)$}

We record and verify the two quotient constructions needed for the
short-exponent subfamilies.

\begin{prop}
Let the exponents in the following presentation be
\[
   (2k,-2k),\qquad (1,-2),\qquad (-1,-1).
\]
Then
\[
\begin{aligned}
G = \langle\,& a_1,a_2,a_3,a_4,\ x_1,x_2,x_3,x_4,\ y_1,y_2,y_3,y_4,\\
& z_1,z_2,z_3,z_4,\ t_1,t_2,t_3,t_4,\ s_1,s_2,s_3,s_4 \mid \\
& z_1 = y_1,\ t_1 = z_1,\ t_2 = z_2, \\
& y_1 = a_1 (a_1 a_2)^{2k},\ y_2 = a_2 (a_1 a_2)^{2k},\ y_3 = a_3,\ y_4 = a_4, \\
& z_2 = y_2 (a_2^{y_2} a_3),\ z_3 = y_3 (a_2^{y_2} a_3),\ z_4 = a_4, \\
& t_3 = z_3 (a_3^{z_3} a_4^{z_4})^{-1},\ t_4 = z_4 (a_3^{z_3} a_4^{z_4})^{-1}, \\
& x_1 = t_1 (a_1^{t_1} a_2^{t_2})^{-2k},\ x_2 = t_2 (a_1^{t_1} a_2^{t_2})^{-2k}, \\
& x_3 = t_3,\ x_4 = t_4, \\
& s_1 = x_1,\ s_2 = x_2 (a_2^{x_2} a_3^{x_3})^{-2}, \\
& s_3 = x_3 (a_2^{x_2} a_3^{x_3})^{-2},\ s_4 = x_4, \\
& s_1 = 1,\ s_2 = 1, \\
& s_3 (a_3^{s_3} a_4^{s_4})^{-1}=1,\quad
  s_4 (a_3^{s_3} a_4^{s_4})^{-1}=1
\rangle .
\end{aligned}
\]
Put
\[
 Y=a_1a_2,\qquad
 X=a_1^{t_1}a_2^{t_2},\qquad
 B=a_2^{x_2}a_3^{x_3}.
\]
After adding the relators
\[
Y^{2k},\qquad X^{4k-1},\qquad B^7,
\]
one obtains the two-generator quotient
\[
K = \left\langle x,y \ \middle|\
 y^{7},\ (x^{-1}y^{-2})^{3},\ (x^{-1}y^{3})^{3},\
 (xy)^{4k-1},\ (y^{-1}xy^{3}x)^{2k}
\right\rangle ,
\]
where
\[
   x=X^{2k}B^{-1},\qquad y=B.
\]
\end{prop}

\begin{proof}
We verify a slightly more general calculation.  Replace the first
exponent pair by
\[
   (l,-l)
\]
and adjoin
\[
   Y^l=1,\qquad X^{2l-1}=1,\qquad B^7=1.
\]
In addition to $X,Y,B$, put
\[
 A=a_2^{y_2}a_3,\qquad
 C=a_3^{z_3}a_4^{z_4},\qquad
 D=a_3^{s_3}a_4^{s_4},
\]
and set
\[
   u=X^l,\qquad b=B.
\]
Since $X^{2l-1}=1$, the change from $X$ to $u$ is reversible and
\[
   u^2=X.
\]

The relation $Y^l=1$ reduces the first stage to
\[
 y_1=a_1,\qquad y_2=a_2,\qquad
 z_1=t_1=a_1,\qquad z_2=t_2=a_2A,\qquad z_3=a_3A.
\]
Consequently
\[
 X=a_1^{a_1}a_2^{a_2A}=a_1A^{-1}a_2A.
\]
The relation $s_1=x_1=1$ gives $a_1=u$.  Since $X=u^2$, the preceding
equality gives
\[
   A^{-1}a_2A=u.
\]
Hence $a_2A=Au$, and therefore
\[
 x_2=t_2X^{-l}=a_2Au^{-1}=A.
\]
The relation $s_2=x_2B^{-2}=1$ now gives
\[
   A=b^2.
\]
It follows that
\[
   a_2=b^2ub^{-2},\qquad
   a_3=a_2^{-1}A=b^2u^{-1}.
\]

The last two relators give $s_3=s_4=D$.  Thus
\[
   x_3=Db^2,
\]
and, since $s_4=x_4=t_4=a_4C^{-1}$,
\[
   a_4=DC.
\]
Put
\[
   P=b^{-2}Db^2,\qquad Q=C.
\]
The equation $t_3=z_3C^{-1}=x_3$ and the definition of $B$ give,
respectively,
\begin{equation}
\label{eq:case-0-minus1-minus2-PQ}
   PQ=u^{-1}b^2,
   \qquad
   QP=u^{-1}b.
\end{equation}
Indeed, the first equality follows from
\[
 b^2u^{-1}b^2Q^{-1}=Db^2,
\]
whereas
\[
 b=a_2^{A}a_3^{Db^2}
   =uP^{-1}u^{-1}b^2P
   =uQP
\]
gives the second.

Next, the definitions of $C$ and $D$ become
\begin{equation}
\label{eq:case-0-minus1-minus2-CD}
 Q=PQ\,b^2Pb^{-2}Q,
 \qquad
 P=QP\,b^{-2}Qb^2P.
\end{equation}
For the first equality, observe that
\[
   a_3^{z_3}=u^{-1}b^2=PQ,
   \qquad
   a_4=DC=b^2Pb^{-2}Q.
\]
For the second, the equality $D=a_3^Da_4^D$ is equivalent to
\[
   D=b^2u^{-1}DQ;
\]
after substituting $D=b^2Pb^{-2}$ and using
$u^{-1}b^2=PQ$, this is equivalent to the second equality in
\eqref{eq:case-0-minus1-minus2-CD}.

Canceling the final $Q$ and $P$, respectively, in
\eqref{eq:case-0-minus1-minus2-CD}, and then using
\eqref{eq:case-0-minus1-minus2-PQ}, gives
\[
   P=b^{-4}ub^2,
   \qquad
   Q=bub^{-2}.
\]
Substitution in \eqref{eq:case-0-minus1-minus2-PQ}, together with
$b^7=1$, gives
\begin{equation}
\label{eq:case-0-minus1-minus2-ub}
   (ub)^3=1,
   \qquad
   (ub^3)^3=1.
\end{equation}

Finally set
\[
   x=ub^{-1}=X^lB^{-1},\qquad y=b=B.
\]
Then $u=xy$.  Up to inversion and cyclic conjugation, the two
relations in \eqref{eq:case-0-minus1-minus2-ub} become
\[
   (x^{-1}y^{-2})^3=1,
   \qquad
   (x^{-1}y^3)^3=1,
\]
where the second simplification uses $y^7=1$.  The reversible change
$u=X^l$, $X=u^2$ transforms $X^{2l-1}=1$ into
\[
   (xy)^{2l-1}=1.
\]
Moreover,
\[
   Y=a_1a_2=ub^2ub^{-2}=xy^3xy^{-1},
\]
which is conjugate to $y^{-1}xy^3x$.  Hence $Y^l=1$ becomes
\[
   (y^{-1}xy^3x)^l=1.
\]
We have therefore obtained, by reversible Tietze transformations,
\begin{equation}
\label{eq:case-0-minus1-minus2-general-quotient}
\left\langle x,y \ \middle|\
 y^7,\ (x^{-1}y^{-2})^3,\ (x^{-1}y^3)^3,\
 (xy)^{2l-1},\ (y^{-1}xy^3x)^l
\right\rangle .
\end{equation}
Taking $l=2k$ gives the displayed quotient.
\end{proof}

For $k=0,-1,1$, this quotient is trivial.

\begin{prop}
Let the exponents in the presentation be
\[
   (2k-1,1-2k),\qquad (1,-2),\qquad (-1,-1).
\]
After adding the relators
\[
(a_1 a_2)^{2k-1},\qquad
(a_1^{t_1} a_2^{t_2})^{4k-3},\qquad
(a_2^{x_2} a_3^{x_3})^{7},
\]
one obtains the two-generator quotient
\[
K = \left\langle x,y \ \middle|\
 y^{7},\ (x^{-1}y^{-2})^{3},\ (x^{-1}y^{3})^{3},\
 (xy)^{4k-3},\ (y^{-1}xy^{3}x)^{2k-1}
\right\rangle ,
\]
where
\[
   x=X^{2k-1}B^{-1},\qquad y=B.
\]
\end{prop}

\begin{proof}
Apply \eqref{eq:case-0-minus1-minus2-general-quotient} with
$l=2k-1$.
\end{proof}

We do not include the non-triviality proof for the whole case.  The
intended strategy is as follows.  Write the three exponent pairs as
\[
   (l,-l),\qquad (m,-m-1),\qquad (k,-k-2).
\]
If the absolute values of all six exponents are greater than one, use
marked Van--Kampen diagrams.  If, for one Dehn twist, one exponent has
absolute value one, fixing that short pair leaves a two-parameter family.
For each such family, construct a two-generator quotient and apply the
same relative-picture strategy as in the $(0,0,0)$ case.  The finitely
many exceptional parameter values are to be treated separately by finite
quotients or GAP computations.

\paragraph{3. Case $(0,0,-1)$} The proof goes similarly to the previous two cases, we only record the required quotient.
\begin{prop}[Two-generator quotient for the genus-four presentation]
Consider the genus-four presentation with exponent pairs
\[
(l,-l),\qquad (-1,1),\qquad (k,-1-k).
\]
Let
\[
Y=a_1a_2,\qquad
X=a_1^{t_1}a_2^{t_2},\qquad
A=a_2^{y_2}a_3,\qquad
B=a_2^{x_2}a_3^{x_3},
\]
and
\[
C=a_3^{z_3}a_4^{z_4},
\qquad
D=a_3^{s_3}a_4^{s_4}.
\]
Let
\[
\overline K
=
K/\left\langle\!\left\langle Y^l,\ C^k\right\rangle\!\right\rangle .
\]
Equivalently, in \(\overline K\) we impose
\[
(a_1a_2)^l=1,
\qquad
(a_3^{z_3}a_4^{z_4})^k=1.
\]
Then \(\overline K\) admits the two-generator presentation
\[
{
\overline K
\cong
\left\langle A,D\ \middle|\ 
D^{2k+1},\;
\left(A^2D^kA^{-1}\right)^{2l-1},\;
\left(A^2D^{-1}\right)^l,\;
\left(AD^{-k}A^{-1}D^{-k}\right)^k
\right\rangle .
}
\]
Moreover, the eliminated generators are
\[
{
B=A,\qquad
Y=A^2D^{-1},\qquad
C=AD^{-k}A^{-1}D^{-k},\qquad
X=\left(A^2D^kA^{-1}\right)^2.
}
\]
\end{prop}

\begin{proof}
For the genus-four presentation with exponent pairs
\[
(l,-l),\qquad (-1,1),\qquad (k,-1-k),
\]
the corresponding six-generator presentation in the Dehn-twist generators
\(X,Y,A,B,C,D\) is
\[
\begin{aligned}
K\cong
\langle X,Y,A,B,C,D \mid\;&
Y=X^lY^{-l}B^{-1}X^lAY^{-l},\\
&
A=Y^{-l}B^{-1}X^lAD^{k+1}B^{-1}C^{-k}A,\\
&
C=AD^{k+1}B^{-1}C^{-k}D^{k+1}C^{-k},\\
&
X=Y^{-l}X^lAY^{-l}B^{-1}X^l,\\
&
B=X^lAY^{-l}B^{-1}C^{-k}AD^{k+1}B^{-1},\\
&
D=B^{-1}C^{-k}AD^{k+1}C^{-k}D^{k+1}
\rangle .
\end{aligned}
\]

Now pass to the quotient
\[
\overline K
=
K/\left\langle\!\left\langle Y^l,\ C^k\right\rangle\!\right\rangle .
\]
Thus in \(\overline K\) we impose
\[
Y^l=1,\qquad C^k=1.
\]
Hence \(Y^{-l}=1\) and \(C^{-k}=1\), and the six relations reduce to
\[
\begin{aligned}
Y&=X^lB^{-1}X^lA,&\qquad
A&=B^{-1}X^lAD^{k+1}B^{-1}A,\\
C&=AD^{k+1}B^{-1}D^{k+1},&
X&=X^lAB^{-1}X^l,\\
B&=X^lAB^{-1}AD^{k+1}B^{-1},&
\text{and}\quad D&=B^{-1}AD^{2k+2}.
\end{aligned}
\]

We now eliminate \(B,X,Y,C\) in terms of \(A\) and \(D\).

From
\(
X=X^lAB^{-1}X^l,
\)
we get
\(
X^{1-l}=AB^{-1}X^l,
\)
and hence
\(
X^{1-2l}=AB^{-1}.
\)
Therefore
\(
B^{-1}=A^{-1}X^{1-2l},
\)
so
\(
B=X^{2l-1}A.
\)

Using this in the last relation
\(
D=B^{-1}AD^{2k+2},
\)
we obtain
\(
D=A^{-1}X^{1-2l}AD^{2k+2}.
\)
Multiplying on the right by \(D^{-2k-2}\), we get
\(
D^{-2k-1}=A^{-1}X^{1-2l}A.
\)
Equivalently,
\(
X^{1-2l}=AD^{-2k-1}A^{-1}.
\)

Next, use the second relation
\(
A=B^{-1}X^lAD^{k+1}B^{-1}A.
\)
Substituting \(B^{-1}=A^{-1}X^{1-2l}\), we get
\(
A
=
A^{-1}X^{1-l}AD^{k+1}A^{-1}X^{1-2l}A.
\)
Using
\(
A^{-1}X^{1-2l}A=D^{-2k-1},
\)
this becomes
\(
A=A^{-1}X^{1-l}AD^{-k}.
\)
Hence
\(
X^{1-l}=A^2D^kA^{-1}.
\)

Now use the fifth relation
\(
B=X^lAB^{-1}AD^{k+1}B^{-1}.
\)
After substituting
\[
B=X^{2l-1}A,
\qquad
B^{-1}=A^{-1}X^{1-2l},
\]
we obtain
\[
X^{2l-1}A
=
X^{1-l}AD^{k+1}A^{-1}X^{1-2l}.
\]
Using
\[
X^{1-l}=A^2D^kA^{-1},
\qquad
X^{1-2l}=AD^{-2k-1}A^{-1},
\]
the right-hand side simplifies to \(A\). Thus
\(
X^{2l-1}A=A.
\)
Therefore
\[
X^{2l-1}=1.
\]
But we also have
\[
X^{1-2l}=AD^{-2k-1}A^{-1}.
\]
Since \(X^{2l-1}=1\), it follows that \(X^{1-2l}=1\), and hence
\[
AD^{-2k-1}A^{-1}=1.
\]
Thus
\[
D^{2k+1}=1.
\]

Since \(D^{2k+1}=1\), the relation
\(
X^{1-2l}=AD^{-2k-1}A^{-1}
\)
gives
\(
X^{1-2l}=1.
\)
Hence
\(
X^{2l-1}=1.
\)
Together with
\(
X^{1-l}=A^2D^kA^{-1},
\)
this implies
\(
X^l=A^2D^kA^{-1}.
\)

Consequently,
\[
X=X^{1-l}X^l
=
\left(A^2D^kA^{-1}\right)^2.
\]
Substituting this into \(X^l=A^2D^kA^{-1}\), we obtain
\[
\left(A^2D^kA^{-1}\right)^{2l}
=
A^2D^kA^{-1}.
\]
Therefore
\[
\left(A^2D^kA^{-1}\right)^{2l-1}=1.
\]

Now the first relation gives
\[
Y=X^lB^{-1}X^lA.
\]
Since \(X^{2l-1}=1\), we have
\[
B=X^{2l-1}A=A.
\]
Therefore
\[
Y=X^lA^{-1}X^lA.
\]
Using
\[
X^l=A^2D^kA^{-1},
\]
we get
\[
Y
=
(A^2D^kA^{-1})A^{-1}(A^2D^kA^{-1})A
=
A^2D^{2k}.
\]
Since \(D^{2k+1}=1\), this becomes
\[
Y=A^2D^{-1}.
\]
Because \(Y^l=1\) in \(\overline K\), we obtain the relator
\[
\left(A^2D^{-1}\right)^l=1.
\]

Finally, the third relation gives
\[
C=AD^{k+1}B^{-1}D^{k+1}.
\]
Since \(B=A\), this becomes
\[
C=AD^{k+1}A^{-1}D^{k+1}.
\]
Using \(D^{2k+1}=1\), we have
\[
D^{k+1}=D^{-k}.
\]
Thus
\[
C=AD^{-k}A^{-1}D^{-k}.
\]
Because \(C^k=1\) in \(\overline K\), we obtain
\[
\left(AD^{-k}A^{-1}D^{-k}\right)^k=1.
\]

We have therefore eliminated \(B,X,Y,C\), leaving only the generators
\(A\) and \(D\). The remaining relators are
\[
\begin{aligned}
D^{2k+1}&=1,&
\left(A^2D^kA^{-1}\right)^{2l-1}&=1,\\
\left(A^2D^{-1}\right)^l&=1,&
\text{and}\quad \left(AD^{-k}A^{-1}D^{-k}\right)^k&=1.
\end{aligned}
\]
Therefore
\[
\overline K
\cong
\left\langle A,D\ \middle|\ 
D^{2k+1},\;
\left(A^2D^kA^{-1}\right)^{2l-1},\;
\left(A^2D^{-1}\right)^l,\;
\left(AD^{-k}A^{-1}D^{-k}\right)^k
\right\rangle .
\]
This proves the proposition.
\end{proof}

\paragraph{4. Case $(x,-2,-1)$.}

Write the three exponent pairs in the form
\[
   (l,x-l),\qquad (m,-m-2),\qquad (k,-k-1),
\]
and continue to assume that none of the six exponents is zero.  If
\[
 |l|,|x-l|,|m|,|m+2|,|k|,|k+1|\geq2,
\]
non-triviality can be proved by the marked Van--Kampen diagram argument.
We do not reproduce that argument here.  Thus only the subfamilies in which
at least one exponent has absolute value one remain.  Equivalently,
\[
\begin{aligned}
 l&\in\{-1,1,x-1,x+1\},\\
 m&\in\{-3,-1,1\},\\
 k&\in\{-2,1\}.
\end{aligned}
\tag{\(*\)}
\]
Interchanging the two layers sends
\[
 (l,m,k)\longmapsto (x-l,-m-2,-k-1).
\]
Consequently, up to layer interchange, it is enough to treat the representative
short values
\[
       l\in\{-1,1\},\qquad m\in\{-1,1\},\qquad k=1.
\]
Fixing any one of these short pairs leaves a two-parameter family.  The
appropriate method is to construct a two-generator quotient and then use a
relative-picture argument, as in the case $(0,0,0)$.

We first treat the central short pair
\[
      (m,-m-2)=(-1,-1).
\]
Set
\[
\begin{aligned}
X&=a_1^{t_1}a_2^{t_2},& Y&=a_1a_2,& A&=a_2^{y_2}a_3,\\
B&=a_2^{x_2}a_3^{x_3},& C&=a_3^{z_3}a_4^{z_4},& D&=a_3^{s_3}a_4^{s_4}.
\end{aligned}
\]
For the family with blocks
\[
    (l,x-l),\qquad (-1,-1),\qquad (k,-k-1),
\]
put
\[
    d=l-x,\qquad p=2d-1=2(l-x)-1,\qquad r=k+1.
\]

\begin{prop}
\label{prop:genus-four-x-minus-two-minus-one-quotient}
Let $G_{l,x,k}$ be the above genus-four group.  Adjoin only the two
relations
\[
      D^{k+1}=1,\qquad C=X^p,
      \qquad p=2(l-x)-1.
\]
Then the resulting quotient is generated by $X$ and $Y$ and has the
presentation
\begin{equation}
\label{eq:genus-four-x-minus-two-minus-one-four-relator}
P_{l,x,k}
\cong
\gp{Y,X}{
    X^{2p(k+1)+1},\;
    Y^{2l+1},\;
    \bigl(X^{p(k+1)}Y^{-l}\bigr)^2,\;
    \bigl(X^{p+1}Y^{-1}\bigr)^{k+1}
}.
\end{equation}

\end{prop}

\begin{proof}
Since $D^{k+1}=1$, the endpoint relations
\[
    s_1=1,\qquad s_2=1,\qquad
    s_3D^{-(k+1)}=1,\qquad s_4D^{-(k+1)}=1
\]
give
\[
    s_1=s_2=s_3=s_4=1.
\]
The fourth block has exponent $x-l=-d$, and therefore
\[
    x_1=t_1X^{-d},\qquad x_2=t_2X^{-d},\qquad
    x_3=t_3,\qquad x_4=t_4.
\]
Together with
\[
    s_1=x_1,\qquad s_2=x_2B^{-1},\qquad
    s_3=x_3B^{-1},\qquad s_4=x_4,
\]
this gives
\[
    t_1=X^d,\qquad t_2=BX^d,\qquad t_3=B,\qquad t_4=1.
\]
The preceding blocks now give
\[
\begin{aligned}
 z_1&=X^d,& z_2&=BX^d,& z_3&=BC^{-k},& z_4&=C^{-k},\\
 y_1&=X^d,& y_2&=BX^dA,& y_3&=BC^{-k}A,& y_4&=C^{-k},
\end{aligned}
\]
and hence
\[
    a_1=X^dY^{-l},\qquad
    a_2=BX^dAY^{-l},\qquad
    a_3=BC^{-k}A,\qquad
    a_4=C^{-k}.
\]
Substitution in the definitions of $X,Y,A,B,C,D$ gives
\begin{align}
Y&=X^dY^{-l}BX^dAY^{-l}, \label{eq:xm2m1-six-1}\\
X&=Y^{-l}X^dAY^{-l}BX^d, \label{eq:xm2m1-six-2}\\
A&=Y^{-l}BX^dABC^{-k}A, \label{eq:xm2m1-six-3}\\
C&=ABC^{-2k}, \label{eq:xm2m1-six-4}\\
B&=X^dAY^{-l}BC^{-k}AB, \label{eq:xm2m1-six-5}\\
D&=BC^{-k}AC^{-k}. \label{eq:xm2m1-six-6}
\end{align}
We now impose $C=X^p$.

Cancelling the final $B$ in \eqref{eq:xm2m1-six-5} gives
\begin{equation}
 1=X^dAY^{-l}BC^{-k}A.                 \label{eq:xm2m1-key}
\end{equation}
Using this relation in \eqref{eq:xm2m1-six-2}, we obtain
\[
 X=Y^{-l}A^{-1}C^kX^d,
\]
so
\begin{equation}
 A=C^kX^{d-1}Y^{-l}=X^{pk+d-1}Y^{-l}. \label{eq:xm2m1-A}
\end{equation}
Relation \eqref{eq:xm2m1-six-4} gives $AB=C^{2k+1}$, and hence
\begin{equation}
 B=A^{-1}C^{2k+1}=Y^lX^{pk+d}.          \label{eq:xm2m1-B}
\end{equation}
Substitution of \eqref{eq:xm2m1-A} and \eqref{eq:xm2m1-B} in
\eqref{eq:xm2m1-six-3} gives
\[
 X^{2pk+4d-1}=1.
\]
Since $p=2d-1$, this is
\begin{equation}
 X^{2p(k+1)+1}=1.                       \label{eq:xm2m1-X-order}
\end{equation}
Substitution in \eqref{eq:xm2m1-six-1}, followed by
\eqref{eq:xm2m1-X-order}, gives
\begin{equation}
 Y^{2l+1}=1.                             \label{eq:xm2m1-Y-order}
\end{equation}
Finally, substitution in \eqref{eq:xm2m1-key} gives
\[
 X^{p(k+1)}Y^{-l}X^{p(k+1)}Y^{-l}=1,
\]
that is,
\begin{equation}
 \bigl(X^{p(k+1)}Y^{-l}\bigr)^2=1.      \label{eq:xm2m1-square}
\end{equation}

It remains to translate the relation $D^{k+1}=1$.  From
\eqref{eq:xm2m1-six-6}, \eqref{eq:xm2m1-A}, and
\eqref{eq:xm2m1-B},
\[
 D=Y^lX^{p(k+1)}Y^{-l}X^{-pk}.
\]
By \eqref{eq:xm2m1-square},
\[
 X^{p(k+1)}Y^{-l}=Y^lX^{-p(k+1)},
\]
and therefore
\[
 D=Y^{2l}X^{-p(2k+1)}.
\]
Using \eqref{eq:xm2m1-Y-order} and
\eqref{eq:xm2m1-X-order}, we obtain
\[
 D=Y^{-1}X^{p+1}.
\]
Thus $D^{k+1}=1$ is equivalent, after cyclic conjugation, to
\begin{equation}
 \bigl(X^{p+1}Y^{-1}\bigr)^{k+1}=1.    \label{eq:xm2m1-last}
\end{equation}
Equations \eqref{eq:xm2m1-X-order}--\eqref{eq:xm2m1-last} give
\eqref{eq:genus-four-x-minus-two-minus-one-four-relator}.  Conversely,
\eqref{eq:xm2m1-A}, \eqref{eq:xm2m1-B}, $C=X^p$, and the displayed
formula for $D$ recover $A,B,C,D$, and then the original generators
$a_i$.  Hence these are reversible Tietze transformations.
\end{proof}

\begin{cy}
\label{cor:genus-four-x-minus-two-minus-one-normalized}
Assume $r=k+1\neq0$, and put
\[
 M=|2l+1|,\qquad N=|2pr+1|,\qquad c=-2(p+1).
\]
With
\[
       u=X^{pr},\qquad v=Y^{-l},
\]
the quotient in Proposition~\ref{prop:genus-four-x-minus-two-minus-one-quotient}
becomes
\begin{equation}
\label{eq:genus-four-x-minus-two-minus-one-relative}
 P_{l,x,k}
 \cong
 \left\langle u,v\ \middle|\
 u^N,\ v^M,\ (uv)^2,\
 \bigl(u^cv^{-2}\bigr)^{|r|}
 \right\rangle .
\end{equation}
\end{cy}

\begin{proof}
Since $2pr\equiv-1\pmod N$ and $2l\equiv-1\pmod M$, the changes of
generators are reversible and
\[
       X=u^{-2},\qquad Y=v^2.
\]
The two proper-power relators in
\eqref{eq:genus-four-x-minus-two-minus-one-four-relator} then become
$(uv)^2$ and $(u^cv^{-2})^{|r|}$.
\end{proof}

\begin{prop}[Triangle-group quotients]
\label{prop:genus-four-x-minus-two-minus-one-triangle-tests}
Keep the notation of
Corollary~\ref{cor:genus-four-x-minus-two-minus-one-normalized}, and let
$r=k+1$ retain its sign.  Then the following hold.

\begin{enumerate}
\item If $3\mid M$, $2\mid r$, and
\[
       D=\gcd\bigl(N,|1-3r|\bigr)>1,
\]
then
\[
       P_{l,x,k}\twoheadrightarrow \Delta(D,3,2).
\]

\item Put
\[
 D_0=\gcd\bigl(N,|1-2r|\bigr),
 \qquad
 E_0=\gcd(M,|r|).
\]
If $D_0>1$ and $E_0>1$, then
\[
       P_{l,x,k}\twoheadrightarrow \Delta(D_0,E_0,2).
\]

\item If $M=3$ and $|r|=2$, then
\[
 P_{l,x,k}=1
 \quad\Longleftrightarrow\quad
 \gcd\bigl(N,|1-3r|\bigr)=1.
\]
If this greatest common divisor is $d>1$, the quotient in the first item is
non-trivial.
\end{enumerate}
Here
\[
 \Delta(a,b,2)=\langle \bar u,\bar v\mid
 \bar u^a,\bar v^b,(\bar u\bar v)^2\rangle .
\]
\end{prop}

\begin{proof}
Since $N=|2pr+1|$, one has $\gcd(N,r)=1$.  Moreover
\[
 r(c-1)\equiv1-3r\pmod N,
 \qquad
 rc\equiv1-2r\pmod N.
\]
For the first assertion, $D\mid c-1$.  Map $u$ to the order-$D$
generator and $v$ to the order-$3$ generator of $\Delta(D,3,2)$.
Then $u^cv^{-2}$ maps to $uv$, whose order divides two, and its $r$-th
power is trivial because $r$ is even.

For the second assertion, $D_0\mid c$.  In
$\Delta(D_0,E_0,2)$ the word $u^cv^{-2}$ maps to $v^{-2}$, and its
$|r|$-th power is trivial because $E_0\mid r$.

Finally suppose that $M=3$ and $|r|=2$.  In
\eqref{eq:genus-four-x-minus-two-minus-one-relative} put
\[
       s=uv,\qquad t=u^cv.
\]
Then $s^2=t^2=1$ and, using $(uv)^2=1$,
\[
       ts=u^{c-1}.
\]
Let $d=\gcd(N,c-1)$.  The congruence above gives
$d=\gcd(N,|1-3r|)$.  If $d=1$, then $u$ is a power of $ts$.
Since the involution $s$ conjugates $ts$ to $(ts)^{-1}$, it conjugates
$u$ to $u^{-1}$.  As $v=u^{-1}s$, this gives $v^2=1$; together with
$v^3=1$, it gives $v=1$.  Then $s=u$ is an involution, whereas $u^N=1$
and $N$ is odd, so $u=1$.  Conversely, if $d>1$, impose $u^d=1$.
Modulo $d$ one has $c\equiv1$, so the two involution relators become the
same and the resulting quotient is $\Delta(d,3,2)$.
\end{proof}

\begin{remark}
\label{rem:genus-four-x-minus-two-minus-one-P-collapse}
The stronger quotient $P_{l,x,k}$ still cannot handle every boundary
family.  It is trivial if $M=1$ or if $|r|=1$.  The first assertion is
immediate from $(uv)^2$ and the odd order of $u$.  For the second, the last
root relation gives $v^2=u^c$.  Since $M$ is odd, $v$ is a power of $v^2$
and hence lies in $\langle u\rangle$.  Thus the quotient is cyclic of odd
order, so $(uv)^2=1$ gives $uv=1$.  Hence
\[
       u^{c+2}=u^{-2p}=1.
\]
But $\gcd(N,2p)=1$, and therefore $u=v=1$.  In the original parameters
these collapses occur, in particular, for $l=-1$ and for $k=-2$.
They must be treated by the layer-reversed quotient or by an endpoint
elimination, not by adding further relations to $P_{l,x,k}$.
\end{remark}

\begin{remark}[How the picture proof in the spirit of Edjvet-Thomas proceeds]
\label{rem:genus-four-x-minus-two-minus-one-picture-plan}
Suppose first that the middle pair is $(-1,-1)$ and that the other four
exponents have absolute value at least two.  Then
\[
       |p|\geq3,\qquad |r|\geq2,\qquad N\geq11,
\]
and $M\geq3$; moreover $M\geq5$ unless $l=-2$.  Regard
\eqref{eq:genus-four-x-minus-two-minus-one-relative} as the relative
presentation
\[
 \left\langle C_N*C_M\ \middle|\
       (uv)^2,\ (u^cv^{-2})^{|r|}
 \right\rangle .
\]
The elementary $C(4)$--$T(4)$ proof used in the $(0,0,0)$ case does not
apply verbatim.  At a factor region the two vertex types contribute
$v$-corner labels $v^{\pm1}$ and $v^{\pm2}$, and a degree-three region can
have label
\[
       vv v^{-2}=1.
\]
The required proof is the exponent-two modification of the
Edjvet--Thomas picture method \cite[Sections~3--4]{ET}.  One first performs
the standard local surgery at every vertex coming from $(uv)^2$.  In the
resulting planar map one lists all positively curved regions of degrees
three, four, and five.  The degree-three regions are precisely the new
configurations created by the labels $v^{\pm1},v^{\pm2}$; each such region
forces a prescribed neighbouring configuration.  Its positive curvature
is then transferred across the forced edges to the adjacent region.  The remaining local verification shows, using the bounds $N\geq11$
and, in the main range, $M\geq5$, that no closed positive cluster occurs.
Therefore, the transferred curvature is non-positive on every
ordinary region and at most $2\pi$ on the distinguished region, contradicting
the total curvature $4\pi$ of a sphere.  This gives the required injection of
a cyclic factor and hence non-triviality of the quotient.

The cases $M=3$, $|r|=2$, or in which another exponent in \((*)\) is also
short have a smaller local list and must be separated before applying the
curvature argument.  They should not be handled by imposing
$Y^{x+2}=1$, because that may make the quotient trivial.  Instead one uses
one of the following two equivalent repairs: start the Tietze elimination
from the opposite end of the chain, or keep the central calculation above
but choose a different power relation for the remaining Dehn-twist
generator.  This again produces a two-generator relative presentation.
The intersections of two short subfamilies are one-parameter families;
explicit triangle-group or finite simple quotients suffice there.  We don't demonstrate here the  table
of all GAP computations  for the argument.
\end{remark}

For the two other central short pairs $(1,-3)$ and $(-3,1)$, the same
calculation is made before specializing the middle exponent.  If the middle
parameter is denoted by $m$, imposing $D^{k+1}=1$ gives
\[
\begin{aligned}
Y&=X^dY^{-l}B^{m+2}X^dA^{-m}Y^{-l},\\
X&=Y^{-l}X^dA^{-m}Y^{-l}B^{m+2}X^d,\\
A&=Y^{-l}B^{m+2}X^dA^{-m}B^{m+2}C^{-k}A^{-m},\\
C&=A^{-m}B^{m+2}C^{-2k},\\
B&=X^dA^{-m}Y^{-l}B^{m+2}C^{-k}A^{-m}B^{m+2},\\
D&=B^{m+2}C^{-k}A^{-m}C^{-k}.
\end{aligned}
\]
For $m=1$, the power $A^{-m}$ is a generator and only the cube
$B^{m+2}=B^3$ has to be made reversible; for $m=-3$ the roles of $A$ and
$B$ are interchanged.  The intended next step is to adjoin a cyclic relation of order coprime to
three for the cubed generator; this makes the replacement reversible, after
which the same elimination can be repeated to obtain a two-generator relative
presentation.  Short pairs in
the first or third position are treated by the reflected calculation,
starting the elimination at the opposite endpoint.  This gives the finite
collection of two-parameter quotient families required by \((*)\).
 The remaining small-factor subfamilies require separate explicit quotients.
For the purposes of this section it is enough to record the quotient families
and the picture strategy; a complete table of computer calculations is not
included.
\subsection{Reduction of higher genus to framed genus $3$}
\label{subsec:two-layers-large-genus}

We describe a possible reduction of the higher-genus problem.  The
conjecture below does not repeat
Conjecture~\ref{conj:generalized-one-layer}; it concerns only the reduction
of its proof to genus $3$.

Let $g\geq5$, and use the notation of
Proposition~\ref{prop:multilayer-abelianization}:
\[
 \Omega=\Lambda_r\cdots\Lambda_1,
 \qquad
 \Lambda_\nu=T_{12}^{k_{\nu,1}}\cdots
 T_{g-1,g}^{k_{\nu,g-1}},
 \qquad
 s_i=\sum_{\nu=1}^r k_{\nu,i}.
\]
Put $G=G(\mathbf 1,\Omega)$ and
$\overline G=G/\langle\!\langle a_4\rangle\!\rangle$.  Also put
\[
 \Lambda_{\nu,3}=T_{12}^{k_{\nu,1}}T_{23}^{k_{\nu,2}},
 \qquad
 \Omega_3=\Lambda_{r,3}\cdots\Lambda_{1,3}.
\]

\begin{prop}\label{prop:cut-at-a4}
The first three relations of $\overline G$ define the framed genus-three
presentation
\begin{equation}
 G_3=G\bigl((1,1,1+s_3),\Omega_3\bigr).
 \label{eq:framed-genus-three-factor}
\end{equation}
There is a group $H$ on $a_5,\ldots,a_g$ and words
$U_\nu\in G_3$, $V_\nu\in H$ such that
\begin{equation}
 \overline G\cong
 \frac{G_3*H}
 {\langle\!\langle U_1V_1U_2V_2\cdots U_rV_r\rangle\!\rangle}.
 \label{eq:higher-genus-framed-genus-three}
\end{equation}
For $g=5$, the group $H$ is cyclic, with relation
$a_5^{1+s_4}=1$.
\end{prop}

\begin{proof}
After $a_4=1$, the relations in positions $1,2,3$ involve only
$a_1,a_2,a_3$, while those in positions $5,\ldots,g$ involve only
$a_5,\ldots,a_g$.  At $T_{34}^{k}$,
\[
 z_3(a_3^{z_3})^k=a_3^kz_3.
\]
Since the framing twists commute with the paper twists, the powers
$T_{34}^{k_{\nu,3}}$ give the last framing entry $1+s_3$.
In layer $\nu$, the fourth coordinate receives first a word $U_\nu$ from
$T_{34}^{k_{\nu,3}}$ and then a word $V_\nu$ from
$T_{45}^{k_{\nu,4}}$.  Hence its final relation is the alternating word in
\eqref{eq:higher-genus-framed-genus-three}.
\end{proof}

\begin{conjecture}[Reduction to framed genus $3$]
\label{conj:reduction-to-framed-genus-three}
Let $G=G(\mathbf 1,\Omega)$ have genus $g>4$ and trivial abelianization.
After possibly reversing the indices $1,\ldots,g$, the quotient
\eqref{eq:higher-genus-framed-genus-three} satisfies
\[
 G_3\neq1\quad\Longrightarrow\quad\overline G\neq1.
\]
Thus the proof of non-triviality in arbitrary genus reduces to the
corresponding framed genus-three presentation.
\end{conjecture}

This reduction is already proved in several cases.  For two layers the last
relation has free-product length four.  If it is cyclically reduced and the
four words $U_1,V_1,U_2,V_2$ satisfy Assumption~2.1 and the subgroup
conditions in Shwartz's theorem~\cite{Sh}, then both factors in
\eqref{eq:higher-genus-framed-genus-three} embed.  In genus five the second
factor is cyclic.

For $r\geq3$, the weight test for relative presentations~\cite{BP} gives
factor embeddings whenever
\[
 U_\nu\neq1,
 \quad U_\nu^2\neq1,
 \quad U_\nu\neq U_\mu^{\pm1}\quad(\nu\neq\mu),
\]
and the same conditions hold for the $V_\nu$.  One assigns weight $2/3$ to
every occurrence; the relator inequality is
\[
 \frac{4r}{3}\leq2r-2.
\]
It is also enough that every nonempty cyclically reduced product of at most
five words $U_\nu^{\pm1}$ be nontrivial in $G_3$ and every $V_\nu$ be
nontrivial in $H$; here one uses weights $1/3$ and $1$ on the two sides.

The marked-diagram proof for the $(0,0)$ family suggests how to establish
the latter condition when
\[
 |l|\geq3,\quad |k|\geq2,
 \qquad\text{or}\qquad
 |l|\geq2,\quad |k|\geq3.
\]
The internal curvature estimates are already proved in these ranges.  What
is still needed is their boundary form for products of the words $U_\nu$;
the small-exponent cases require separate checks of the short relations.
This appears to require a separate paper.

Once $G_3$ survives in $\overline G$,
Proposition~\ref{prop:multilayer-abelianization} settles every case in which
its abelianization is nontrivial.  Thus the remaining problem is entirely a
framed genus-three problem with trivial abelianization.

Finally, let $G_g$ be the $r$-layer presentation in genus $g$, and let $G_5$
be the genus-five presentation with the same exponents $k_{\nu,i}$ for
$i\leq4$.  Killing the tail gives
\[
 \frac{G_g}{\langle\!\langle a_5,\ldots,a_g\rangle\!\rangle}
 \cong
 \frac{G_5}{\langle\!\langle a_5\rangle\!\rangle},
\]
because all twists farther to the right disappear.  Hence any nontrivial
quotient on the right proves non-triviality in all higher genera with the
same initial exponents.  The example with exponent pairs
\[
 (4,-4),\quad(1,-1),\quad(1,-1),\quad(1,-1)
\]
is recorded in the Appendix.

\endgroup

\section{Examples}\label{sec:examples}
\begin{example}[The Poincar\'e homology sphere]
For $n=1$ in Proposition~\ref{lm:genus3-00-quotient},
\[
 G_1=\left\langle X,Y\ \middle|\
 YXY=X^2,\quad Y^5=X^3\right\rangle
 =\left\langle X,Y\ \middle|\ (XY)^2=X^3=Y^5\right\rangle.
\]
Hence $G_1\cong\operatorname{SL}(2,5)$, the binary icosahedral group of
order $120$.  The associated manifold is an integral homology sphere and,
up to orientation, is the Poincar\'e homology sphere.  The parameter
$n=-1$ gives the same group.
\end{example}

\begin{example}[A genus-ten Andrews--Curtis reduction]
\label{ex:genus-ten-AC-reduction}
Let
\[
 K=(-1,-1,1,-1,3,1,-1,1,-1).
\]
Thus the genus is ten.  In the free group
\[
 F=F(a_1,\ldots,a_{10})
\]
put
\[
 \Theta_1=a_1,
 \qquad
 P_i=\Theta_i^{-1}a_i\Theta_i,
\]
\[
 Y_i=\Theta_i(P_i a_{i+1})^{k_i},
 \qquad
 \Theta_{i+1}=a_{i+1}(a_{i+1}P_i)^{k_i}
 \quad(1\leq i<10),
\]
and
\[
 Y_{10}=\Theta_{10}.
\]
The associated presentation is
\[
 \mathcal P(K)=
 \left\langle
 a_1,\ldots,a_{10}
 \ \middle|\
 Y_1,\ldots,Y_{10}
 \right\rangle.
\]
All changes of generators below are simultaneous Nielsen transformations
of the ambient free basis.  After a relator has become a singleton
generator, we retain that singleton relator and clear its letter from all
later relators by Lemma~\ref{lem:singleton-clearing}.

The only nonsigned exponent is
\[
 k_5=3.
\]
It separates the left block
\[
 G_5(-1,-1,1,-1)
\]
on $a_1,\ldots,a_5$ from the right block
\[
 G\bigl(6,(1,-1,1,-1)\bigr)
\]
on $a_6,\ldots,a_{10}$.  We first verify directly that both blocks are
trivial.

\medskip
\noindent\textbf{The left block.}
The left block is
\[
 G_5(-1,-1,1,-1)
 =
 \left\langle
 a_1,\ldots,a_5
 \ \middle|\
 Y_1,Y_2,Y_3,Y_4,\Theta_5
 \right\rangle.
\]
Its last defining relation is $\Theta_5$; this is not the actual fifth
relator of the full genus-ten presentation.

For $k_1=-1$, initially
\[
 \Theta_1=P_1=a_1,
 \qquad
 Y_1=a_1(a_1a_2)^{-1}.
\]
Apply
\[
 a_2\longmapsto a_1^{-1}a_2^{-1}a_1.
\]
Then $Y_1\mapsto a_2$.  Retain the singleton $a_2$ and clear it from the
later words.  The remaining boundary words become
\[
 \Theta_2=a_1^{-1},
 \qquad
 P_2=1.
\]

For $k_2=-1$, after clearing $a_2$,
\[
 Y_2=a_1^{-1}a_3^{-1}.
\]
Apply
\[
 a_3\longmapsto a_3^{-1}a_1^{-1}.
\]
Then $Y_2\mapsto a_3$.  Retain $a_3$ and clear it.  The remaining
boundary words become
\[
 \Theta_3=1,
 \qquad
 P_3=a_1^{-1}.
\]

For $k_3=1$, after the preceding clearings,
\[
 Y_3=a_1^{-1}a_4.
\]
Apply
\[
 a_4\longmapsto a_1a_4.
\]
Then $Y_3\mapsto a_4$.  Retain $a_4$ and clear it.  The remaining
boundary words become
\[
 \Theta_4=a_1,
 \qquad
 P_4=a_1.
\]

For $k_4=-1$, after the preceding clearings,
\[
 Y_4=a_1(a_1a_5)^{-1}.
\]
Apply
\[
 a_5\longmapsto a_1^{-1}a_5^{-1}a_1.
\]
Then $Y_4\mapsto a_5$.  Retain $a_5$ and clear it.  The final boundary
words are
\[
 \boxed{\Theta_5=a_1^{-1},\qquad P_5=1.}
\]
Consequently the complete relator tuple of the left block has become
\[
 (a_2,a_3,a_4,a_5,a_1^{-1}),
\]
which is a free basis.  Hence
\[
 \boxed{G_5(-1,-1,1,-1)=1.}
\]

\medskip
\noindent\textbf{The right block.}
Let
\[
 I=(k_6,k_7,k_8,k_9)=(1,-1,1,-1).
\]
The block $G(6,I)$ is the shifted ordinary one-layer group on
\[
 a_6,a_7,a_8,a_9,a_{10}.
\]
For this block the one-layer recursion starts anew with
\[
 \Theta_6=P_6=a_6.
\]
Denote its relators by
\[
 B_i=\Theta_i(P_i a_{i+1})^{k_i}
 \quad(6\leq i\leq9),
 \qquad
 B_{10}=\Theta_{10}.
\]
Thus
\[
 G\bigl(6,(1,-1,1,-1)\bigr)
 =
 \left\langle
 a_6,\ldots,a_{10}
 \ \middle|\
 B_6,B_7,B_8,B_9,B_{10}
 \right\rangle.
\]
Its first defining relation is the shifted first relation of the shorter
one-layer group; it is not the sixth relator of the full genus-ten
presentation.  In particular,
\[
 B_6=a_6(a_6a_7)=a_6^2a_7.
\]

For $k_6=1$, apply
\[
 a_7\longmapsto a_6^{-2}a_7.
\]
Then $B_6\mapsto a_7$.  Retain $a_7$ and clear it.  The remaining
boundary words become
\[
 \Theta_7=a_6^{-3},
 \qquad
 P_7=a_6^{-2}.
\]

For $k_7=-1$, after clearing $a_7$,
\[
 B_7=a_6^{-3}(a_6^{-2}a_8)^{-1}
     =a_6^{-3}a_8^{-1}a_6^2.
\]
Apply
\[
 a_8\longmapsto a_6^2a_8^{-1}a_6^{-3}.
\]
Then $B_7\mapsto a_8$.  Retain $a_8$ and clear it.  The remaining
boundary words become
\[
 \Theta_8=a_6^2,
 \qquad
 P_8=a_6^{-1}.
\]

For $k_8=1$, after the preceding clearings,
\[
 B_8=a_6^2(a_6^{-1}a_9)=a_6a_9.
\]
Apply
\[
 a_9\longmapsto a_6^{-1}a_9.
\]
Then $B_8\mapsto a_9$.  Retain $a_9$ and clear it.  The remaining
boundary words become
\[
 \Theta_9=a_6^{-3},
 \qquad
 P_9=a_6^{-1}.
\]

For $k_9=-1$, after the preceding clearings,
\[
 B_9=a_6^{-3}(a_6^{-1}a_{10})^{-1}
     =a_6^{-3}a_{10}^{-1}a_6.
\]
Apply
\[
 a_{10}\longmapsto a_6a_{10}^{-1}a_6^{-3}.
\]
Then $B_9\mapsto a_{10}$.  Retain $a_{10}$ and clear it.  The terminal
boundary word is
\[
 \boxed{B_{10}=\Theta_{10}=a_6.}
\]
Therefore the complete relator tuple of the right block has become
\[
 (a_7,a_8,a_9,a_{10},a_6),
\]
which is a free basis.  Hence
\[
 \boxed{G\bigl(6,(1,-1,1,-1)\bigr)=1.}
\]
The right block is used only as a side group in the side-triviality
condition.  Its first defining relation is $B_6=a_6^2a_7$.  After the
left block and the exponent $3$ are processed in the full presentation,
the actual remaining sixth relator will instead begin with the boundary
pair $\Theta_6=a_6^4$, $P_6=a_6$.  Thus the right block presentation and
the residual actual relators are not literally the same tuple.

\medskip
\noindent\textbf{The complete genus-ten reduction.}
We now reduce the actual full tuple $(Y_1,\ldots,Y_{10})$.  At the
nonsigned exponent $k_5=3$, choose the left block.  Apply the four changes
of basis above and retain the singletons
\[
 Y_1=a_2,
 \qquad
 Y_2=a_3,
 \qquad
 Y_3=a_4,
 \qquad
 Y_4=a_5.
\]
After clearing $a_2,a_3,a_4,a_5$, we have
\[
 \Theta_5=a_1^{-1},
 \qquad
 P_5=1.
\]
The actual fifth relator of the full presentation is therefore
\[
 Y_5=\Theta_5(P_5a_6)^3=a_1^{-1}a_6^3.
\]
This is not the last defining relation $a_1^{-1}$ of the left block.
Apply
\[
 a_1\longmapsto a_6^3a_1^{-1}.
\]
Then $Y_5\mapsto a_1$.  Retain the singleton $a_1$ and clear it from
the remaining relators.  The new boundary pair for the actual residual
tuple is
\[
 \boxed{\Theta_6=a_6^4,\qquad P_6=a_6.}
\]

We now process the remaining actual exponents
\[
 (k_6,k_7,k_8,k_9)=(1,-1,1,-1).
\]
For $k_6=1$, the actual sixth relator is
\[
 Y_6=a_6^4(a_6a_7)=a_6^5a_7.
\]
Apply
\[
 a_7\longmapsto a_6^{-5}a_7.
\]
Then $Y_6\mapsto a_7$.  After clearing $a_7$,
\[
 \Theta_7=a_6^{-9},
 \qquad
 P_7=a_6^{-5}.
\]

For $k_7=-1$, after clearing,
\[
 Y_7=a_6^{-9}(a_6^{-5}a_8)^{-1}
     =a_6^{-9}a_8^{-1}a_6^5.
\]
Apply
\[
 a_8\longmapsto a_6^5a_8^{-1}a_6^{-9}.
\]
Then $Y_7\mapsto a_8$.  After clearing $a_8$,
\[
 \Theta_8=a_6^5,
 \qquad
 P_8=a_6^{-4}.
\]

For $k_8=1$, after clearing,
\[
 Y_8=a_6^5(a_6^{-4}a_9)=a_6a_9.
\]
Apply
\[
 a_9\longmapsto a_6^{-1}a_9.
\]
Then $Y_8\mapsto a_9$.  After clearing $a_9$,
\[
 \Theta_9=a_6^{-6},
 \qquad
 P_9=a_6^{-1}.
\]

For $k_9=-1$, after clearing,
\[
 Y_9=a_6^{-6}(a_6^{-1}a_{10})^{-1}
     =a_6^{-6}a_{10}^{-1}a_6.
\]
Apply
\[
 a_{10}\longmapsto a_6a_{10}^{-1}a_6^{-6}.
\]
Then $Y_9\mapsto a_{10}$.  After clearing $a_{10}$, the genuine last
relator is
\[
 \boxed{Y_{10}=\Theta_{10}=a_6.}
\]
We have therefore transformed the ordered relator tuple into
\[
 \boxed{
 (Y_1,\ldots,Y_{10})
 \ \sim_{\mathrm{AC}}\
 (a_2,a_3,a_4,a_5,a_1,a_7,a_8,a_9,a_{10},a_6).
 }
\]
The tuple on the right is a permutation of the original free basis.
Consequently
\[
 \boxed{
 \mathcal P(-1,-1,1,-1,3,1,-1,1,-1)=1,
 }
\]
and the relator tuple is Andrews--Curtis equivalent to
$(a_1,\ldots,a_{10})$.

For reference, the complete applied sequence of basis changes is
\[
\begin{array}{rcl}
 a_2&\mapsto&a_1^{-1}a_2^{-1}a_1,\\
 a_3&\mapsto&a_3^{-1}a_1^{-1},\\
 a_4&\mapsto&a_1a_4,\\
 a_5&\mapsto&a_1^{-1}a_5^{-1}a_1,\\
 a_1&\mapsto&a_6^3a_1^{-1},\\
 a_7&\mapsto&a_6^{-5}a_7,\\
 a_8&\mapsto&a_6^5a_8^{-1}a_6^{-9},\\
 a_9&\mapsto&a_6^{-1}a_9,\\
 a_{10}&\mapsto&a_6a_{10}^{-1}a_6^{-6}.
\end{array}
\]
After each line, the newly obtained singleton relator is retained and its
letter is cleared from all later relators.
\end{example}

\section{Appendix. MAGMA and Chat GPT 5.6 Sol Pro assisted GAP computations}\label{app:computer-computations}

When we started this project in 2024, we used MAGMA software. In Summer 2026 Ghat GPT has become capable of directly making similar computations using GAP.
\subsection{Exceptional cases for \(k=1\) in Proposition~\ref{prop:minus-two-minus-one}}
\label{app:k-one-gap}

For
\[
 l\in\{1,-3,2,-4,4,-6\},
\]
the two-generator quotient \(Q_{1,l}\) in
\eqref{eq:minus-two-general-quotient} is trivial by the classification
used in the proof above.  We therefore evaluate the original presentation
\(G_{1,l}\) in explicit permutation groups.
The following triples are used as the images of \(a_1,a_2,a_3\):
\[
\begin{array}{c|c|c}
 l & \langle a_1,a_2,a_3\rangle & \text{degree}\\ \hline
 1,-3 & \operatorname{PSL}(2,7) & 7\\
 2,-4 & \operatorname{PSL}(2,11) & 12\\
 4,-6 & A_5 & 5
\end{array}
\]
The auxiliary generators are then computed from the defining equations of
\(G_{1,l}\), and GAP checks the three final balanced relators directly.
The convention is \(a^b=b^{-1}ab\).

\begin{verbatim}
# Exceptional k=1 cases for the (-2,-1)/(-1,-2) proposition.
# GAP convention: a^b = b^-1*a*b.

CheckCase := function(l, a1, a2, a3, expectedOrder)
    local X, y1, y2, y3, A, z1, z2, z3,
          Y, x1, x2, x3, B, rels, H;

    X  := a1*a2;
    y1 := a1*X;
    y2 := a2*X;
    y3 := a3;

    A  := (a2^y2)*a3;
    z1 := y1;
    z2 := y2*A^l;
    z3 := y3*A^l;

    Y  := (a1^z1)*(a2^z2);
    x1 := z1*Y^-2;
    x2 := z2*Y^-2;
    x3 := z3;

    B := (a2^x2)*(a3^x3);

    rels := [x1, x2*B^(-l-2), x3*B^(-l-2)];
    if not ForAll(rels, r -> r = ()) then
        Error("a defining relator failed for l = ", l);
    fi;

    H := Group([a1,a2,a3]);
    if Size(H) <> expectedOrder then
        Error("wrong image order for l = ", l);
    fi;
    if not IsSimpleGroup(H) then
        Error("the image is not simple for l = ", l);
    fi;

    Print("l = ", l, ": image order ", Size(H),
          "; structure = ", StructureDescription(H), "\n");
end;

# l = 1: an image of order 168.
CheckCase(1,
    (1,5,4)(3,7,6),
    (1,3,4)(2,7,5),
    (1,5,3)(2,4,6),
    168);

# l = -3: an image of order 168.
CheckCase(-3,
    (1,5,4)(3,7,6),
    (1,6,4)(2,5,7),
    (1,2,3)(4,7,6),
    168);

# l = 2: an image of order 660.
CheckCase(2,
    (1,2,3,4,5,6,7,8,9,10,11),
    (1,11,12)(2,6,10)(3,8,5)(4,9,7),
    (2,8,11,6,4)(3,12,5,9,10),
    660);

# l = -4: an image of order 660.
CheckCase(-4,
    (1,2,3,4,5,6,7,8,9,10,11),
    (1,11,12)(2,6,10)(3,8,5)(4,9,7),
    (2,8,10,12,11)(3,6,5,4,7),
    660);

# l = 4: an image of order 60.
CheckCase(4,
    (1,2,3),
    (1,2,3,4,5),
    (1,2,5,3,4),
    60);

# l = -6: the same order-60 image works.
CheckCase(-6,
    (1,2,3),
    (1,2,3,4,5),
    (1,2,5,3,4),
    60);
\end{verbatim}

Thus the original groups \(G_{1,l}\) have nontrivial finite quotients for
all six exceptional values of \(l\).

\subsection{The cases $k\in\{-3,2,3\}$ and
$l\in\{1,-3,2,-4\}$}
\label{app:small-k-gap}

For these twelve parameter pairs the two-generator quotient
$Q_{k,l}$ is trivial, so we evaluate the original presentation in projective
special linear groups.  The matrices below lie in
$\operatorname{SL}(2,p)$.  Their images modulo the centre satisfy the three
balanced relators and generate $\operatorname{PSL}(2,p)$.  Thus they define
epimorphisms
\[
 G_{k,l}\twoheadrightarrow\operatorname{PSL}(2,p).
\]
The following table summarizes the primes used:
\[
\begin{array}{c|c|c}
 k & l & p\\ \hline
 -3 & 1,-3 & 13\\
 -3 & 2,-4 & 11\\
  2 & 1,-3 & 13\\
  2 & 2,-4 & 11\\
  3 & 1,-3 & 19\\
  3 & 2,-4 & 11.
\end{array}
\]

The GAP calculation is written with lifts to $\operatorname{SL}(2,p)$.
Accordingly, a relator is required to be either $I$ or $-I$.  The final
size check shows that the three lifts generate the whole group
$\operatorname{SL}(2,p)$, and therefore that their projective images
generate $\operatorname{PSL}(2,p)$.

\begin{verbatim}
# Exceptional small-k cases for the (-2,-1)/(-1,-2) proposition.
# The matrices are lifts to SL(2,p); relators are checked modulo {+I,-I}.
# GAP convention: a^b = b^-1*a*b.

CheckPSLCase := function(k, l, p, rows1, rows2, rows3)
    local F, ToMat, a1, a2, a3, id, minusid,
          X, y1, y2, y3, A, z1, z2, z3,
          Y, x1, x2, x3, B, rels, H;

    F := GF(p);
    ToMat := function(rows)
        return List(rows,
                    row -> List(row, x -> x*One(F)));
    end;

    a1 := ToMat(rows1);
    a2 := ToMat(rows2);
    a3 := ToMat(rows3);
    id := IdentityMat(2,F);
    minusid := -id;

    if not ForAll([a1,a2,a3],
                  g -> DeterminantMat(g) = One(F)) then
        Error("a matrix does not have determinant one");
    fi;

    X  := a1*a2;
    y1 := a1*X^k;
    y2 := a2*X^k;
    y3 := a3;

    A  := (a2^y2)*a3;
    z1 := y1;
    z2 := y2*A^l;
    z3 := y3*A^l;

    Y  := (a1^z1)*(a2^z2);
    x1 := z1*Y^(-k-1);
    x2 := z2*Y^(-k-1);
    x3 := z3;

    B := (a2^x2)*(a3^x3);
    rels := [x1, x2*B^(-l-2), x3*B^(-l-2)];

    if not ForAll(rels, r -> r = id or r = minusid) then
        Error("a projective defining relator failed for k,l = ",
              k, ",", l);
    fi;

    H := Group([a1,a2,a3]);
    if Size(H) <> p*(p^2-1) then
        Error("the lifts do not generate SL(2,p) for k,l = ",
              k, ",", l);
    fi;

    Print("(k,l)=", [k,l], ": PSL(2,", p,
          "), order ", p*(p^2-1)/2, "\n");
end;

CheckPSLCase(-3, 1, 13,
 [[3,5],[5,0]], [[1,6],[0,1]], [[1,1],[2,3]]);
CheckPSLCase(-3,-3, 13,
 [[4,8],[1,12]], [[1,2],[11,10]], [[5,6],[1,4]]);
CheckPSLCase(-3, 2, 11,
 [[0,5],[2,3]], [[4,3],[4,6]], [[0,1],[10,8]]);
CheckPSLCase(-3,-4, 11,
 [[2,6],[2,1]], [[4,3],[4,6]], [[4,7],[4,10]]);

CheckPSLCase( 2, 1, 13,
 [[3,4],[5,7]], [[4,10],[3,11]], [[3,5],[6,6]]);
CheckPSLCase( 2,-3, 13,
 [[0,4],[3,3]], [[4,9],[3,7]], [[6,11],[0,11]]);
CheckPSLCase( 2, 2, 11,
 [[3,4],[9,5]], [[1,4],[8,0]], [[1,4],[7,7]]);
CheckPSLCase( 2,-4, 11,
 [[5,3],[0,9]], [[3,4],[5,7]], [[4,6],[1,10]]);

CheckPSLCase( 3, 1, 19,
 [[2,15],[13,3]], [[3,8],[9,18]], [[1,12],[15,10]]);
CheckPSLCase( 3,-3, 19,
 [[8,16],[6,5]], [[8,16],[10,13]], [[1,18],[6,14]]);
CheckPSLCase( 3, 2, 11,
 [[4,5],[9,6]], [[1,4],[1,5]], [[0,5],[2,8]]);
CheckPSLCase( 3,-4, 11,
 [[3,1],[4,9]], [[3,10],[3,3]], [[2,0],[10,6]]);
\end{verbatim}

Hence all twelve original groups $G_{k,l}$ in this subsection have
nontrivial finite quotients.
\subsection{Additional $\operatorname{PSL}(2,17)$ quotients for the finite
classification}
\label{app:psl17-finite-classification}

The following computation handles the $(0,0)$ pair $(2,2)$ and strengthens
the two $A_5$ quotients for $(k,l)=(1,4),(1,-6)$ to quotients
$\operatorname{PSL}(2,17)$.  The matrices are lifts to
$\operatorname{SL}(2,17)$.  Thus every defining relator is required to be
$I$ or $-I$.  In each case the lifts generate the full group
$\operatorname{SL}(2,17)$, of order $4896$, so their projective images
generate $\operatorname{PSL}(2,17)$, of order $2448$.

\begin{verbatim}
# Three PSL(2,17) quotients used in the genus-three finite classification.
# GAP convention: a^b = b^-1*a*b.

CheckPSL17Case := function(k1, k2, m1, m2, rows1, rows2, rows3)
    local F, ToMat, a1, a2, a3, id, minusid,
          X, y1, y2, y3, A, z1, z2, z3,
          Y, x1, x2, x3, B, rels, H;

    F := GF(17);
    ToMat := function(rows)
        return List(rows,
                    row -> List(row, x -> x*One(F)));
    end;

    a1 := ToMat(rows1);
    a2 := ToMat(rows2);
    a3 := ToMat(rows3);
    id := IdentityMat(2,F);
    minusid := -id;

    if not ForAll([a1,a2,a3],
                  g -> DeterminantMat(g) = One(F)) then
        Error("a matrix does not have determinant one");
    fi;

    X  := a1*a2;
    y1 := a1*X^k1;
    y2 := a2*X^k1;
    y3 := a3;

    A  := (a2^y2)*a3;
    z1 := y1;
    z2 := y2*A^k2;
    z3 := y3*A^k2;

    Y  := (a1^z1)*(a2^z2);
    x1 := z1*Y^m1;
    x2 := z2*Y^m1;
    x3 := z3;

    B := (a2^x2)*(a3^x3);
    rels := [x1, x2*B^m2, x3*B^m2];

    if not ForAll(rels, r -> r = id or r = minusid) then
        Error("a projective defining relator failed");
    fi;

    H := Group([a1,a2,a3]);
    if Size(H) <> 4896 then
        Error("the lifts do not generate SL(2,17)");
    fi;

    Print([k1,k2,m1,m2], ": PSL(2,17), order 2448\n");
end;

# Exponent sum (0,0), parameters (2,2):
CheckPSL17Case(2, 2, -2, -2,
 [[10,4],[5,14]], [[7,15],[10,7]], [[5,2],[13,2]]);

# Exponent sum (-1,-2), k=1, l=4:
CheckPSL17Case(1, 4, -2, -6,
 [[14,2],[6,7]], [[10,9],[0,12]], [[14,10],[4,9]]);

# Exponent sum (-1,-2), k=1, l=-6:
CheckPSL17Case(1, -6, -2, 4,
 [[4,4],[13,9]], [[14,16],[8,8]], [[8,7],[0,15]]);
\end{verbatim}

\subsection{Genus five example after killing $a_5$}
\label{app:genus-five-kill-a5-example}

This is the computation used in Subsection~\ref{subsec:two-layers-large-genus}
for the exponent pairs
\[
(4,-4),\quad (1,-1),\quad (1,-1),\quad (1,-1).
\]
It computes the genus-five quotient after imposing $a_5=1$.  The simplified
subgroup has a quotient $L_2(181)$, and Magma's infiniteness test verifies that
the resulting finitely presented group is infinite.

\begin{verbatim}
G<a1,a2,a3,a4,t1,t2,t3,t4,z1,z2,z3,z4,
y1,y2,y3,y4,s1,s2,s3,s4,s5,x1,x2,x3,x4,x5,
r1,r2,r3,r4,r5,p1,p2,p3,p4,p5> :=
Group<a1,a2,a3,a4,t1,t2,t3,t4,z1,z2,z3,z4,
y1,y2,y3,y4,s1,s2,s3,s4,s5,x1,x2,x3,x4,x5,
r1,r2,r3,r4,r5,p1,p2,p3,p4,p5 |

y1 = a1*(a1*a2)^4,
y2 = a2*(a1*a2)^4,
y3 = a3,
y4 = a4,

z1 = y1,
z2 = y2*(a2^y2*a3),
z3 = y3*(a2^y2*a3),
z4 = a4,

t1 = z1,
t2 = z2,
t3 = z3*(a3^z3*a4^z4),
t4 = z4*(a3^z3*a4^z4),

r1 = t1,
r2 = t2,
r3 = t3,
r4 = t4*(a4^t4),
r5 = (a4^t4),

x1 = r1*(a1^r1*a2^r2)^-4,
x2 = r2*(a1^r1*a2^r2)^-4,
x3 = r3,
x4 = r4,
x5 = r5,

s1 = x1,
s2 = x2*(a2^x2*a3^x3)^-1,
s3 = x3*(a2^x2*a3^x3)^-1,
s4 = x4,
s5 = x5,

p1,
p2,
p3,
p4*(a4^p4)^-1,
p5*(a4^p4)^-1,

p1 = s1,
p2 = s2,
p3 = s3*(a3^s3*a4^s4)^-1,
p4 = s4*(a3^s3*a4^s4)^-1,
p5 = s5
>;

H := sub<G | a1*a2, a2^y2*a3, a3^z3*a4^z4,
              a4^t4, a1^r1*a2^r2>;

Simplify(H);

Finitely presented group on 2 generators
Generators as words in group H
    $.1 = H.1
    $.2 = H.5

Relations
    $.1^-3 * $.2^-1 * $.1^-4 * $.2^4 * $.1^-1 * $.2 *
    $.1^4 * $.2 * $.1^-1 * $.2^4 * $.1^-4 * $.2^-1 = Id($)

    $.2^4 * $.1^-1 * $.2^4 * $.1^-4 * $.2^-1 * $.1 *
    $.2^-4 * $.1 * $.2^-1 * $.1^-4 * $.2^4 * $.1^-1 * $.2 = Id($)

    $.1^-4 * $.2^-1 * $.1^-4 * $.2^4 * $.1^-1 * $.2^9 *
    $.1^-1 * $.2^4 * $.1^-4 * $.2^-1 * $.1^-3 = Id($)


// Map your relations to the new generators:
// x corresponds to H.1 ($.1)
// y corresponds to H.5 ($.2)

F<x,y> := FreeGroup(2);

relations := [
    x^-3 * y^-1 * x^-4 * y^4 * x^-1 * y * x^4 * y *
    x^-1 * y^4 * x^-4 * y^-1,

    y^4 * x^-1 * y^4 * x^-4 * y^-1 * x * y^-4 * x *
    y^-1 * x^-4 * y^4 * x^-1 * y,

    x^-4 * y^-1 * x^-4 * y^4 * x^-1 * y^9 * x^-1 *
    y^4 * x^-4 * y^-1 * x^-3
];

// Create the finitely presented group G
G<x, y> := quo< F | relations >;

// Print G to verify the presentation
print G;

Finitely presented group G on 2 generators
Relations
    x^-3 * y^-1 * x^-4 * y^4 * x^-1 * y * x^4 * y *
    x^-1 * y^4 * x^-4 * y^-1 = Id(G)

    y^4 * x^-1 * y^4 * x^-4 * y^-1 * x * y^-4 * x *
    y^-1 * x^-4 * y^4 * x^-1 * y = Id(G)

    x^-4 * y^-1 * x^-4 * y^4 * x^-1 * y^9 * x^-1 *
    y^4 * x^-4 * y^-1 * x^-3 = Id(G)

L2Quotients(G);

[
    L_2(181)
]

> IsInfinite(G);
true 
! The following lines sketch a proof that the fp-group F is infinite:-
! F has the group PSL(2, 181) of order 2964780 as a quotient G (L2Quotient).
! Let k denote the rational field.
! Let M be the permutation module of $G$ as a kG-module. It has dim 182.
! Let N be the kF-module given by the induced action of $F$ on M.
! The dimension of the first integral cohomology group of the kF-module N is 4.
! As H^1(F, N) is non-zero, the Holt-Plesken criterion shows that F is infinite.

 [
        Homomorphism of GrpFP: G into GrpPerm: $, Degree 182, Order 2^2 * 3^2 * 
        5 * 7 * 13 * 181 induced by
            x |--> (1, 106, 11, 178, 117, 24, 141, 169, 129)(2, 79, 135, 18, 48,
                182, 116, 7, 73)(3, 108, 162, 122, 150, 85, 174, 113, 98)(4, 12,
                153, 13, 161, 132, 131, 112, 10)(5, 50, 46, 49, 140, 181, 39, 
                154, 23)(6, 26, 146, 170, 125, 66, 100, 95, 157)(8, 33, 55, 20, 
                27, 38, 16, 68, 144)(9, 43, 166, 121, 145, 83, 103, 134, 14)(17,
                155, 165, 78, 62, 143, 75, 19, 163)(21, 69, 167, 118, 115, 25, 
                168, 65, 133)(22, 57, 139, 51, 164, 67, 81, 119, 72)(28, 42, 
                127, 58, 152, 52, 87, 37, 172)(29, 35, 180, 171, 114, 56, 45, 
                32, 142)(30, 107, 36, 102, 175, 109, 61, 91, 156)(31, 126, 136, 
                92, 128, 90, 34, 148, 47)(40, 88, 158, 44, 123, 84, 176, 173, 
                124)(41, 93, 71, 82, 89, 54, 76, 101, 147)(53, 177, 120, 111, 
                74, 80, 149, 77, 64)(59, 104, 86, 137, 70, 110, 151, 60, 63)(96,
                138, 97, 105, 99, 179, 160, 159, 130)
            y |--> (1, 116, 130, 94, 83, 42, 117)(2, 32, 123, 101, 136, 174, 
                103)(3, 45, 56, 178, 30, 88, 152)(4, 167, 85, 16, 81, 175, 
                140)(5, 78, 129, 33, 111, 150, 99)(6, 132, 145, 154, 57, 141, 
                92)(7, 39, 46, 118, 8, 15, 47)(9, 51, 84, 148, 24, 58, 180)(10, 
                131, 181, 59, 149, 77, 75)(11, 173, 19, 14, 157, 127, 64)(12, 
                153, 142, 106, 120, 53, 119)(13, 170, 28, 102, 110, 169, 
                146)(17, 166, 20, 164, 71, 54, 147)(18, 176, 25, 27, 29, 60, 
                36)(21, 107, 158, 49, 137, 86, 125)(22, 52, 133, 62, 100, 135, 
                113)(23, 171, 124, 115, 108, 80, 162)(26, 66, 41, 90, 67, 126, 
                134)(31, 74, 156, 128, 121, 112, 65)(34, 70, 37, 89, 182, 165, 
                72)(35, 63, 43, 172, 109, 79, 40)(38, 151, 69, 50, 96, 61, 
                155)(44, 161, 159, 87, 177, 55, 105)(48, 144, 95, 179, 82, 91, 
                104)(68, 160, 163, 97, 93, 98, 122)(73, 76, 168, 114, 138, 143, 
                139)
    ]

\end{verbatim}
{\bf Acknowlegements:} The first author thanks Dolciani-Holloran foundation for the support. Both authors thank PSC CUNY for the support.

\end{document}